\documentclass{amsart}

\usepackage{amssymb}
\usepackage{mathabx}
\usepackage{url}
\usepackage{enumitem}
\usepackage{epsfig}  		
\usepackage{subcaption}

\usepackage{tikz-cd}
\usepackage{multirow}
\usepackage[new]{old-arrows}

\usepackage[style=ext-numeric,
sorting=nyt,
url=false,
doi=false,
isbn=false,
articlein=false,
date=year,
giveninits=true
]{biblatex}

\renewbibmacro*{volume+number+eid}{%
  \printfield{volume}%
  \setunit*{\adddot}
  \setunit*{\addcomma\space}
  \printfield{number}%
  \setunit{\addcomma\space}%
  \printfield{eid}
 }
 
\DeclareFieldFormat{pages}{#1}
\DeclareFieldFormat*{title}{\mkbibitalic{#1}}
\DeclareFieldFormat*{citetitle}{\mkbibitalic{#1}}
\DeclareFieldFormat{journaltitle}{#1\isdot}

\newtheorem{thm}{Theorem}[section]
\newtheorem{prop}[thm]{Proposition}
\newtheorem{lem}[thm]{Lemma}
\newtheorem{cor}[thm]{Corollary}
\newtheorem{conj}[thm]{Conjecture} 
\theoremstyle{definition}

\newtheorem{example}[thm]{Example}
\theoremstyle{remark}
\newtheorem{remark}[thm]{Remark}

\newtheoremstyle{TheoremNum}
    {\topsep}{\topsep}              
    {\itshape}                      
    {}                              
    {\bfseries}                     
    {.}                             
    { }                             
    {\thmname{#1}\thmnote{ \bfseries #3}}
\theoremstyle{TheoremNum}

\newenvironment{claim}[1]
    {\noindent\textit{Claim#1:} \itshape}
    {}

\newcommand{\Q}{\mathbb{Q}}  
\newcommand{\Z}{\mathbb{Z}}  

\newcommand{\la}{\langle}  
\newcommand{\ra}{\rangle}  
\newcommand{\lf}{\lfloor}  
\newcommand{\rf}{\rfloor}  

\newcommand{\gr}{\mathsf{gr}}  
\newcommand{\cl}{\mathsf{cl}}  
\newcommand{\modop}{\,\mathsf{mod}\,}  
\newcommand{\oGamma}{\overline{\Gamma}}  
\newcommand{\Gammatop}{\Gamma_{\mathsf{top}}}

\newcommand\bigzero{\makebox(0,0){\text{\huge0}}}

\begin{document}

\title{Residual Torsion-Free Nilpotence, Bi-Orderability and Two-Bridge Links}

\author{Jonathan Johnson}
\address{Department of Mathematics, University of Texas at Austin, 
Austin, TX}
\email{jonjohnson@utexas.edu}

\begin{abstract}
Residual torsion-free nilpotence has proven to be an important property for knot groups with applications to bi-orderability \cite{LRR08} and ribbon concordance \cite{Gor81}.
Mayland \cite{May74} proposed a strategy to show that a two-bridge knot group has a commutator subgroup which is a union of an ascending chain of parafree groups.
This paper proves Mayland's assertion and expands the result to the subgroups of two-bridge link groups that correspond to the kernels of maps to $\mathbb{Z}$. We call these kernels the Alexander subgroups of the links.
As a result, we show the bi-orderability of a large family of two-bridge link groups.
This proof makes use of a modified version of a graph theoretic construction of Hirasawa and Murasugi \cite{HirMur07} in order to understand the structure of the Alexander subgroup for a two-bridge link group.
\end{abstract}

\maketitle

\section{Introduction}

Given an oriented smooth link $L$ in $S^3$, the \emph{link group} of $L$, denoted $\pi(L)$, is the fundamental group of the complement of $L$ in $S^3$.
Also, let $\Delta_L(t)$ denote the Alexander polynomial of $L$; see \cite[Chapter 6]{Mur96} for details.

Let $h:\pi(L)\to H_1(S^3-L)$ be the Hurewicz map,
and let $\varphi:H_1(S^3-L) \to \Z$ be the map defined by identifying the oriented meridians of each component of $L$ with each other.
The group $\pi(L)$ is canonically an extension of $\Z$ by $\ker(\varphi\circ h)$ as follows.
\begin{equation} \label{extendingmap}
  \begin{tikzcd}
    1 \arrow{r} 
    & \ker(\varphi\circ h) \arrow{r}
    & \pi(L) \arrow[swap]{d}{h} \arrow{r}{\varphi\circ h}
    & \Z \arrow{r}
	& 1\\
	& & H_1(S^3-L) \arrow[swap]{ur}{\varphi}
  \end{tikzcd}
\end{equation}
We call the subgroup $\ker(\varphi\circ h)$ the \emph{Alexander subgroup} of the oriented link $L$.
When $L$ is a knot, the Alexander subgroup is the commutator subgroup of $\pi(L)$.

A group $G$ is \emph{residually torsion-free nilpotent} if for every nontrivial element $x\in G$, there is a normal subgroup $N\lhd G$ such that $x\notin N$ and $G/N$ is a torsion-free nilpotent group.
The residual torsion-free nilpotence of the Alexander subgroup of a link groups has applications to bi-orderability \cite{LRR08} and ribbon concordance \cite{Gor81}.
Several knots are known to have groups with residually torsion-free nilpotent commutator subgroups including fibered knots
(since free groups are residually torsion-free nilpotent \cite{Mal49} and the commutator subgroup of a fibered knot group is a finitely generated free group),
twist knots \cite{May72}, all knots in Reidemeister's knot table (see \cite{Reid32}) except $8_{13}$, $9_{25}$, $9_{35}$, $9_{38}$, $9_{41}$, and $9_{49}$ \cite{May72}, and pseudo-alternating links whose Alexander polynomials have prime power leading coefficients \cite{MayMur76}.
This paper confirms that many two-bridge links, including all two-bridge knots, have groups with residually-torsion free nilpotent Alexander subgroups.

\begin{thm} \label{mainthm}
If $L$ is an oriented two-bridge link with an Alexander polynomial with relatively prime coefficients (collectively, not pairwise),
then the Alexander subgroup of $\pi(L)$ is residually torsion-free nilpotent.
\end{thm}

\begin{remark} \label{remrtfn}
The condition on the coeficients of the Alexander polynomial cannot be removed.
For example, if $L$ is the $(4,2)$-torus link, shown in Figure \ref{torus42}, then $L$ has Alexander subgroup isomorphic to
\[
    \la \{S_i\}_{i\in\Z} \; | \; S_i^2=S_{i+1}^2,i\in\Z \ra
\]
which is not residually nilpotent.
(For details on computing the Alexander subgroup, see section \ref{GroupPres}.)  The Alexander polynomial of the $L$ is $\Delta_L(t)=2t-2$.
\end{remark}

\begin{figure}[t]
\includegraphics{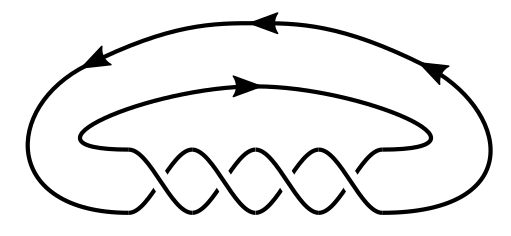}
\caption{The $(4,2)$-torus link.}
\label{torus42}
\end{figure}

It's a well known fact that $\Delta_K(1)=\pm 1$ for every knot $K$.
It follows that the coefficients of the Alexander polynomial of $K$ are relatively prime so we have the following corollary. 

\begin{cor} \label{maincor}
The commutator subgroup of a two-bridge knot group is residually torsion-free nilpotent.
\end{cor}

The following conjecture is an analog of a question by Mayland in \cite{May72}.

\begin{conj}
    The link groups of alternating knots have residually torsion-free nilpotent Alexander subgroups when the link's Alexander polynomial has relatively prime coefficients.
\end{conj}

\subsection{Summary of the Technique Used}

The proof of Theorem \ref{mainthm} relies on Baumslag's work on parafree groups \cite{Baum67,Baum69}.
Let $G$ be a group.
Define $\gamma_1G:=G$, and
for each positive integer $n$,
define $\gamma_{n+1}G:=[G,\gamma_nG]$.
A group $G$ is \emph{parafree of rank $r$} if
\begin{enumerate}
\item for some free group $F$ of rank $r$, $G/\gamma_nG\cong F/\gamma_n F$ for each $n$, and
\item $G$ is residually nilpotent.
\end{enumerate}

Baumslag provides a sufficient condition for a group to be residually torsion-free nilpotent.

\begin{prop}[{Baumslag \cite[Proposition 2.1(i)]{Baum69}}] \label{baumprop}
Suppose $G$ is a group which is the union of an ascending chain of subgroups as follows.
$$G_0<G_1<G_2<\cdots<G_n<\cdots<G=\bigcup_{n=1}^{\infty}G_n$$
Suppose each $G_n$ is parafree of the same rank.
If for each non-negative integer $n$, $|G_{n+1}:G_n[G_{n+1},G_{n+1}]|$ is finite then
$G$ is residually torsion-free nilpotent.
\end{prop}

Thus, Theorem \ref{mainthm} follows from the following lemma.

\begin{lem}\label{mainlem}
Suppose $L$ is an oriented two-bridge link whose Alexander polynomial has relatively prime coefficients.
The Alexander subgroup $Y$ of $L$ can be written as a union of an ascending chain of subgroups $Y_0<Y_1<Y_2<\cdots<Y$ such that
\begin{enumerate}[label=(\alph*)]
\item each $Y_n$ is parafree of the same rank and \label{mainind}
\item $|Y_{n+1}:Y_n[Y_{n+1},Y_{n+1}]|$ is finite for each $n$. \label{mainquot}
\end{enumerate}
\end{lem}

Let $H$ be a parafree group of rank $r$.
An element $h\in G$ is \emph{homologically primitive}
if the class of $h$ in $H/[H,H]\cong \Z^r$ can be extended to a basis.

\begin{prop}[{Baumslag \cite[Proposition 3]{Baum67}}] \label{baumextprop}
Let $H$ be a parafree group of rank $r$,
and let $\langle t \rangle$ be an infinite cyclic group generated by $t$.
Let $h$ be an element in $H$, and $n$ be a positive prime integer.
If $h$ generates its own centralizer and $h$ is homologically primitive in $H$,
then the group
\[
H\underset{h=x^n}{*}\langle x \rangle
\]
is parafree of rank $r$.
\end{prop}

A theorem of Baumslag \cite[Theorem 4.2]{Baum69} states that any two-generator subgroup of a parafree group is free.
If follows that an element homologically primitive in a parafree group must generate its own centralizer.

Suppose $n$ from Proposition \ref{baumextprop} is composite, and let $n=p_1\cdots p_k$, be the prime decomposition of $n$, and define
\[
	G_j=\la H*\la x_1 \ra*\cdots*\la x_j\ra\; | \;h=x_1^{p_1},x_1=x_2^{p_2},\ldots,x_{j-1}=x_j^{p_j}\ra
\]
for $j=1,\ldots,k$ so
\[
	G_k\cong H\underset{h=x^n}{*}\la x \ra.
\]
For each $j=1,\ldots,k-1$, $x_j$ is homologically primitive in $G_j$.
Therefore, Proposition \ref{baumextprop} is strengthened to the following statement.

\begin{prop} \label{baumextpropr}
Let $H$ be a parafree group of rank $r$,
and let $\langle x \rangle$ be an infinite cyclic group generated by $x$.
Let $h$ be an element in $H$, and $n$ be any positive integer.
If $h$ is homologically primitive in $H$,
then
\[
H\underset{h=x^n}{*}\langle x \rangle
\]
is parafree of rank $r$.
\end{prop}

In a talk, Mayland \cite{May74} proposes a strategy that uses the Reidemeister-Schreier rewriting process to describe the commutator subgroup of a two-bridge knot group as the union of an ascending chain of subgroups satisfying the conditions of Lemma \ref{mainlem}.
The first term $Y_0$ is a free group, and ideally, for each $n\geq 1$, $Y_n$ is isomorphic to $Y_{n-1}$ after adjoining roots of homologically primitive elements, in the manner of Proposition \ref{baumextpropr}, a finite number of times.
Mayland attempts to show that, for a given two-bridge knot, each $Y_n$ is obtained by adjoining roots to $Y_{n-1}$ using a recursive argument.
However, it is not at all obvious that Mayland's recursive argument is valid.
While it is straightforward to verify Mayland's argument on a case by case basis, proving his recursive argument works in general is quite difficult.
Also, in Mayland's talk notes, there are errors in the argument that the elements, whose roots are adjoined, are homologically primitive.
Unfortunately, Mayland never published a proof of his assertion.
In a later paper by Mayland and Murasugi \cite{MayMur76}, it is stated that Mayland plans to present a proof using a different strategy.
This paper has not appeared.

Here we use a slightly different approach. In this paper, we use a graph theoretic construction similar to one used by Hirasawa and Murasugi \cite{HirMur07} to relate the Alexander subgroups of more complicated two-bridge link groups to those of simpler two-bridge link groups.
Then, it is proven inductively that the Alexander subgroups of all two-bridge links can be described by adjoining roots to a free group,
and we show that when two-bridge links have Alexander polynomials with relatively prime coefficients, their Alexander subgroups satisfy Lemma \ref{mainlem} via Mayland's strategy.

\subsection{Application to Bi-Orderability}

Residually torsion-free nilpotence is useful for determining when a link group is \emph{bi-orderable} i.e. admits a total order invariant under both left and right multiplication \cite{PerRolf03, CDN16, Yam17}.
Let $L$ be a smooth link in $S^3$.
The link group $\pi(L)$ is an extension of $\la t\ra$ (an infinite cyclic group generated by $t$) by the Alexander subgroup $Y$.
Let $Y^{\mathsf{ab}}$ denote the abelianization of $Y$,
and let $L_t$ be the linear map induced on $\Q\otimes Y^{\mathsf{ab}}$ by conjugating $Y$ by $t$.
The following result is shown by Linnell, Rhemtulla, and Rolfsen in \cite{LRR08} and is stated more explicitly by Chiswell, Glass, and Wilson \cite{CGW15}.

\begin{thm}[{Chiswell-Glass-Wilson \cite[Theorem B]{CGW15}}]
Suppose $Y$ is residually torsion-free nilpotent.
If the dimension of $\Q\otimes Y^{\mathsf{ab}}$ is finite and all the eigenvalues of $L_t$ are real and positive, then $\pi(L)$ is bi-orderable.
\end{thm}

The Alexander polynomial of $L$, $\Delta_L(t)$, is a scalar multiple of the characteristic polynomial of $L_t$,
and the dimension of $\Q\otimes Y^{\mathsf{ab}}$ is the degree of $\Delta_L(t)$ (see \cite[Chapter VIII]{Rolf76}) which implies the following corollary.

\begin{cor} \label{lrrcor}
Let $L$ be a link in $S^3$.
If the Alexander subgroup of $L$ is residually torsion-free nilpotent and $\Delta_L(t)$ has all real positive roots, then $\pi(L)$ is bi-orderable.
\end{cor}

\begin{remark}
Linnell, Rhemtulla, and Rolfsen actually show a weaker condition on the Alexander polynomial is sufficient for bi-orderability.
However, since two bridge links are alternating, the coefficients of their Alexander polynomials alternate sign \cite{Crow59}
so the signs of the even degree terms are all opposite to the signs of the odd degree terms.
It follows that the Alexander polynomials of two-bridge links cannot have negative roots.
Therefore, for a two-bridge link, having an Alexander polynomial which is ``special" in the sense of Linnell, Rhemtulla, and Rolfsen \cite{LRR08} is equivalent to the Alexander polynomial having all real and positive roots.
\end{remark}

By combining Theorem \ref{mainthm} with Corollary \ref{lrrcor}, we have the following result.

\begin{thm} \label{bocorthm} Let $L$ be an oriented two-bridge link with Alexander polynomial $\Delta_L(t)$. If all the roots of $\Delta_L(t)$ are real and positive and the coefficients of $\Delta_L(t)$ are relatively prime, then the link group of $L$ is bi-orderable.
In particular, if $K$ is a two-bridge knot and all the roots of $\Delta_K(t)$ are real and positive, then the knot group of $K$ is bi-orderable.
\end{thm}

\begin{remark}
    Theorem \ref{bocorthm} is not true if either condition on the Alexander polynomial is removed.
    The link group of the $(4,2)$-torus link has presentation
    \[
        \la x,y| x^{-1}y^{-2}xy^2\ra .
    \]
    Since $x$ and $y$ do not commute but $x$ and $y^2$ does, the $(4,2)$-torus link doe not have bi-orderable link group \cite[Lemma 1.1]{Neu49}.
    As stated in Remark \ref{remrtfn}, the $(4,2)$-torus link, oriented as in Figure \ref{torus42} has Alexander polynomial $2t-2$,
    which as only on real positive root but does not have relatively prime coefficients.
    If we reverse the orientation of one of the components, the Alexander polynomial is $t^3-t^2+t-1$, which has relatively prime coefficients but no real roots.
\end{remark}

\subsection{A Family of Bi-Orderable Two-Bridge Links}

Every oriented two-bridge link is the closure of rational tangle.
Thus, by Conway's correspondence,
we can associate a two-bridge link to a rational fraction $p/q$ with $p>0$; see \cite[Chapter 12]{BurZie85} for details.
Let $L(p/q)$ denote the two-bridge link represented by $p/q$.
Choose an orientation of $L(p/q)$ so that the two overstrands of Schubert’s projection of $L(p/q)$ are oriented away from each other as in Figure \ref{twobridge}.
This correspondence satisfies the following properties:

\begin{figure}[t]
\includegraphics[scale=0.9]{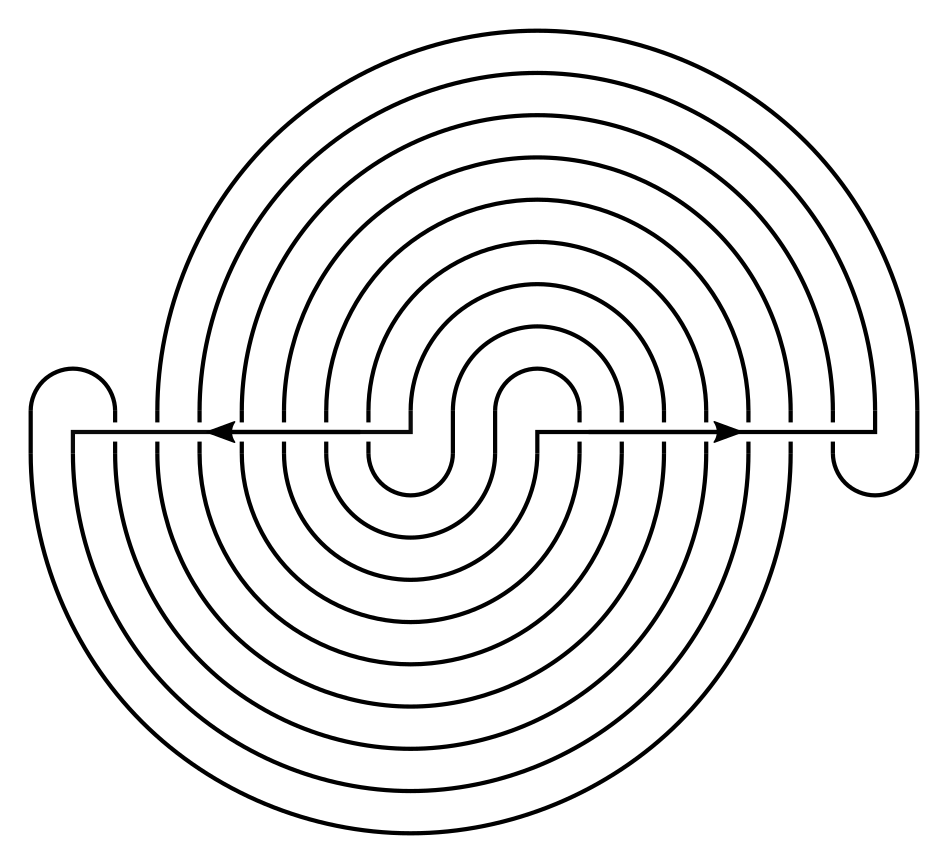}
\caption{Schubert’s projection of $L(8/3)$.}
\label{twobridge}
\end{figure}

\begin{enumerate}
    \item $L(p/q)$ and $L(p'/q')$ are equivalent as unoriented links if and only if
    \begin{enumerate}
        \item $p=p'$ and
        \item $q\cong q'\;(\mathsf{mod}\, p)$ or $qq'\cong 1\;(\mathsf{mod}\, p)$. 
    \end{enumerate}
    \item $L(p/q)$ and $L(p'/q')$ are equivalent as oriented links if and only if
    \begin{enumerate}
        \item $p=p'$ and
        \item $q\cong q'\;(\mathsf{mod}\, 2p)$ or $qq'\cong 1\;(\mathsf{mod}\, 2p)$. 
    \end{enumerate}
    \item $L(p/q)$ is a knot if and only if $p$ is odd.
    \item $L(p/q)$ and $L(-p/q)$ are mirrors.
    \item If $L(p/q)$ is a link, $L(p/(q\pm p))$ is the oriented link obtained by reversing the orientation of one of the components of $L(p/q)$.
\end{enumerate}

When $q$ is odd, there are non-zero integers $k_1,\ldots,k_n$
such that $p/(p-q)=[2k_1,\ldots,2k_n]$.
Here $[2k_1,\ldots,2k_n]$ denotes the continued fraction expansion
\[
[2k_1,\ldots,2k_n]=2k_1+\frac{1}{2k_2+\frac{1}{2k_3+\frac{1}{\cdots+\frac{1}{2k_n}}}}.
\]
The integers $2k_1,\ldots,2k_n$ correspond to the number of twist in the rational tangle $p/q$; see Figure \ref{twobridgeknot}.
For details on fraction expansions and rational tangles, see \cite[Chapter 9]{Mur96}.
When $n$ is even, $L(p/q)$ is a knot with genus $n/2$.
When $n$ is odd, $L(p/q)$ is a two-component link with genus $(n-1)/2$.

\begin{figure}[t]
\includegraphics[scale=0.7]{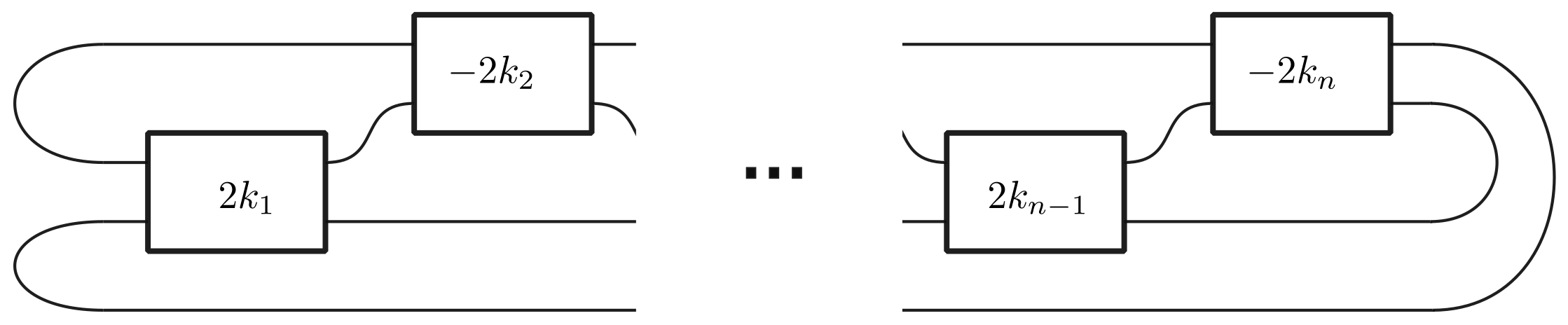}

\includegraphics[scale=0.7]{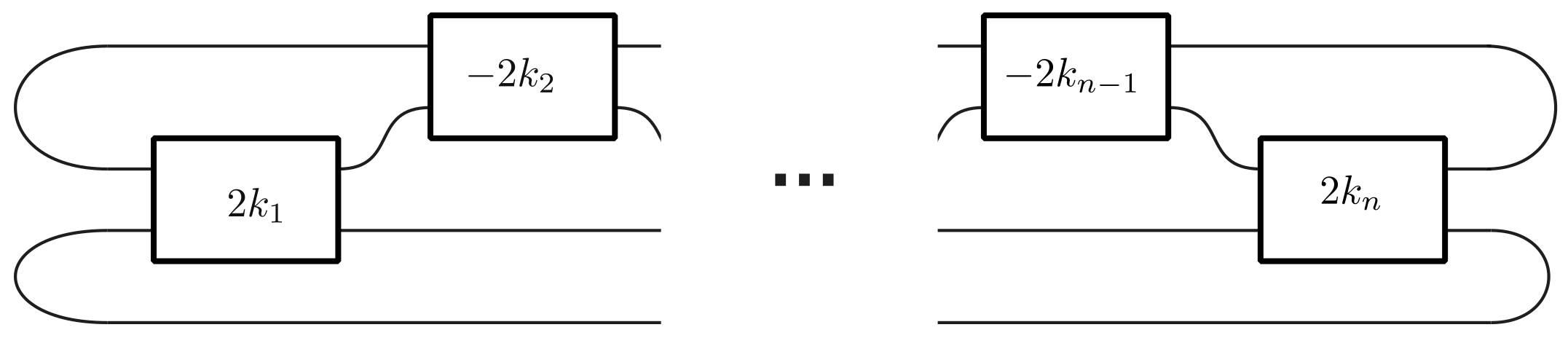}
\caption{Rational tangle form of a two-bridge knot (top) and link (bottom).}
\label{twobridgeknot}
\end{figure}

Every oriented two-bridge link is associated to a fraction $p/q$ with $q$ odd and $|p/q|>1$.
When $L(p/q)$ is a link, $p$ is always even and $q$ is always odd.
Suppose $L(p/q)$ is a knot with $q$ even.
Let $q'$ be the inverse $q$ modulo $2p$.
Since $q$ is even, $q'$ is odd, and
$L(p/q)$ is equivalent to $L(p/q')$.
Furthermore, since $L(p/q)$ is equivalent to $L(p/(q+2pk))$ for all integers $k$, $q$ can be chosen such that $-p<q<p$ so $|p/q|>1$.
Therefore, we adopt the the convention that $p>|q|>0$ and $q$ is odd.

Chiswell, Glass, and Wilson showed that groups which admit presentations with two generators and one relator satisfying certain conditions have residually torsion-free nilpotent commutator subgroups \cite{CGW15}.
Clay, Desmarius, and Naylor used this to show that twist knots (knots represented by $[2,2k]$ with $k>0$) have bi-orderable knot groups in \cite{CDN16}.
In \cite{Yam17}, Yamada used the same idea to extend this to the family of two-bridge links represented by $[2,2,\ldots,2,2k]$ where $k>0$.
Using the following result of Lyubich and Murasugi, this paper extends this family further.

\begin{thm}[{Lyubich-Murasugi \cite[Theorem 2]{LyuMur12}}] \label{lyumurthm}
    Let $p/q$ be a rational fraction,
    and let $L$ be the two-bridge link $L(p/q)$.
    If for some positive integer $n$, $p/q=[2k_1,\ldots,2k_n]$ with $k_i>0$ for each $i=1,\ldots,n$ then
    all the roots of $\Delta_L(t)$ are real and positive.
\end{thm}

Combining this theorem with Corollary \ref{maincor} implies the following.

\begin{cor} \label{positivetwobridgecor}
    Suppose $p/q$ is a rational fraction, and
    $p/(p-q)=[2k_1,\ldots,2k_n]$ with $k_i>0$ for each $i=1,\ldots,n$.

    If the coefficients of the Alexander polynomial of $L(p/q)$ are relatively prime, then the link group of $L(p/q)$ is bi-orderable.
    In particular, when $L(p/q)$ is a knot, the knot group of $L(p/q)$ is bi-orderable.
\end{cor}

Theorem \ref{lyumurthm} does not characterize all two-bridge links with Alexander polynomial that have all real and positive roots.

\begin{example}
Let $K=L(81/49)$.
$81/(81-49)=[2,2,-8,-2]$.
\[
\Delta_K(t)=4t^4-20t^3+33t^2-20t+4=(t-2)^2(2t-1)^2
\]
which has two real roots of multiplicity 2.
Thus, the knot group of $K$ is bi-orderable.
\end{example}

\subsection{Genus One Two-Bridge Links}

Suppose $L$ is an oriented genus one two-bridge link $L(p/q)$.
When $L$ is a genus one knot, $p/(p-q)=[2k_1,2k_2]$ for some non-zero integers $k_1$ and $k_2$.
The Alexander polynomial of $L$ is
\[
    \Delta_{L}(t)=k_1k_2t^2-(2k_1k_2+1)t+k_1k_2.
\]
When $k_1k_2>0$, $\Delta_L(t)$ has two positive real roots
so $\pi(L)$ is bi-orderable by Theorem \ref{bocorthm}.
When $k_1k_2<0$, $\Delta_L(t)$ has no real roots.
In this case, since $\deg\Delta_L=2$, an obstruction by Clay, Desmarais, and Naylor \cite[Theorem 3.3]{CDN16} implies that $\pi(L)$ is not bi-orderable.

\begin{prop}
    Suppose $L$ is the two-bridge knot $L(p/q)$ with $p/(p-q)=[2k_1,2k_2]$.
    The knot group $\pi(L)$ is bi-orderable if and only if $k_1k_2>0$.
\end{prop}

When $L$ is a genus one two-component link, $p/(p-q)=[2k_1,2k_2,2k_3]$ for some non-zero integers $k_1$, $k_2$, and $k_3$.
The Alexander polynomial of $L(p/q)$ is
\begin{align*}
    \Delta_L(t)=&k_1k_2k_3t^3-(3k_1k_2k_3+k_1+k_3)t^2+(3k_1k_2k_3+k_1+k_3)t-k_1k_2k_3\\
    =&(t-1)(k_1k_2k_3t^2-(2k_1k_2k_3+k_1+k_3)t+k_1k_2k_3).
\end{align*}
The discriminant, $D$, of the second factor is
\begin{align*}
    D=&4k_1k_2k_3(k_1+k_3)+(k_1+k_3)^2\\
    =&(k_1+k_3)(k_1(2k_2k_3+1)+k_3(2k_1k_2+1))
\end{align*}
so $D\geq 0$ if and only if $k_1k_2k_3(k_1+k_3)\geq 0$.
It follows that $\Delta_L(t)$ has three real positive roots when $k_1k_2k_3(k_1+k_3)\geq 0$.

Let $A=k_1k_2k_3$ and $B=3k_1k_2k_3+k_1+k_3$.
The coefficients of $\Delta_L$ are relatively prime precisely when $\gcd(A,B)=1$, and
$\gcd(A,B)=1$ if and only if $\gcd(k_1,k_3)=1$ and $\gcd(k_2,k_1+k_3)=1$.

Therefore, Theorem \ref{bocorthm} implies the following result.

\begin{prop}
    Suppose $L$ is the two-component two-bridge link $L(p/q)$ with $p/(p-q)=[2k_1,2k_2,2k_3]$.
    If $\gcd(k_1,k_3)=1$, $\gcd(k_2,k_1+k_3)=1$, and $k_1k_2k_3(k_1+k_3)\geq 0$ then $\pi(L)$ is bi-orderable.
\end{prop}

\subsection{Application to Ribbon Concordance}

The residual torsion-free nilpotence of the commutator subgroup of a knot group has an application to ribbon concordance as well.
Given two knots $K_0$ and $K_1$ in $S^3$,
A \emph{ribbon concordance from $K_1$ to $K_0$} is a smoothly embedded annulus $C$ in $[0,1]\times S^3$
such that $C$ has boundary $-(\{0\}\times K_0) \cup \{1\}\times K_1$ and $C$ has only index 0 and 1 critical points.
$K_1$ is said to be \emph{ribbon concordant} to $K_0$, denoted $K_1\geq K_0$, if there is a ribbon concordance from $K_1$ to $K_0$.
The relation $\geq$ is clearly reflexive and transitive.
Gordon \cite{Gor81} conjectures that $\geq$ is a partial order on knots in $S^3$.

Gordon gives conditions under which $\geq$ behaves anti-symmetrically.
\begin{thm}[Gordon \cite{Gor81}]
If $K_0\geq K_1$ and $K_1\geq K_0$ and the commutator subgroup of $\pi(K_0)$ is transfinitely nilpotent,
then $K_0$ and $K_1$ are ambient isotopic.
\end{thm}

\begin{remark}
Transfinite nilpotence follows from residual torsion-free nilpotence; see \cite{Gor81} for a definition of transfinitely nilpotent.
\end{remark}

Here we state the following corollary.
\begin{cor}
If $K_1\geq K_0$ and $K_0\geq K_1$ and $K_0$ is a two-bridge knot,
then $K_0$ and $K_1$ are ambient isotopic.
\end{cor}

\subsection{Outline}

The rest of this paper is devoted to the proof of Lemma \ref{mainlem}.
Section \ref{PresentationMatrices} covers some preliminaries about presentation matrices of modules over a PID.
In section \ref{ProofExample}, we illustrate the proof of Lemma \ref{mainlem} by verifying the lemma for the two-bridge knot $L(17/13)$.
Section \ref{GroupPres} investigates the properties of a presentation for the Alexander subgroup $Y$ obtained by the Reidemeister-Schreier rewriting procedure.
The proof of Lemma \ref{mainlem} is completed in section \ref{Mainproof}.
In section \ref{CycleGraphs}, we define the cycle graph of a two-bridge link.
Cycle graphs are used to prove a key lemma in section \ref{SecLemmaProof}.

\subsection{Acknowledgments}

The author would like to thank Cameron Gordon for his guidance and encouragement throughout this project.
The author would like to thank Ahmad Issa for providing the example of the knot with all real positive roots.
The author would like to thank Hannah Turner for many helpful writing suggestions and support.
The author would like to thank Jae Choon Cha, Charles Livingston and Allison Moore for creating and maintaining \textit{KnotInfo} \cite{Knotinfo} and \textit{LinkInfo} \cite{Linkinfo} which were invaluable to this project.
This research was supported in part by NSF grant DMS-1937215.

\section{Preliminaries on Presentation Matrices}\label{PresentationMatrices}

Let $R$ be a PID.
Suppose $X$ is an $R$-module
with presentation
\[
    \la x_1,\ldots,x_n | s_1,\ldots,s_m\ra .
\]
For each $i$,
\[
    s_i=\sum_{j=1}^n r_{i,j}x_j
\]
where each $r_{i,j}$ is in $R$.
The matrix of $r_{i,j}$ coefficients
\[
\left(
\begin{array}{ccc}
r_{1,1} & \cdots & r_{1,n} \\
\vdots & & \vdots \\
r_{m,1} & \cdots & r_{m,n} 
\end{array}
\right)
\]
is called a \emph{presentation matrix} of $X$.

Suppose $A$ is a presentation matrix of $X$.
Performing row and column operations on $A$ will always produce another presentation matrix of $X$.
In particular, using row and column operations,
$A$ can be diagonalized into the following form
\[
\left(
\begin{array}{ccc|c}
d_1 & & & \\
 & \ddots & & \text{\huge0}\\
& & d_k&  \\
\hline
& \bigzero{} & & \bigzero{}
\end{array}
\right)
\]
where each $d_i$ is nonzero and $d_i$ divides $d_{i+1}$ for each $i=1,\ldots,k-1$.
Therefore,
\begin{equation} \label{structure}
    X\cong R^{n-k}\oplus \frac{R}{d_1R}\oplus\cdots\oplus\frac{R}{d_kR}.
\end{equation}
The $d_i$ which are not units are the invariant factors of $X$.

The following lemma plays a key role in showing that elements in a parafree group are homologically primitive.

\begin{lem} \label{helpfullemma}
    Suppose $X$ is an $R$-module with an $m\times n$ presentation matrix $A$ of full rank.
    If the greatest common divisor of every $m\times m$ minor of $A$ is a unit, then $X$ is a free $R$-module.
    Otherwise, the greatest common divisor of every $m\times m$ minor of $A$ is equal to the product of the invariant factors of $X$ up to multiplication by a unit.
\end{lem}

\begin{proof}
Let $B$ be $A$ after diagonalization.
Since $A$ has full rank,
$B$ has no extra rows of zeros
so $B$ has the following form.
\[
B=\left(
\begin{array}{ccc|c}
d_1 & & & \\
 & \ddots & & \text{\huge0}\\
& & d_m& 
\end{array}
\right)
\]

For any $m\times n$ matrix with entries in $R$, the greatest common divisor of its $m\times m$ minors is invariant under row and column operations up to multiplication by a unit.
Therefore, up to a unit, the greatest common divisor of the $m\times m$ minors of $A$ is $\prod_{i=1}^m d_i$.
When $\prod_{i=1}^m d_i$ is a unit, each $d_i$ is a unit so
by (\ref{structure}), $X$ is a free $R$-module.
If $\prod_{i=1}^m d_i$ is not a unit, it is the product of the invariant factors of $X$.
\end{proof}

\section{An Example}\label{ProofExample}

In this section, we use the two-bridge knot $K:=L(17/13)$ to provide an example of the proof of Lemma \ref{mainlem}.
Using the Schubert normal form \cite{Schu56}, we obtain a presentation of $\pi(K)$.
\[
    \pi(K)=\la a,b\; | \; avb^{-1}v^{-1} \ra
\]
where
\[
    v = ba^{-1} ba^{-1} b^{-1}a b^{-1}a ba^{-1} ba^{-1} b^{-1}a b^{-1}a.
\]
Denote the Alexander subgroup of $\pi(K)$ by $Y$.
Using the Reidemeister-Schreier rewriting process, we obtain the following presentation of $Y$; see section \ref{GroupPres} for details.
\[
Y\cong\langle \{S_k\}_{k\in\Z}\; | \;\{R_k\}_{k\in\Z}\rangle
\]
where $S_k=a^kba^{-k-1}$ and the relators $R_k$ are defined as follows.
\[
	\vdots
\]		
\begin{align*}
    R_{-1}=& S_0 S_0 S_{-1}^{-1} S_{-1}^{-1}  S_0 S_0 S_{-1}^{-1} S_{-1}^{-1} S_{-1}^{-1}  S_{-2} S_{-2} S_{-1}^{-1} S_{-1}^{-1}  S_{-2} S_{-2} S_{-1}^{-1} S_{-1}^{-1} \\
    R_0=& S_1 S_1 S_0^{-1} S_0^{-1}  S_1 S_1 S_0^{-1} S_0^{-1} S_0^{-1}  S_{-1} S_{-1} S_0^{-1} S_0^{-1}  S_{-1} S_{-1} S_0^{-1} S_0^{-1} \\
    R_1=& S_2 S_2 S_1^{-1} S_1^{-1}  S_2 S_2 S_1^{-1} S_1^{-1} S_1^{-1}  S_0 S_0 S_1^{-1} S_1^{-1}  S_0 S_0 S_1^{-1} S_1^{-1}
\end{align*}
\[
    \vdots
\]

Define a sequence of groups $\{Y_n\}_{n=0}^{\infty}$ as follows.
\begin{align*}
	Y_0:=&\langle S_{-1}, S_0\rangle\\
	Y_1:=&\langle S_{-2}, S_{-1}, S_0, S_1\; | \; R_{-1}, R_0\rangle\\
	Y_2:=&\langle S_{-3}, S_{-2}, S_{-1}, S_0, S_1, S_2\; | \; R_{-2}, R_{-1}, R_0, R_1\rangle
\end{align*}
\[
	\vdots
\]

Define $\widehat{A}_1$, $\widehat{A}_2$, $\widehat{V}_1$ and $\widehat{V}_2$ as follows.
\begin{equation} \label{hats}
\begin{split}
	\widehat{A}_1=&S_1^2S_0^{-2}\\
	\widehat{A}_2=&S_1\\
	\widehat{V}_1=&S_0^{-1}  S_{-1}^2 S_0^{-2} S_{-1}^2 S_0^{-2}\\
	\widehat{V}_2=&S_0^{-2}
\end{split}
\end{equation}
Let $H_1$ be the group obtained by adjoining a square root of $\widehat{V}_1^{-1}$ to $Y_0$ as follows.
\[
	H_1:=Y_0\underset{\widehat{V}_1^{-1}=t_1^2}{*}\la t_1 \ra
\]
Similarly, let $H_2$ be the group obtained by adjoining a square root of $t_1\widehat{V}_2^{-1}$ to $H_1$.
\[
	H_2:=Y_1\underset{t_1\widehat{V}_2^{-1}=S_1^2}{*}\la S_1 \ra
\]
Thus, $H_2$ has the following group presentation.
\begin{align*}
	H_2\cong & \la S_{-1}, S_0, S_1, t_1\; | \; t_1^2\widehat{V}_1=1,t_1=S_1^2\widehat{V}_2\ra \\
	\cong & \la S_{-1}, S_0, S_1\; | \; (S_1^2\widehat{V}_2)^2\widehat{V}_1=1,\ra \\
	\cong & \la S_{-1}, S_0, S_1\; | \; R_0\ra
\end{align*}

Define $\widecheck{A}_1$, $\widecheck{A}_2$, $\widecheck{V}_1$ and $\widecheck{V}_2$ as follows.
\begin{equation} \label{checks}
\begin{split}
	\widecheck{A}_1=&S_{-2}^2S_{-1}^{-2}\\
	\widecheck{A}_2=&S_{-2}\\
	\widecheck{V}_1=&S_0^2 S_{-1}^{-2} S_0^2 S_{-1}^{-3}\\
	\widecheck{V}_2=&S_{-1}^{-2}
\end{split}
\end{equation}
Let $H_3$ be the group obtained by adjoining a square root of $\widecheck{V}_1^{-1}$ to $H_2$.
\[
	H_3:=H_2\underset{\widecheck{V}_1^{-1}=t_2^2}{*}\la t_2 \ra
\]
Let $H_4$ be the group obtained by adjoining a square root of $t_2\widecheck{V}_2^{-1}$ to $H_3$.
\[
	H_4:=H_3\underset{t_2\widecheck{V}_2^{-1}=S_{-2}^2}{*}\la S_{-2} \ra
\]
Therefore, $H_4$ is isomorphic to $Y_1$.
\begin{align*}
	H_4\cong & \la S_{-2}, S_{-1}, S_0, S_1, t_2\; | \; \widecheck{V}_1t_2^2=1,t_2=S_{-2}^2\widecheck{V}_2\ra \\
	\cong & \la S_{-2}, S_{-1}, S_0, S_1\; | \; R_{-1}, R_0\ra \\
	\cong & Y_1
\end{align*}

In conclusion, $Y_1$ is $Y_0$ after adjoining roots four times,
and since $R_{n\pm1}$ is $R_{n}$ with all the subscripts changed by $\pm1$,
$Y_{n+1}$ is $Y_n$ after adjoining roots four times.
Thus, for each $n$, $Y_n$ embeds into $Y_{n+1}$, and $|Y_{n+1}:Y_n[Y_{n+1},Y_{n+1}]|$ is finite.
Therefore, $Y$ is the union of an ascending chain of subgroups as follows.
\[
	Y_0< Y_1<\cdots< Y=\bigcup_{n=0}^{\infty}Y_n
\]

By Proposition \ref{baumprop}, if each $Y_n$ is parafree of the same rank then $Y$ is residually torsion-free nilpotent.
$Y_0$ is clearly parafree of rank 2 since it is a rank 2 free group.
We need to verify that each time we adjoin a root of an element, that element is homologically primitive.
Then, by Proposition \ref{baumextpropr}, we can conclude that each $Y_n$ is also parafree of rank 2.
\bigskip{}

\begin{claim}{}
    For each $n\geq0$, if $Y_n$ is parafree of rank 2, then so is $Y_{n+1}$.
\end{claim}

\begin{proof}
Let $n$ be a non-negative integer, and suppose $Y_n$ is parafree of rank 2.
In an abuse of notation, let $\widehat{A}_1$, $\widehat{A}_2$, $\widehat{V}_1$ and $\widehat{V}_2$ be as defined in (\ref{hats}) except with the subscripts of each $S_i$ increased by $n$.
Similarly, let $\widecheck{A}_1$, $\widecheck{A}_2$, $\widecheck{V}_1$ and $\widecheck{V}_2$ be as defined in (\ref{checks}) except with the subscripts of each $S_i$ decreased by $n$.
Also, let $H_1$, $H_2$, $H_3$ and $H_4$ be the groups obtained by adjoining square roots of $\widehat{V}_1^{-1}$, $t_1\widehat{V}_2^{-1}$, $\widecheck{V}_1^{-1}$ and $t_2\widecheck{V}_2^{-1}$ to $Y_n$ as before.

Let $Y_n^{\mathsf{ab}}$ denote the abelianization of $Y_n$,
and let $B_1$ be the quotient of $Y_n^{\mathsf{ab}}$ obtained by killing the class of $\widehat{V}_1^{-1}$ in $Y_n^{\mathsf{ab}}$.
Since $Y_n$ is parafree of rank 2, $Y_n^{\mathsf{ab}}\cong\Z\oplus\Z$.
Thus,
\[
    B_1\cong\Z\oplus\frac{\Z}{C\Z}
\]
for some integer $C$.

Now, we view $Y_n^{\mathsf{ab}}$ as a $\Z$-module and use addition as the group operation.
$Y_n^{\mathsf{ab}}$ is generated by
$S'_{-n-1},S_{-n}',\dots,S_n'$
where $S'_i$ donotes the class of $S_i$ in $Y_n^{\mathsf{ab}}$.
Using this generating set, $Y_n^{\mathsf{ab}}$ has a $(2n)\times (2n+2)$ presentation matrix:
\[
\left(
\begin{array}{cccccc}
4 & -9 & 4 & & & \\
 & 4 & -9 & 4 & & \\
 & & \ddots & \ddots & \ddots & \\
 & & & 4 & -9 & 4
\end{array}
\right) .
\]
The class of $\widehat{V}_1^{-1}$ in $Y_n^{\mathsf{ab}}$ is $-4S_{n-1}'+5S_n'$.
Thus, $B_1$ has the following $(2n+1)\times(2n+2)$ presentation matrix,
which we will also call $B_1$.
\[
B_1=\left(
\begin{array}{cccccc}
4 & -9 & 4 & & & \\
 & 4 & -9 & 4 & & \\
 & & \ddots & \ddots & \ddots & \\
 & & & 4 & -9 & 4\\
 & & & & -4 & 5
\end{array}
\right)
\]
By Lemma \ref{helpfullemma}, the integer $C$ is the greatest common divisor of the determinants of every $(2n+1)\times(2n+1)$ submatrix of $B_1$.
By deleting the last column, we get a square submatrix of $B_1$ with determinant $-4^{2n+1}$.
However, by deleting the first column, we see $B_1$ has a submatrix with odd determinant. (Modulo 2, $B_1$ is the identity matrix.)
Thus, $C=1$.

Therefore, $B_1$ is a rank 1 free abelian group.
It follows that $\widehat{V}_1^{-1}$ is homologically primitive in $Y_n$, and
$H_1$ is parafree of rank 2 by Proposition \ref{baumextpropr}.

Let $B_2$ be the quotient of $H_1^{\mathsf{ab}}$ obtained by killing the class of $t_1\widehat{V}_2^{-1}$ in $H_1^{\mathsf{ab}}$, the abelianization of $H_1$.
$H_1^{\mathsf{ab}}$ is generated by
$S'_{-n-1},S_{-n}',\dots,S_n',t'_1$
where $t'_1$ is the class of $t_1$ in $H_1^{\mathsf{ab}}$.
$H_1^{\mathsf{ab}}$ has a $(2n+1)\times (2n+3)$ presentation matrix:
\[
\left(
\begin{array}{ccccccc}
4 & -9 & 4 & & & & \\
 & 4 & -9 & 4 & & & \\
 & & \ddots & \ddots & \ddots & & \\
 & & & 4 & -9 & 4 & \\
 & & & & -4 & 5 & 2
\end{array}
\right) .
\]
The class of $t_1\widehat{V}_2^{-1}$ in $H_1^{\mathsf{ab}}$ is $2S_n'+t'_1$.
Thus, $B_2$ has the following $(2n+2)\times(2n+3)$ presentation matrix.
\[
B_2=\left(
\begin{array}{ccccccc}
4 & -9 & 4 & & & & \\
 & 4 & -9 & 4 & & & \\
 & & \ddots & \ddots & \ddots & & \\
 & & & 4 & -9 & 4 & \\
 & & & & 4 & -5 & 2\\
 & & & & & 2 & 1
\end{array}
\right)
\]

Using the 1 in the bottom right corner, we apply a row operation and kill the last row and column to get the following presentation matrix.
\[
B_2\cong\left(
\begin{array}{cccccc}
4 & -9 & 4 & & & \\
 & 4 & -9 & 4 & & \\
 & & \ddots & \ddots & \ddots & \\
 & & & 4 & -9 & 4\\
 & & & & 4 & -9 \\
\end{array}
\right)
\]

Thus, $B_2$ is a rank 1 free abelian group, by a argument similar to the one used for $B_1$.
It follows that $t_1\widehat{V}_2^{-1}$ is homologically primitive in $H_1$, and
$H_2$ is parafree of rank 2 by Proposition \ref{baumextpropr}.

Similarly, $\widecheck{V}_1^{-1}$ and $t_2\widecheck{V}_2^{-1}$ are homologically primitive in $H_2$ and $H_3$ respectively.
Therefore, $H_4\cong Y_{n+1}$ is parafree of rank 2. \end{proof}

Since $Y_0$ is parafree of rank 2, each $Y_n$ is parafree of rank 2 by induction.
Also, $|Y_{n+1}:Y_n[Y_{n+1},Y_{n+1}]|=16$.
Therefore, $Y$ is residually torsion-free nilpotent by Proposition \ref{baumprop}.

\section{A Group Presentation of the Alexander Subgroup}\label{GroupPres}

In this section, we give a group presentation of the Alexander subgroup of an arbitrary two-bridge link group using the Reidemeister-Schreier rewriting process.
From this presentation of the Alexander subgroup, we can describe the subgroup as the union of an ascending chain of subgroups which satisfy conditions \ref{mainind} and \ref{mainquot} of Lemma \ref{mainlem} when the Alexander polynomial of the link has relatively prime coefficients.

\subsection{A Presentation from Reidemeister-Schreier} \label{reidschpres}

Consider the 2-bridge link $L:=L(p/q)$ where $1\leq |q|<p$ with $q$ odd.
For each integer $i$, define
\begin{equation} \label{epsilons}
    \epsilon_i:=(-1)^{\lfloor \frac{iq}{p}\rfloor}.
\end{equation}

\begin{prop}[{Schubert \cite{Schu56}}]
    Given the 2-bridge link $L(p/q)$,
    \[
        \pi(L(p/q))\cong\langle a,b | w\rangle
    \]
    where $w=a^{\epsilon_0}b^{\epsilon_1}\ldots a^{\epsilon_{2p-2}}b^{\epsilon_{2p-1}}$.
\end{prop}

Let $Y$ be the Alexander subgroup of $L$.
A group presentation for $Y$ can be obtained using the Reidemeister-Schreier rewriting procedure,
developed by Reidemeister \cite{Reid32} and Schreier \cite{Schr27},
which is described in detail in section 2.3 of the text by Karrass, Magnus, and Solitar \cite{KMS66}.
The application of this procedure to the situation at hand is discussed below.

Consider $\mathcal{A}:=\{a^k\}_{k\in\Z}$ as a set of coset representatives for $\pi(L)/Y$.
Given an element $x$ in $\pi(L)$, let $\overline{x}$ be the coset representative of $x$ in $\mathcal{A}$.
For each $x\in\{a,b\}$ and $k\in\Z$, define
\[
    \gamma(a^k,x):=a^kx(\overline{a^kx})^{-1}.
\]
Note that $\gamma(a^k,a)=1$, and $\gamma(a^k,b)=a^kba^{-k-1}$.
Given a word $u=x_1^{s_1}x_2^{s_2}\cdots x_n^{s_n}$ with $x_i\in\{a,b\}$ and $s_i\in\{1,-1\}$ for all $i$, define
\[
    \tau(u):=\gamma(\overline{t_1},x_1)^{s_1}\gamma(\overline{t_2},x_2)^{s_2}\cdots\gamma(\overline{t_n},x_n)^{s_n}
\]
where
\[
    t_i:=
    \left\{\begin{array}{ll}
        x_1^{s_1}\cdots x_{i-1}^{s_{i-1}}\text{ (possibly trivial)},&s_i=1\\
        x_1^{s_1}\cdots x_i^{s_i},&s_i=-1\\
    \end{array}\right.
    .
\]

For each integer $k$, define
\[
    S_k:=\gamma(a^k,b).
\]
and define
\[
	\mathcal{S}:=\{S_k\}_{k\in\Z}
\]
Since, for all $k$, $\gamma(a^k,a)=1$, for each word $u$, $\tau(u)$ is a product $S_{k_1}S_{k_2}\cdots S_{k_l}$.
For each integer $k$, define
\[
    R_k:=\tau(a^kwa^{-k}).
\]

Define
\begin{equation}\label{sigmas}
    \sigma_i:=\left\{
    \begin{array}{ll}
        \sum_{j=0}^{i-1}\epsilon_j & \text{when }i> 0\\
        \sum_{j=i}^{-1}\epsilon_j & \text{when }i< 0\\
        0&\text{when }i=0
    \end{array}
    \right.
\end{equation}
for each integer $i$.

\begin{prop} \label{subscriptsprop}
    Suppose $R_0=\tau(w)=S_{i_1}^{\eta_1}S_{i_2}^{\eta_2}\ldots S_{i_{n}}^{\eta_{n}}$
    where each $i_j$ is an integer
    and each $\eta_j$ is $\pm1$.
    Then,
    \begin{enumerate}[label=(\alph*)]
        \item $n=p$, \label{ssplen}
        \item $\eta_j=\epsilon_{2j-1}$, for each $j=1,\ldots,p$ \label{sspexp},
        \item $i_j=\sigma_{2j}$ if $\eta_j=1$ and $i_j=\sigma_{2j+1}$ if $\eta_j=-1$ for each $j=1,\ldots,p$, and \label{sspsub}
        \item for every integer $k$, $R_k=S_{i_1+k}^{\eta_1}S_{i_2+k}^{\eta_2}\ldots S_{i_{p+k}}^{\eta_{p}}$. \label{ssprel}
    \end{enumerate}
\end{prop}

\begin{proof}
    Since $\gamma(a^k,a)$ is trivial, the $S_i$-generators in $R_0$ come from the $b$-generators in $w$.
    For \ref{ssplen}, notice that the length of the word $R_0$ is the number of times $b$ and $b^{-1}$ appear in $w$ which is equal to $p$.
    By definition $\eta_j$ is equal to the exponent of the corresponding $b$ or $b^{-1}$ in $w$ which is $\epsilon_{2j-1}$ showing \ref{sspexp}.
    Since $a=b$ modulo $Y$, then for any word $u$ in $a$ and $b$, $\overline{u}=a^s$ where $s$ is the sum of the exponents of the $a$'s and $b$'s in $u$.
    Thus, both \ref{sspsub} and \ref{ssprel} follow by a straightforward computation.
\end{proof}

\begin{prop}[{Karrass-Magnus-Solitar \cite[Theorem 2.9]{KMS66}}]\label{reishrei}
    \[
        Y\cong\langle \{S_k\}_{k\in\Z}\; | \;\{R_k\}_{k\in\Z}\rangle
    \]
\end{prop}

\subsection{Group Presentation Properties}

This group presentation of $Y$ has a few notable properties which will be of use.

Given a word $W$ in $\mathcal{S}$,
let $[W]$ denote the class of $W$ in the free abelian group generated by $\mathcal{S}$.
For each integer $k$, define $S'_k:=[S_k]$.
Denote the maximal and minimal subscripts of $S$ appearing in the word $R_0$ by $M$ and $m$ respectively so that
\[
    [R_0]=a_MS'_{M}+a_{M-1}S'_{M-1}+\cdots+a_{m+1}S'_{m+1}+a_mS'_{m}.
\]
for some integers $a_m,\ldots,a_M$.

\begin{prop}\label{alexanderprop}
    Suppose $L$ is a two-bridge link, and
    suppose $Y$ is the Alexander subgroup of $L$ with presentation as defined in section \ref{reidschpres}.
    \begin{enumerate}[label=(\alph*)]
        \item \label{alexalln} For each integer $n$,
	        \[
	            [R_n]=a_MS'_{M+n}+a_{M-1}S'_{M-1+n}+\cdots+a_{m+1}S'_{m+1+n}+a_nS'_{m+n}.
	        \]
	    \item \label{alexmaxmin} Let $g$ be the genus of $L$.
	    When $L$ is a knot, $M-m=2g$, and when $L$ is a link, $M-m=2g+1$.
	    \item \label{alexcoef} For all $j=m,\ldots,M$
        \[
            a_j=\left\{
            \begin{array}{ll}
                \underline{a}_{g+m-j} & \text{if } m\leq j \leq m+g \\
                \underline{a}_{g+j-M} & \text{if } M-g\leq j \leq M
            \end{array}
            \right.
        \]
        where
        \[
            \Delta_L(t)=\underline{a}_gt^{2g}+\cdots+\underline{a}_0t^g+\cdots+\underline{a}_g
        \]
        when $L$ is a knot, and
        \[
            \Delta_L(t)=\underline{a}_gt^{2g+1}+\cdots+\underline{a}_0t^{g+1}+\underline{a}_0t^{g}+\cdots+\underline{a}_g
        \]
        when $L$ is a link.
        In particular, for all $j=0,\ldots, M-m$,
        \[
            a_{M-j}=a_{m+j}.
        \]
    \end{enumerate}
\end{prop}

\begin{proof}
    Part \ref{alexalln} follows from Proposition \ref{subscriptsprop}\ref{ssprel}.
    
    For each $i=1,\ldots,2p$, denote by $w_i$ the word obtained from the first $i$ generators of the relation $w$.
    Also, define
    \[
	    \theta(s):=
	    \left\{\begin{array}{ll}
	        1 & \text{if }s=1\\
	        0 & \text{if }s=-1
	    \end{array}
	    \right. .
    \]	
    We compute the Alexander polynomial by performing Fox calculus on $w$ with respect to $b$ (see \cite[Section 3]{Fox53}),
    \begin{align*}
        \frac{\partial w}{\partial b}=&a^{\epsilon_0}\Big(\frac{\partial}{\partial b}(b^{\epsilon_1})
        +b^{\epsilon_1}a^{\epsilon_2}\Big(\frac{\partial}{\partial b}(b^{\epsilon_3})+\cdots+
        b^{\epsilon_{2p-3}}a^{\epsilon_{2p-2}}\Big(\frac{\partial}{\partial b}(b^{\epsilon_{2p-1}})
        \Big)\\
        =&\sum_{i=1}^{p}w_{2i-1}\frac{\partial}{\partial b}(b^{\epsilon_{2i-1}})\\
        =&\sum_{i=1}^{p}\epsilon_{2i-1} w_{f(i)}
    \end{align*}
    where
    \[
        f(i)=2i-\theta(\epsilon_{2i-1}).
    \]

    For each $i=1,\ldots,2p$, $\overline{w_i}=a^{\sigma_i}$.
    Let $t$ the generator of $\pi(L)/Y$ which is identified with $\overline{a}=\overline{b}$.
    Under the quotient map $\pi\circ h$ from (\ref{extendingmap}).
    Up to multiplication by powers of $t$,
    \begin{equation}\label{foxcalc}
        \Delta_L(t)=\pi'\Big(\frac{\partial w}{\partial b}\Big)=\sum_{i=1}^{p}\epsilon_{2i-1} t^{\sigma_{f(i)}}
    \end{equation}
    where $\pi':\Z[\pi(L)]\to\Z[t]$ is the map induced by $\pi\circ h$.

    By Proposition \ref{subscriptsprop},
    \[
        R_k=S_{\sigma_{f(1)}}^{\epsilon_1}S_{\sigma_{f(2)}}^{\epsilon_3}
        \cdots S_{\sigma_{f(p)}}^{\epsilon_{2p-1}}
    \]
    so
    \begin{equation}\label{abelrk}
        \begin{split}
            [R_k]=&\epsilon_1S'_{\sigma_{f(1)}}+\epsilon_3S'_{\sigma_{f(2)}}+
            \cdots+\epsilon_{2p-1}S'_{\sigma_{f(p)}}\\
            =&\sum_{i=1}^{p}\epsilon_{2i-1}S'_{\sigma_{f(i)}}.
        \end{split}
    \end{equation}
    Parts \ref{alexmaxmin} and \ref{alexcoef} follow from (\ref{foxcalc}) and (\ref{abelrk}).
\end{proof}

\subsection{An Ascending Chain of Subgroups}

With the group presentation from Proposition \ref{reishrei}, we can describe $Y$ as an ascending chain of subgroups.

Define $Y_0$ to be the free group
\begin{equation}\label{Ys1}
	Y_0:=\la S_m,S_{m+1},\ldots,S_{M-1}\ra,
\end{equation}
and define $Y_n$ to be the group with presentation
\begin{equation}\label{Ys2}
	Y_n:=\la S_{m-n},S_{m-n+1},\ldots,S_{M+n-1}\; | \;R_{-n},\ldots,R_{n-1}\ra.
\end{equation}
for each positive integer $n$.

$Y_{n+1}$ is $Y_n$ with two extra generators, $S_{m-n-1}$ and $S_{M+n}$, and two extra relators, $R_{-n-1}$ and $R_n$.
It turns out that all of the appearances of $S_{M+n}$ in $R_{n}$ are contained in nested repeating patterns of words.
Similarly, all of the appearances of $S_{m-n-1}$ in $R_{-n-1}$ are contained in nested repeating patterns of words.
Given an explicit two-bridge link, one can find these patterns easily, as we did in section \ref{ProofExample} for $L(17/13)$,
yet showing that these patterns exist for any two-bridge knot is much more complicated.

Once it is established that these patterns exists, however, it follows that for each non-negative integer $n$, $Y_{n+1}$ is $Y_n$ after adjoining roots a finite number of times.
This implies that each $Y_n$ embeds into $Y_{n+1}$.
Since $Y$ is the direct limit of the sequence of $Y_n$'s,
$Y$ is the union of the ascending chain of $Y_n$'s.
When the coefficients of $\Delta_L$ are relatively prime, the elements whose roots are adjoining are homologically primitive.

The following lemma explicitly describes the relator $R_0$ (and hence any $R_n$ by Proposition \ref{subscriptsprop}) as nested patterns of repeating words.

\begin{lem}\label{nestedword}
There exist a positive integer $N$, sequences of words in $\mathcal{S}$,
\[
\widehat{A}_0,\widehat{A}_1,\ldots,\widehat{A}_N,
\]
and
\[
\widehat{V}_1,\ldots,\widehat{V}_N,
\]
and a sequence of positive integers $n_1,\ldots,n_N$
such that all of the following hold:
\begin{enumerate}[label=(M\arabic*)]
\item $R_0=\widehat{A}_0$, \label{nwstart}
\item $\widehat{A}_N=S_M^{\pm1}$, \label{nwend}
\item for each $i=1,\ldots,N$,
$\widehat{A}_{i-1}=\widehat{A}_i^{n_i}\widehat{V}_i$ (up to conjugation), \label{nwrec}
\item for each $i=1,\ldots,N$,
$S_M^{\pm1}$ does not appear in $\widehat{V}_i$, and  \label{nwvsm}
\item for each $i=1,\ldots,N$,
there is some $l$ with $m<l\leq M$ and integers $b_l,\ldots,b_M$ (which depend on $i$) such that
\[
[\widehat{A}_i]=\sum_{j=l}^Mb_jS'_j=b_lS'_{l}+b_{l+1}S'_{l+1}+\cdots+b_MS'_M
\]
with $|b_{l+j}|=|b_{M-j}|$.\label{nwsym}
\end{enumerate}
Also, there are sequences
\[
\widecheck{A}_0,\widecheck{A}_1,\ldots,\widecheck{A}_N,
\]
and
\[
\widecheck{V}_1,\ldots,\widecheck{V}_N,
\]
such that
\begin{enumerate}[label=(m\arabic*)]
\item $R_0=\widecheck{A}_0$, \label{nwstart2}
\item $\widecheck{A}_N=S_m^{\pm1}$, \label{nwend2}
\item for each $i=1,\ldots,N$,
$\widecheck{A}_{i-1}=\widecheck{A}_i^{n_i}\widecheck{V}_i$ (up to conjugation), \label{nwrec2}
\item for each $i=1,\ldots,N$,
$S_m^{\pm1}$ does not appear in $\widecheck{V}_i$, and  \label{nwvsm2}
\item for each $i=1,\ldots,N$,
there is some $l'$ with $m\leq l'< M$, and integers $b_m,\ldots,b_{l'}$ (which depend on $i$) such that
\[
[\widecheck{A}_i]=\sum_{j=m}^{l'}b_jS'_j=b_mS'_{m}+\cdots+b_{l'}S'_{l'}
\]
with $|b_{m+j}|=|b_{l'-j}|$.\label{nwsym2}
\end{enumerate}
\end{lem}

\begin{remark}
$Y_1$ is obtained from $Y_0$ by adding $2N$ roots.
In order of increasing index, each $\widehat{A}_i$ is added as the $n_i$th root of some element,
then each $\widecheck{A}_i$ is added as an $n_i$th root.
The conditions \ref{nwsym} and \ref{nwsym2} are used to show that the elements whose roots are added are homologically primitive.
\end{remark}

Lemma \ref{nestedword} is proven in section \ref{LemmaProof}.

\begin{prop} \label{propadjoiningroots}
The Alexander subgroup $Y$ of any oriented two-bridge link is a union of an ascending chain of subgroups
\[
	Y_0<Y_1<Y_2<\cdots<Y_i<\cdots<\bigcup_{n=1}^{\infty}Y_n\cong Y
\]
where $Y_{n+1}$ is obtained from $Y_n$ by adjoining a finite number of roots.
\end{prop}

\begin{proof}
Define the sequence $Y_0,Y_1,Y_2,\ldots$ as in (\ref{Ys1}) and (\ref{Ys2}).
Consider $Y_n$ for some non-negative integer $n$.
\[
    	Y_n=\la S_{m-n},\ldots,S_{M+n-1}\; | \;R_{-n},\ldots,R_{n-1}\ra
\]
and
\[
    	Y_{n+1}=\la S_{m-n-1},\ldots,S_{M+n}\; | \;R_{-n-1},\ldots,R_{n}\ra .
\]

By Proposition \ref{subscriptsprop}\ref{ssprel} and Lemma \ref{nestedword} there is an integer $N$, sequences of words
\[
	\widehat{A}_0,\ldots,\widehat{A}_N,
\]
and
\[
	\widehat{V}_1,\ldots,\widehat{V}_N,
\]
and a sequence of integers
\[
	n_1,\ldots,n_N.
\]
such that
\[
	\widehat{A}_0=R_n,
\]
\[
	\widehat{A}_N=S^{\pm}_{M+n},
\]
and for some $\widehat{W}_i$,
\[
	\widehat{W}_i^{-1}\widehat{A}_{i-1}\widehat{W}_i=\widehat{A}_i^{n_i}\widehat{V}_i
\]
for each $i=1,\ldots,N$.

Let $\langle t_i\rangle$ be an infinite cyclic group generated by $t_i$ for each $i=1,\ldots,N$.
Also, let $t_0$ be trivial in $Y_n$.

Define 
\begin{equation} \label{H0}
H_0=Y_n,
\end{equation}
and for each $i=1,\ldots,N$, recursively define
\begin{equation} \label{Hs}
H_i=H_{i-1}\underset{\widehat{h}_i=t_i^{n_i}}{*}\langle t_i\rangle
\end{equation}
where
\[
\widehat{h}_i=\widehat{W}_i^{-1}t_{i-1}\widehat{W}_i\widehat{V}_i^{-1}.
\]
Thus,
\begin{align*}
    H_{N}\cong\la S_{m-n},\ldots,S_{M+n}, t_1,\ldots,t_N\; | \; & R_{-n},\ldots,R_{n-1}, \\
     & \{\widehat{h}_i^{-1}t_i^{n_i}\}_{i=2}^{N}, \\
     & \widehat{V}_1t_1^{n_1}, t_{N}^{-1}\widehat{A}_{N} \ra.
\end{align*}
By backwards substitution using \ref{nwstart}, \ref{nwend}, and \ref{nwrec} of Lemma \ref{nestedword},
\begin{align*}
    H_N\cong&\la S_{m-n},\ldots,S_{M+n}, t_1,\ldots,t_N\; | \; R_{-n},\ldots,R_{n-1},\widehat{A}_0, t_1^{-1}\widehat{A}_1, \ldots, t_N^{-1}\widehat{A}_N \ra \\
    \cong&\la S_{m-n},\ldots,S_{M+n}\; | \;R_{-n},\ldots,R_n\ra.
\end{align*}

Likewise, by Proposition \ref{subscriptsprop}\ref{ssprel} and Lemma \ref{nestedword} there are sequences of words
\[
	\widecheck{A}_0,\ldots,\widecheck{A}_N,
\]
and
\[
	\widecheck{V}_1,\ldots,\widecheck{V}_N,
\]
such that
\[
	\widecheck{A}_0=R_{-n-1},
\]
\[
	\widecheck{A}_N=S^{\pm}_{m-n-1},
\]
and for some $\widecheck{W}_i$,
\[
	\widecheck{W}_i^{-1}\widecheck{A}_{i-1}\widecheck{W}_i=(\widecheck{A}_i)^{n_i}\widecheck{V}_i
\]
for each $i=1,\ldots,N$.

For each $i=1,\ldots,N$, define
\begin{equation} \label{Hs2}
H_{i+N}=H_{i+N-1}\underset{\widecheck{h}_i=t_i^{n_i}}{*}\langle t_i\rangle
\end{equation}
where
\[
    \widecheck{h}_i=\widecheck{W}_i^{-1}t_{i-1}\widecheck{W}_i\widecheck{V}_i^{-1}.
\]

\begin{align*}
    H_{2N}\cong\la S_{m-n-1},\ldots,S_{M+n}, t_1,\ldots,t_N\; | \; & R_{-n},\ldots,R_n, \\
     & \{\widecheck{h}_i^{-1}t_i^{n_i}\}_{i=2}^{N}, \\
     & \widecheck{V}_1t_1^{n_1}, t_{N}^{-1}\widecheck{A}_{N} \ra.
\end{align*}
By backwards substitution using \ref{nwstart2}, \ref{nwend2}, and \ref{nwrec2} of Lemma \ref{nestedword},
\begin{equation} \label{Hinduct}
\begin{split}
    H_{2N}\cong&\la S_{m-n-1},\ldots,S_{M+n}, t_1,\ldots,t_N\; | \; R_{-n},\ldots,R_{n-1}, \\
    &\hspace{5cm} \widecheck{A}_0, t_1^{-1}\widecheck{A}_1, \ldots, t_N^{-1}\widecheck{A}_N \ra \\
    \cong&\la S_{m-n-1},\ldots,S_{M+n}\; | \;R_{-(n+1)},\ldots,R_n\ra \\
    \cong&Y_{n+1}.
\end{split}
\end{equation}

Consider $Y_n$ and $Y_{n+1}$ for a non-negative integer $n$.
For each $i=0,\ldots,2N-1$, $H_i$ embeds into $H_{i+1}$ since $H_{i+1}$ is a free product of $H_i$ and $\Z$ amalgamated along infinite cyclic subgroups.
Let $\varphi_i:H_i\to H_{i+1}$ be the embedding which maps $S_k\mapsto S_k$ and $t_k\mapsto t_k$ for all $k$.
The composition $f_n=\varphi_{2N-1}\circ\cdots\circ\varphi_{0}$ is an embedding of $Y_n$ into $Y_{n+1}$ which maps $S_k\mapsto S_k$ for all $k$.

Thus, we have the following sequence of embeddings.
\[
	Y_0\overset{f_0}{\longhookrightarrow} Y_1\overset{f_1}{\longhookrightarrow}Y_2\overset{f_2}{\longhookrightarrow}\cdots\overset{f_{n-1}}{\longhookrightarrow}Y_n\overset{f_n}{\longhookrightarrow}\cdots
\]
The Alexander subgroup $Y$ is the direct limit of this sequence, since each $f_n$ is an embedding,
$Y$ is a union of an ascending chain of subgroups as desired.
\end{proof}

\subsection{Proof of Lemma \ref{mainlem}} \label{Mainproof}

We now turn our attention to proving Lemma \ref{mainlem}.
First, we state a more precise and detailed version of Lemma \ref{mainlem}.

\begin{lem}
Suppose that $Y$ is the Alexander subgroup of a two-bridge link whose Alexander polynomial has relatively prime coefficients so that $Y$ is an ascending chain of subgroups
\[
	Y_0<Y_1<Y_2<\cdots<Y=\bigcup_{n=1}^{\infty}Y_n
\]
as defined in (\ref{Ys1}) and (\ref{Ys2}).
For each $n$,
\begin{enumerate}[label=(\alph*)]
\item $Y_n$ is parafree of the rank $M-m$ and \label{mainindr}
\item $|Y_{n+1}:Y_n[Y_{n+1},Y_{n+1}]|=\underline{a}_g^2$ where $\underline{a}_g$ is the leading coefficient of the Alexander polynomial of $L$. \label{mainquotr}
\end{enumerate}
\end{lem}

\begin{proof}
First we show \ref{mainindr}.
$Y_0$ is a parafree of rank $M-m$ since it's a rank $M-m$ free group.
Suppose that for some $n\geq 0$, $Y_n$ is parafree of rank $M-m$.
Define $H_0,\ldots,H_{2N}$ as in (\ref{H0}), (\ref{Hs}), and (\ref{Hs2}) so $H_{2N}\cong Y_{n+1}$ as in (\ref{Hinduct}).

Suppose $H_{k-1}$ is parafree of rank $M-m$ for some $k$ such that $0<k\leq N$ so $H_{k-1}^{\mathsf{ab}}\cong\Z^{M-m}$.
Define
\[
    B:=\frac{H_{k-1}}{\la \widehat{h}_k \ra[H_{k-1},H_{k-1}]}\cong\Z^{M-m-1}\oplus\frac{\Z}{C\Z}
\]
where
\[
    \widehat{h}_k=\widehat{W}_k^{-1}t_{k-1}\widehat{W}_k\widehat{V}_k^{-1}
\]
and $C$ is an integer.
If $B\cong \Z^{M-m-1}$, then $\widehat{h}_k$ is homologically primitive in $H_{k-1}$, and inductively, by Proposition \ref{baumextpropr}, each $H_k$ is parafree of rank $M-m$.

By Proposition \ref{alexanderprop}, $H^{\mathsf{ab}}_0=Y^{\mathsf{ab}}_n$ has $2n\times 2n+M-m$ presentation matrix
\[
\left(
\begin{array}{ccccccccccccc}
a_m & a_{m+1} & \cdots & a_{M-1} & a_M \\
& \ddots & & \ddots & & \ddots \\
& & a_m & a_{m+1} & \cdots & a_{M-1} & a_M 
\end{array}\right) .
\]
$H_{k-1}$ is $H_0$ with the $n_j$ root of $\widehat{h}_j$ added for each $j=1,\ldots,k-1$.
Thus, $B$ is $H^{\mathsf{ab}}_0$ after killing the classes $[\widehat{h}_j^{-1}t_j^{n_j}]$ for each $j=1,\ldots,k-1$.
$B$ is generated by $S'_{m-n},\ldots,S'_{M+n-1},t'_1,\ldots,$ $t'_{k-1}$
where $t'_j$ is the class $[t_j]$.
Using these generators, $B$ has the following $(2n+k)\times(2n+k+M-m-1)$ presentation matrix.
\[
\left(
\begin{array}{ccccccccccccc}
a_m & a_{m+1} & \cdots & a_{M-1} & a_M \\
& \ddots & & \ddots & & \ddots \\
& & a_m & a_{m+1} & \cdots & a_{M-1} & a_M \\
& & 0 & \multicolumn{4}{c}{\xleftarrow{\hspace{0.75cm}} [\widehat{V}_1]\xrightarrow{\hspace{0.75cm}}} & n_1 \\
& & 0 & \multicolumn{4}{c}{\xleftarrow{\hspace{0.75cm}}[\widehat{V}_2]\xrightarrow{\hspace{0.75cm}}} & -1 & n_2 \\
& & 0 & \multicolumn{4}{c}{\xleftarrow{\hspace{0.75cm}}[\widehat{V}_3]\xrightarrow{\hspace{0.75cm}}} & 0 & -1 & n_3 \\
& & & \multicolumn{4}{c}{\vdots} & & & \ddots & \ddots\\
& & 0 & \multicolumn{4}{c}{\xleftarrow{\hspace{0.6cm}}[\widehat{V}_{k-1}]\xrightarrow{\hspace{0.6cm}}} & 0 & \cdots & 0 & -1 & n_{k-1}\\
& & 0 & \multicolumn{4}{c}{\xleftarrow{\hspace{0.75cm}}[\widehat{V}_k]\xrightarrow{\hspace{0.75cm}}} & 0 & \multicolumn{2}{c}{\cdots} & 0 & -1 \\
\end{array}\right) .
\]
Applying the row operations $\mathbf{row}_{j}+n_{j+1}\mathbf{row}_{j+1}\to \mathbf{row}_j$ for each row $j=k-1,\ldots, 1$ results in the matrix
\[
\left(
\begin{array}{ccccccccccccc}
a_m & a_{m+1} & \cdots & a_{M-1} & M_g \\
& \ddots & & \ddots & & \ddots \\
& & a_m & a_{m+1} & \cdots & a_{M-1} & a_M \\
& & 0 & \multicolumn{4}{c}{\xleftarrow{\hspace{0.75cm}}[U_1]\xrightarrow{\hspace{0.75cm}}} & 0 \\
& & 0 & \multicolumn{4}{c}{\xleftarrow{\hspace{0.75cm}}[U_2]\xrightarrow{\hspace{0.75cm}}} & -1 & 0 \\
& & 0 & \multicolumn{4}{c}{\xleftarrow{\hspace{0.75cm}}[U_3]\xrightarrow{\hspace{0.75cm}}} & 0 & -1 & 0 \\
& & & \multicolumn{4}{c}{\vdots} & & & \ddots & \ddots\\
& & 0 & \multicolumn{4}{c}{\xleftarrow{\hspace{0.6cm}}[U_{k-1}]\xrightarrow{\hspace{0.6cm}}} & 0 & \cdots & 0 & -1 & 0\\
& & 0 & \multicolumn{4}{c}{\xleftarrow{\hspace{0.75cm}}[U_k]\xrightarrow{\hspace{0.75cm}}} & 0 & \multicolumn{2}{c}{\cdots} & 0 & -1 \\
\end{array}\right)
\]
where 
\[
[U_j]=[\widehat{V}_j]+n_1([\widehat{V}_{j+1}]+n_2([\widehat{V}_{j+2}]+\cdots+n_{k-2}([\widehat{V}_{k-1}] +n_{k-1}[\widehat{V}_k])\cdots)).
\]
Eliminating the last $k-1$ rows and columns results in $(2n+1)\times(2n+M-m)$ the presentation matrix $D$
\[
D=\left(
\begin{array}{cccccccccccccc}
a_m & a_{m+1} & \cdots & a_{M-1} & a_M \\
& a_m & a_{m+1} & \cdots & a_{M-1} & a_M \\
& & \ddots & & \ddots & & \ddots \\
& & & a_m & a_{m+1} & \cdots & a_{M-1} & a_M \\
& & & & c_{m} & c_{m+1} & \cdots & c_{M-1} \\
\end{array}\right)
\]
where
\[
[U_1]=c_{m}S'_{m+n} + c_{m+1}S'_{m+n+1} + \cdots + c_{M-1}S'_{M+n-1}.
\]

By Lemma \ref{nestedword}\ref{nwsym}, for some $l$ with $m<l\leq M$, there are integers $b_l,\ldots,b_M$ such that
\begin{equation} \label{awidehat}
[\widehat{A}_k]=\sum_{j=l}^Mb_jS'_{j+n}
\end{equation}
and $|b_{l+j}|=|b_{M-j}|$.
\bigskip{}

\begin{claim}{ 1}
    For each $j=m,\ldots,M-1$,
    \[
        c_j=\left\{
        \begin{array}{ll}
            a_{j} & \text{when }m\leq j<l\\
            a_{j} - (\prod_{s=1}^kn_s)b_{j}& \text{when }l\leq j< M-1
        \end{array}\right. .
    \]
\end{claim}

From the row operations,
\begin{align*}
[U_1]=&[\widehat{V}_1]+n_1([\widehat{V}_2]+n_2([\widehat{V}_3]+\cdots+n_{k-2}([\widehat{V}_{k-1}] +n_{k-1}[\widehat{V}_k])\cdots))\\
=&[\widehat{V}_1]+n_1[\widehat{V}_2]+n_1n_2[\widehat{V}_3]+\cdots+(\prod_{s=1}^{k-2}n_s)[\widehat{V}_{k-1}]+(\prod_{s=1}^{k-1}n_s)[\widehat{V}_k]\\
=&\sum_{j=1}^k(\prod_{s=1}^{j-1}n_s)[\widehat{V}_j].
\end{align*}
By Lemma \ref{nestedword}\ref{nwrec}, $\widehat{V}_j=\widehat{A}_j^{-n_j}\widehat{W}_j^{-1}\widehat{A}_{j-1}\widehat{W}_j$ so $[\widehat{V}_j]=[\widehat{A}_{j-1}]-n_j[\widehat{A}_j]$.
Thus,
\begin{align*}
    \sum_{j=1}^k(\prod_{s=1}^{j-1}n_s)[\widehat{V}_j]=&\sum_{j=1}^k(\prod_{s=1}^{j-1}n_s)([\widehat{A}_{j-1}]-n_j[\widehat{A}_j])\\
    =&\sum_{j=1}^k(\prod_{s=1}^{j-1}n_s)[\widehat{A}_{j-1}]-\sum_{j=1}^k(\prod_{s=1}^jn_s)[\widehat{A}_j]\\
    =&[\widehat{A}_0]-(\prod_{s=1}^kn_s)[\widehat{A}_k].
\end{align*}
Therefore, since $\widehat{A}_0=R_n$,
\begin{equation} \label{sumvirnsumai}
    [U_1]=[R_n]-(\prod_{s=1}^kn_s)[\widehat{A}_k].
\end{equation}
The statement of the claim follows from Proposition \ref{alexanderprop}\ref{alexalln}, (\ref{awidehat}), and (\ref{sumvirnsumai}).
\bigskip{}

By Lemma \ref{helpfullemma}, $C$ is the $\gcd$ of all the $(2n+1)\times (2n+1)$ minors of $D$.
Suppose a prime $d$ divides $C$
so $d$ divides the determinant of every $(2n+1)\times (2n+1)$ submatrix of $D$.
The determinant of the submatrix of $D$ given by the first $2n+1$ columns is $-a_m^{2n+1}$
so $d$ divides $a_m$.
\bigskip{}

\begin{claim}{ 2}
    There is some $(2n+1)\times (2n+1)$ submatrix of $D$ whose determinant is, not divisible by $d$.
\end{claim}
\bigskip{}

By Proposition \ref{alexanderprop}\ref{alexcoef},
the integers $a_m,\ldots,a_M$ are the coeffiecients of the Alexander polynomial.
Since the coefficients of $\Delta_L(t)$ are relatively prime, there is some coefficient that $d$ does not divide.
Let $m+i$ be the minimal index such that $d$ does not divide $a_{m+i}$.
We prove this claim in two cases.
\bigskip{}

\noindent\textit{Case 1.} Suppose $m+i<l$, $d$ divides some $n_s$ with $s\leq k$,
or $d$ divides $b_{j}$ for all $j=l,\ldots,i$.
Then, either $m+i<l$ or $d$ must divide $(\prod_{s=1}^kn_s)b_{j}$ for all $j=l,\ldots,m+i$.
By Claim 1, $d$ divides $c_j$ when $j<m+i$ and $d$ doesn't divide $c_{m+i}$.

Let $E$ be the $(2n+1)\times(2n+1)$ submatrix of $D$ consisting of the $n+1$ consecutive columns starting with the first row which with $a_{m+i}$ (or $c_{m+i}$ if $n=0$) at the top.
Thus, working modulo $d$, we have the following submatrix.
\[
E=\left(
\begin{array}{cccccc}
a_{m+i} & * & * & \cdots & * & *\\
0 & a_{m+i} & * & \cdots & * & *\\
0 & 0 & a_{m+i} & \cdots & * & *\\
\vdots & \vdots & \vdots  & \ddots & \vdots & \vdots\\
0 & 0 & 0 & \cdots & a_{m+i} & *\\
0 & 0 & 0 & \cdots & 0 & c_{m+i}\\
\end{array}\right).
\]
Since $d$ doesn't divide $a_{m+i}$ or $c_{m+i}$, $d$ cannot divide $\det(E)$.
\bigskip{}

\noindent\textit{Case 2.} Suppose that $l\leq m+i$, $d$ does not divide any $n_s$ with $s\leq k$,
and there is some $j\leq m+i$ such that $d$ does not divide $b_j$.

Let $F_1$ be the $(2n+1)\times 2n$ submatrix given by the $n$ consecutive columns with the coefficient $a_{M-i}$.
By Proposition \ref{alexanderprop}\ref{alexcoef}, $a_{m+j}=a_{M-j}$ for all $j=0,M-m$
so $M-i$ is the maximal index such that $d$ divides $a_{M-i}$.
Thus, modulo $d$, $F_1$ has the following form.
\[
F_1=\left(
\begin{array}{ccccc}
a_{M-i} & 0 & 0 & \cdots  & 0\\
* & a_{M-i} & 0 & \cdots & 0\\
* & * & a_{M-i} & \cdots & 0\\
\vdots & \vdots & \vdots & \ddots & \vdots\\
* & * & * & \cdots & a_{M-i}\\
* & * & * & \cdots & *\\
\end{array}\right).
\]
We need to find a column in $D$ with the first $n$ entries divisible by $d$ and the last entry not divisible by $d$.

Let $l+i'$ be the minimal index such that $d$ does not divide $b_{l+i'}$ so $l+i'\leq m+i$.

Since $d$ does not divide $b_{l+i'}$ and $b_{l+i'}=b_{M-i'}$, $d$ does not divide $b_{M-i'}$.
By Lemma \ref{nestedword}\ref{nwvsm}, for all $j$, the coefficient of $S'_{M+n}$ in $[\widehat{V}_j]$ is zero so by (\ref{sumvirnsumai}),
\[
a_M=b_M\prod_{s=1}^kn_s.
\]
Since $a_m=a_M$ and $d$ divides $a_m$, $d$ must also divide $b_M$.
Therefore, $d$ divides $b_l$ so $i'>0$ and $M-i'\leq M-1$.

Since $M-i'\leq M-1$ there is some column $F_2$ which ends with $c_{M-i'}$.
Every other entry in $F_2$ is 0 or $a_j$ for some $j>M-i'$.
Since $l+i'\leq m+i$ and $m<l$,
\[
    0 < l - m \leq i-i'
\]
so $M-i<M-i'$.
Thus, by Claim 1, $d$ does not divide $c_{M-i'}$, and
for all $j>M-i'$, $d$ divides $a_j$.

Combine $F_1$ and $F_2$ to get an $(2n+1)\times(2n+1)$ submatrix $F$ of $D$.
Working modulo $d$, we have the submatrix.
\[
F=\left(
\begin{array}{cccccc}
a_{M-i} & 0 & 0 & \cdots & 0 & 0\\
* & a_{M-i} & 0 & \cdots & 0 & 0\\
* & * & a_{M-i} & \cdots & 0 & 0\\
\vdots & \vdots & \vdots & \ddots & \vdots & \vdots\\
* & * & * & \cdots & a_{M-i} & 0\\
* & * & * & \cdots & * & c_{M-i'}\\
\end{array}\right).
\]
Since $d$ doesn't divide $a_{M-i}$ or $c_{M-i'}$, $d$ cannot divide $\det(F)$.

In conclusion, there are no primes which divide every determinant of $(2n+1)\times(2n+1)$ submatrices of $D$
so $C=1$.
Thus, $B\cong\Z^{M-m-1}$, and $H_k$ is parafree of rank $M-m$.
By induction, $H_N$ is parafree of rank $M-m$.

By a similar induction argument, $H_N,\ldots,H_{2N}$ are also parafree of rank $M-m$.
Therefore, $Y_{n+1}\cong H_{2N}$ is parafree of rank $M-m$
so by induction $Y_n$ is parafree of rank $M-m$ for each non-negative integer $n$.
\bigskip{}

For \ref{mainquotr}, consider the group $Y_{n+1}/Y_n[Y_{n+1},Y_{n+1}]$ which is an abelian group with the following presentation.
\[
    \frac{Y_{n+1}}{Y_n[Y_{n+1},Y_{n+1}]}\cong \la S'_{m-n-1},\ldots,S'_{M+n}\; | \;[R_{-n-1}],\ldots,[R_{n}],S'_{m-n},\ldots,S'_{M+n-1}\ra 
\]

By Proposition \ref{alexanderprop},
\[
	[R_j]=\underline{a}_gS'_{M+j}+\underline{a}_{g-1}S'_{M-1+j}+\cdots+\underline{a}_{g-1}S'_{m+1+j}+\underline{a}_gS'_{m+j}.
\]
After eliminating the generators $S'_{m-n},\ldots,S'_{M+n-1}$,
we have that
\[
    \frac{Y_{n+1}}{Y_n[Y_{n+1},Y_{n+1}]}\cong \la S'_{m-n-1},S'_{M+n}\; | \;\underline{a}_gS'_{M-n-1},
    \underline{a}_gS'_{m+n}\ra
\]
so
\[
\Big|Y_{n+1}/Y_n[Y_{n+1},Y_{n+1}]\Big|=\Big|\frac{\Z}{\underline{a}_g\Z}\oplus\frac{\Z}{\underline{a}_g\Z}\Big|=\underline{a}_g^2.
\]
\end{proof}

\section{Cycle Graphs} \label{CycleGraphs}

Explicitly, Lemma \ref{nestedword} is about nested patterns of repeating words in the relator $R_0$.
However, this pattern is inherited from patterns in the sequences of $\epsilon_i$'s and $\sigma_i$'s defined in (\ref{epsilons}) and (\ref{sigmas}).
In the spirit of Hirasawa and Murasugi \cite{HirMur07}, graphs are used in order to gain intuition about how the sequences of $\epsilon_i$'s and $\sigma_i$'s behave;
however, the construction here slightly differs from the one Hirasawa and Murasugi used.

\subsection{Incremental Paths and Cycles}

A \emph{graded directed graph} is a directed graph $\Gamma$ with map $\gr:V(\Gamma)\to\Z$ called the \emph{grading}.
Here $V(\Gamma)$ denotes the set of vertices of $\Gamma$.
Two graded directed graphs $\Gamma$ and $\Gamma'$ are \emph{isomorphic} if there is a directed graph isomorphism $f:\Gamma\to\Gamma'$ such that for every vertex $P$ in $\Gamma$, $\gr(f(P))=\gr(P)$.
$\Gamma$ and $\Gamma'$ are called \emph{relatively isomorphic} if there is a directed graph isomorphism $f:\Gamma\to\Gamma'$ and an integer $k$ such that for every vertex $P$ in $\Gamma$, $\gr(f(P))=\gr(P)+k$.

An \emph{incremental path} is a graded directed path graph $\Gamma$ where the gradings of adjacent vertices differ by $\pm1$.
Similarly, an \emph{incremental cycle} is a graded directed cycle graph $\Gamma$ where the gradings of adjacent vertices differ by $\pm1$.

Let $\Gamma$ and $\Gamma'$ be two incremental paths
in which the grading of the last vertex in $\Gamma$ is equal to the grading of the first vertex in $\Gamma'$.
Define the \emph{concatenation} of $\Gamma$ and $\Gamma'$, denoted $\Gamma*\Gamma'$, to be the graded directed graph obtained by identifying the last vertex in $\Gamma$ with the first vertex in $\Gamma'$ (see Figure \ref{figconcat}).

\begin{figure}[t]
\includegraphics{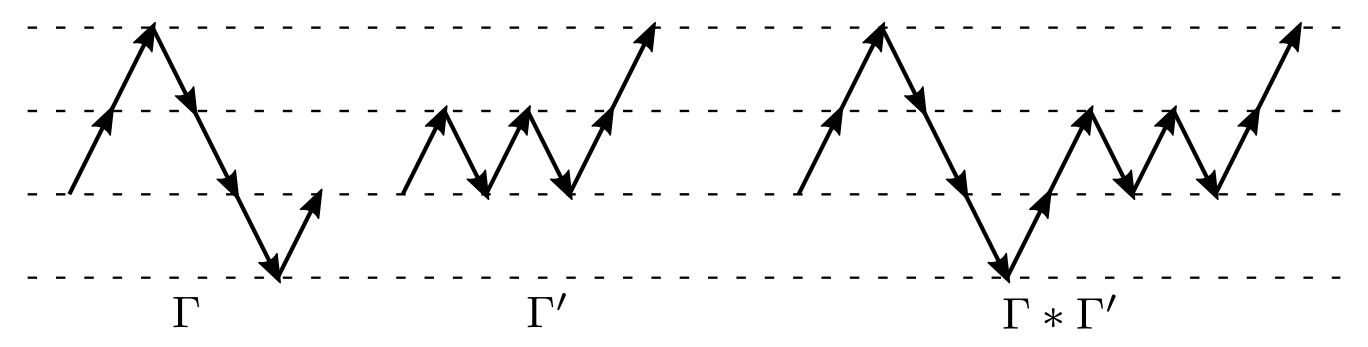}
\caption{The concatenation of $\Gamma$ and $\Gamma'$}
\label{figconcat}
\end{figure}

If the grading of the first and last vertices in $\Gamma$ are the same,
$\Gamma$ is called \emph{closable} and the \emph{closure of $\Gamma$}, $\cl(\Gamma)$,is defined to be the incremental cycle obtained by identifying the first and last vertex in $\Gamma$.

\subsection{Cycle Graphs of Co-prime Pairs} \label{cycledef}

Ultimately, Lemma \ref{nestedword} is a statement about the sequences of $\epsilon_i$'s and $\sigma_i$'s for co-prime pairs of integers.
As computed in Proposition \ref{subscriptsprop}, the $i$th $S$-generator in $R_0$ is determined by the values of $\sigma_{2i-1}$ and $\sigma_{2i}$.
Here we construct a graph to analyze the sequences of $\epsilon_i$'s and $\sigma_i$'s.

\begin{figure}[b]
\includegraphics{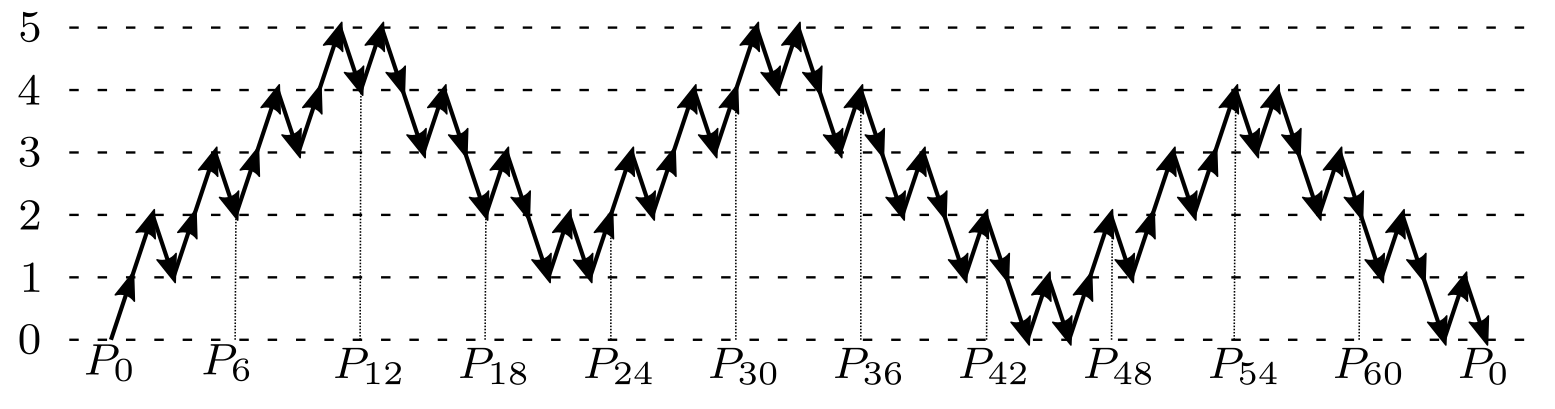}
\caption{$\oGamma(33,23)$}
\label{gamma3323}
\end{figure}

Let $(p,q)$ denote a co-prime pair of integers $p$ and $q$ such that $p$ is positive, $q$ is odd and $p>|q|>0$.
Define the sequences $\epsilon_i$ and $\sigma_i$ as in (\ref{epsilons}) and (\ref{sigmas}) for each integer $i$.
Define the incremental path $\Gamma(p,q)$ as follows.
The vertex set of $\Gamma(p,q)$ is $\{P_0,\ldots,P_{2p}\}$, and
the edge set of $\Gamma(p,q)$ is 
$$E(\Gamma(p,q))=\{(P_0,P_1),(P_1,P_2),\ldots,(P_{2p-1},P_{2p})\}.$$
The grading of each vertex is defined by $\gr(P_i)=\sigma_i$.
$\Gamma(p,q)$ is always closable, and
the \emph{cycle graph of $p$ and $q$}, $\oGamma(p,q)$ is defined to be $\cl(\Gamma(p,q))$.
When studying $\oGamma(p,q)$, it's convenient to think of its vertices $\{P_0,\ldots,P_{2p-1}\}$ being indexed by elements of $\Z/(2p\Z)$. See Figure \ref{gamma3323} for an example.

\begin{prop} \label{negrelsio}
Let $(p,q)$ be a co-prime pair.
The cycle graphs $\oGamma(p,q)$ and $\oGamma(p,-q)$ are relatively isomorphic.
\end{prop}

\begin{proof}
Let $\{\epsilon_i\}_{i\in\Z}$ be the sequence of signs of $(p,q)$ defined in (\ref{epsilons}).
For each integer $i$, define
\[
	\varepsilon_i=(-1)^{\lfloor \frac{-iq}{p}\rfloor}
\]
which is the sequence of signs of $(p,-q)$.
Let $q'$ be the unique integer such that $0<q'<2p$ and $q'q\cong p-1$ modulo $2p$.
Then
\begin{equation} \label{negepsilons}
    \varepsilon_i=\epsilon_{i+q'}
\end{equation}
for every $i$ in $\Z/(2p\Z)$.
For each integer $i=0,\ldots,2p$, define
\[
	\varsigma_i:=
	\sum_{j=0}^{i-1}\varepsilon_i ,
\]
which are the gradings of the vertices of $\oGamma(p,-q)$.
By (\ref{negepsilons}),
\[
    \varsigma_i=\sigma_{i+q'}-\sigma_{q'}
\]
for every positive integer $i$.
Since the $\sigma_i$'s are the gradings of the vertices of $\oGamma(p,q)$,
it follows that $\oGamma(p,q)$ and $\oGamma(p,-q)$ are relatively isomorphic.
\end{proof}

\subsection{Structure of $\oGamma(p,q)$}

Given an incremental cycle $\Gamma$, a \emph{positive(negative) $k$-segment} is a set of $k$ consecutive positive(negative) increment edges in $\Gamma$ which are followed and preceded by negative(positive) increment edges; see Figure \ref{sb:segment}.
For each co-prime integer pair $(p,q)$,
$\oGamma(p,q)$ is the closure of the concatenation of segments of alternating sign as follows.
\[
    \oGamma(p,q)=\cl(\Lambda_0*\Lambda_1*\cdots*\Lambda_{n-1})
\]
As a convention, let $\Lambda_0$ denote the segment in $\oGamma(p,q)$ containing the edge which corresponds to $\epsilon_0$.

\begin{figure}[b]
    \begin{subfigure}[t]{4cm}
        \centering
        \includegraphics{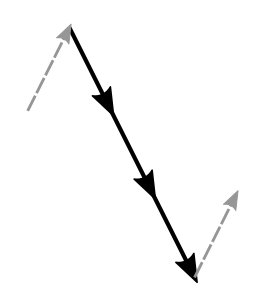}
        \caption{A negative 3-segment}
        \label{sb:segment}
    \end{subfigure}
    \hspace{1cm}
    \begin{subfigure}[t]{4cm}
        \centering
        \includegraphics{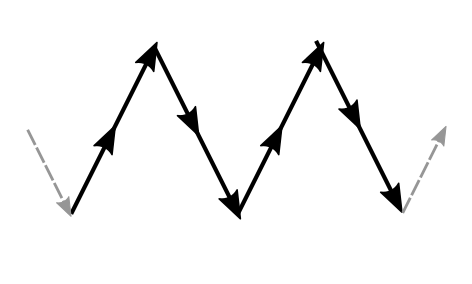}
        \caption{A 2-block of length 4}
        \label{sb:block}
    \end{subfigure}
    \caption{}
    \label{segsandblocks}
\end{figure}

Proposition \ref{segsizeprop} and Proposition \ref{blocksizeprop} are analogs of the properties proved in section 6 of Hirasawa and Murasugi's paper \cite{HirMur07}.

\begin{prop} \label{segsizeprop}
    Let $(p,q)$ be a co-prime pair with $q>0$.
    Let $P_0,\ldots,P_{2p-1}$ be the vertices of $\oGamma(p,q)$ as defined in section \ref{cycledef}, and let
    \[
        \oGamma(p,q)=\cl(\Lambda_0*\Lambda_1*\cdots*\Lambda_{n-1})
    \]
    where $\Lambda_0,\ldots,\Lambda_{n-1}$ are segments.
    Also, let $\kappa$ and $\xi$ be integers such that $p=\kappa q+\xi$ and $0<\xi<q$.
    \begin{enumerate}[label=(\alph*)]
        \item The number of segments $n$ in $\oGamma(p,q)$ is equal to $2q$. \label{segnum}
        \item $P_i$ is at the beginning of a segment precisely when $iq\modop p <q$. \label{segstart}
        \item When $\xi\leq iq\modop p <q$, $P_i$ is at the beginning of a $\kappa$-segment, and
        when $iq\modop p <\xi$, $P_i$ is at the beginning of a $(\kappa+1)$-segment. \label{segtype}
        \item $\Lambda_0$ is a $(\kappa+1)$-segment. \label{segbegin}
        \item There are a total of $2\xi$, $(\kappa+1)$-segments in $\oGamma(p,q)$. \label{segbig}
    \end{enumerate}
\end{prop}

\begin{proof}
    For \ref{segnum}, notice that the segments of $\oGamma(p,q)$ correspond to the number of distinct floored quotients $\lf\frac{iq}{p}\rf$ there are when $i=0,\ldots,2p-1$.
    Since $p>q$, these quotients range from $0$ to $2q-1$ without skipping so there are exactly $2q$ segments.
    
    A segment begins when
    \[
        \lf\frac{(i-1)q}{p}\rf\neq\lf\frac{iq}{p}\rf,
    \]
    which happens when $(iq\modop p)< q$, proving \ref{segstart}.
    
    For \ref{segtype}, suppose $P_i$ is the beginning of a $k$-segment.
    $k$ is the smallest positive integer such that 
    \[
        \lf\frac{iq}{p}\rf\neq\lf\frac{(i+k)q}{p}\rf .
    \]
    so
    \[
        (iq\modop p) + (k - 1)q< p
    \]
    and
    \[
        (iq\modop p) + kq \geq p.
    \]
    When $\xi\leq (iq\modop p) < q$, $k=\kappa$.
    Likewise, when $(iq\modop p)<\xi$, $k=\kappa+1$.
    
    Parts \ref{segbegin} and \ref{segbig} immediately follow from \ref{segtype}.
\end{proof}

A \emph{$k$-block of length $l$} in $\oGamma(p,q)$ is a sequence of $l$ consecutive $k$-segments that is not proceeded or followed by a $k$-segment; see Figure \ref{sb:block}.
A $k$-block of length $1$ is called an \emph{isolated block}.

\begin{prop} \label{blocksizeprop}
    Let $(p,q)$ be a co-prime pair with $q>0$, and
    let $P_0,\ldots,P_{2p-1}$ be the vertices of $\oGamma(p,q)$ as defined in section \ref{cycledef}.
    Define $\kappa$, $\xi$, $\kappa'$, and $\xi'$ be integers such that
    \begin{equation} \label{eucdivpq}
        p=\kappa q+\xi\text{ with }0<\xi<q
    \end{equation}
    and
    \begin{equation}\label{eucdivqr}
        q=\kappa'\xi+\xi'\text{ with }0<\xi'<\xi.
    \end{equation}
    \begin{enumerate}[label=(\alph*)]
        \item All of the $\kappa$-blocks in $\oGamma(p,q)$ have length $\kappa'$ or $\kappa'-1$. \label{blocksizes}
        \item If $P_j$ is the start of a $\kappa$-block, then when
        \[
            q-\xi'\leq jq\modop p < q ,
        \]
        the $\kappa$-blocks has length $\kappa'$
        and when
        \[
            q-\xi\leq jq\modop p < q-\xi' ,
        \]
        the $\kappa$-blocks has length $\kappa'-1$. \label{blockstart}
        \item If $\kappa'\geq 2$ then all the $(\kappa+1)$-blocks in $\oGamma(p,q)$ are isolated. \label{blockkg2}
        \item If $\kappa'=1$ then all the $\kappa$-blocks in $\oGamma(p,q)$ are isolated. \label{blockke1}
    \end{enumerate}
\end{prop}

\begin{proof}
Similar to the proof of Proposition \ref{segsizeprop},
this proposition is just matter of determining when $\kappa$-blocks and $(\kappa+1)$-blocks appear is $\oGamma(p,q)$.

Suppose $P_i$ is the beginning of a $(\kappa+1)$-segment.
The next segment begins at $P_j$ where $j=i+\kappa+1$, and
by (\ref{eucdivpq}),
\begin{align*}
    jq\modop p=&((i+\kappa+1)q)\modop p\\
    =&(iq+\kappa q +q)\modop p\\
    =&(iq+p-\xi+q)\modop p\\
    =&((iq\modop p)+q-\xi)\modop p.
\end{align*}
Since $P_i$ is the beginning of a $(\kappa+1)$-segment, $(iq\modop p)<\xi$ by Proposition \ref{segsizeprop}\ref{segtype} so
\begin{equation} \label{jqmodpbound}
    q-\xi \leq (iq\modop p)+q-\xi<q<p .
\end{equation}
Thus,
\begin{equation} \label{jqmodp}
    jq\modop p=(iq\modop p)+q-\xi.
\end{equation}

For \ref{blocksizes} and \ref{blockstart}, suppose a $\kappa$-block starts at vertex $P_j$.
The length of the $\kappa$-block starting at $P_j$ is the smallest positive integer $n$,
such that $P_{s(n)}$ is the start of a $(\kappa+1)$-block where $s(k)=j+k\kappa$ so
$n$ is the smallest positive integer such that
\[
    0\leq s(n)q\modop p\xi < \xi.
\]

By (\ref{eucdivpq}),
\begin{align*}
    s(k)q\modop p=&(j+k\kappa)q\modop p\\
    =&(jq+k\kappa q)\modop p\\
    =&(jq+kp-k\xi)\modop p\\
    =&((jq\modop p)-k\xi)\modop p .
\end{align*}

By (\ref{jqmodpbound}) and (\ref{jqmodp}), since $P_j$ is the beginning of a $\kappa$-segment,
\[
    q-\xi\leq jq\modop p < q.
\]

We compute the length $n$ for each of the two cases $q-\xi\leq (jq\modop p) < q-\xi' $ and $q-\xi'\leq (jq\modop p) < q $.

Suppose that
\begin{equation} \label{pj1}
    q-\xi'\leq jq\modop p < q.
\end{equation}
By (\ref{eucdivqr}),
\[
    ((jq\modop p)-\kappa'\xi=((jq\modop p)-q+\xi'
\]
and
\[
0\leq ((jq\modop p)-q+\xi'<\xi'
\]
so
\[
    0\leq s(\kappa')q\modop p<\xi'<\xi .
\]
Thus, $n\leq \kappa'$.

Suppose $k\leq\kappa'-1$.
By (\ref{eucdivqr}) and (\ref{pj1}),
\begin{align*}
    \xi \leq & ((jq\modop p)-q+\xi'+\xi\\
    =&((jq\modop p)-\kappa'\xi+\xi\\
    =&((jq\modop p)-(\kappa'-1)\xi
\end{align*}
so
\[
    \xi\leq ((jq\modop p)-k\xi<q .
\]
Thus,
\[
    \xi\leq s(k)q\modop p<q
\]
so $n\geq \kappa'$.
Therefore, $n=\kappa'$.

Suppose 
\[
    q-\xi\leq (jq\modop p)< q-\xi',
\]
By (\ref{eucdivqr}),
\[
    ((jq\modop p)-(\kappa'-1)\xi=((jq\modop p)-q+\xi'+\xi
\]
and
\[
0\leq \xi'\leq ((jq\modop p)-q+\xi'+\xi<\xi
\]
so
\[
    0\leq s(\kappa'-1)q\modop p<\xi .
\]
Thus, $n\leq \kappa'-1$.

Suppose $k\leq\kappa'-2$.
By (\ref{eucdivqr}) and (\ref{pj1}),
\begin{align*}
    \xi \leq & ((jq\modop p)-q+\xi'+2\xi\\
    =&((jq\modop p)-(\kappa'-2)\xi
\end{align*}
so
\[
    \xi\leq ((jq\modop p)-k\xi<q .
\]
Thus,
\[
    \xi\leq s(k)q\modop p<q
\]
so $n\geq \kappa'-1$.
Therefore,
$n=\kappa'-1$.
Thus, all of the $\kappa$-blocks have length $\kappa'$ or $\kappa'-1$.

For \ref{blockkg2}, suppose that $\kappa'\geq 2$.
By (\ref{eucdivqr}),
\[
    q-\xi=(\kappa'-1)\xi+\xi' ,
\]
and since $\kappa'\geq2$,
\[
	\xi\leq\xi+\xi'\leq q-\xi
\]
so by (\ref{jqmodpbound}),
\[
    \xi\leq (iq\modop p)+q-\xi <q.
\]
Thus, by (\ref{jqmodp}),
\[
    \xi\leq jq\modop p <q.
\]
By Proposition \ref{segsizeprop}\ref{segtype}, $P_j$ must be the beginning of a $\kappa$-segment so $(\kappa+1)$-segments cannot occur consecutively.
Therefore, $(\kappa+1)$-blocks are isolated.

Statement \ref{blockke1} follows immediately from $\ref{blocksizes}$.
\end{proof}

\subsection{Reducing Cycle Graphs}

Let $(p,q)$ be a co-prime pair with $q>0$.
Let $\kappa$, $\xi$, $\kappa'$ and $\xi'$ be defined as in Proposition \ref{blocksizeprop},
and let the decomposition of $\oGamma(p,q)$ be
\begin{equation} \label{segmentdecomp}
\oGamma(p,q)=\cl(\Lambda_0*\cdots*\Lambda_{2q-1}).
\end{equation}
Define a reduction of $\oGamma(p,q)$, denoted $R(\oGamma)(p,q)$, by
\begin{enumerate}
\item eliminating all $\kappa$-segments,
\item replacing each $(\kappa+1)$-segment with a positive or negative increment according to the sign of the segment, and
\item setting the grading of the vertex preceding the edge corresponding to $\Lambda_0$ equal to zero.
\end{enumerate}
For an example, see Figure \ref{redgamma3323}.

\begin{figure}[t]
    \begin{subfigure}[t]{10.5cm}
        \centering
        \includegraphics{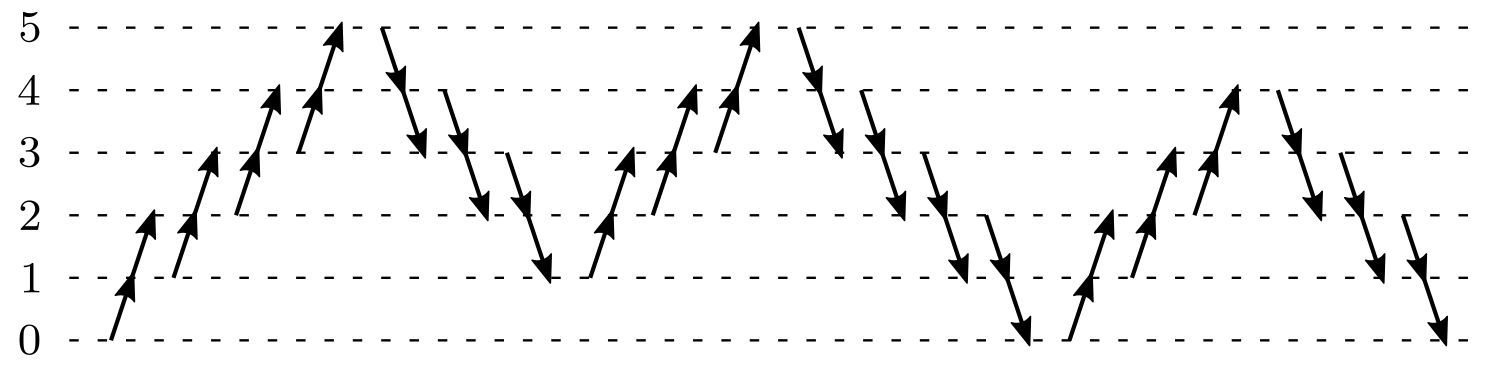}
        \caption{All the 1-segments have been removed from $\oGamma(33,23)$; see Figure     \ref{gamma3323}.}
        \label{rd:a}
    \end{subfigure}
    \begin{subfigure}[t]{10.5cm}
        \centering
        \includegraphics{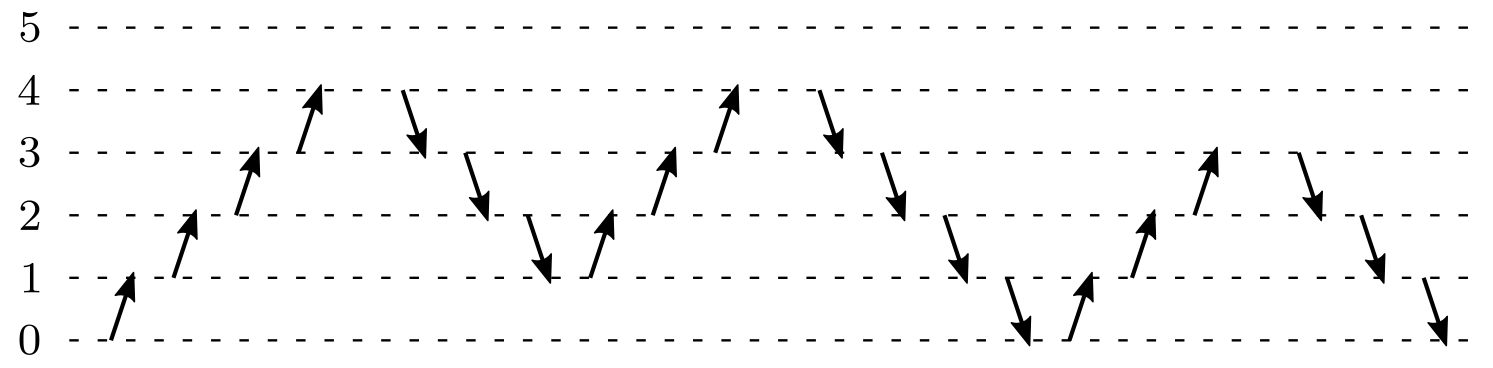}
        \caption{The 2-segments have been replaced by edges.}
        \label{rd:b}
    \end{subfigure}
    \begin{subfigure}[t]{10.5cm}
        \centering
        \includegraphics{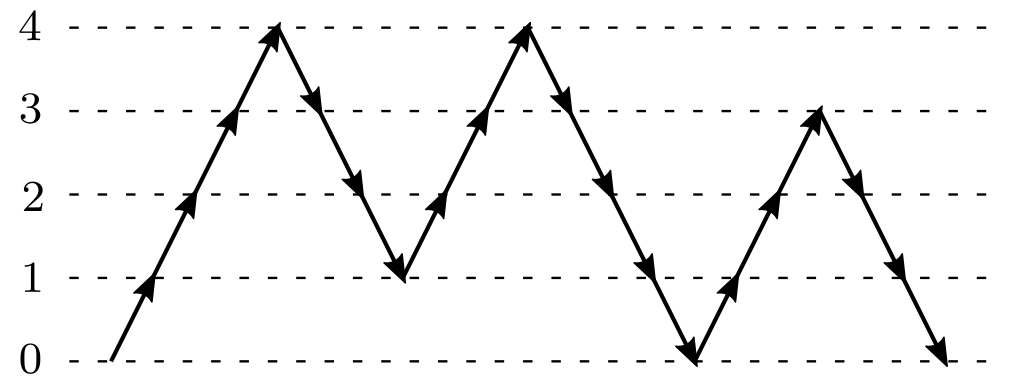}
        \caption{The resulting graph $R(\oGamma)(33,23)$ is isomorphic to $\oGamma(10,3)$.}
        \label{rd:c}
    \end{subfigure}
    \caption{Reducing $\oGamma(33,23)$}
    \label{redgamma3323}
\end{figure}

\begin{lem}\label{redpair}
    Let $(p,q)$ be a co-prime pair with $q>1$ and $\xi>1$.
    Define $p^*$ to be $\xi$, and define $q^*$ as follows.
    \[
        q^*=\left\{
        \begin{array}{ll}
            \xi' & \text{when }\kappa'\text{ is even}\\
            \xi'-\xi & \text{when }\kappa'\text{ is odd}
        \end{array}
        \right.
    \]
    \begin{enumerate}[label=(\alph*)]
        \item $p^*$ is always positive and $q^*$ is always odd. \label{rednew}
        \item $R(\oGamma)(p,q)$ is isomorphic to                    $\oGamma(p^*,q^*)$. \label{rediso}
    \end{enumerate}
\end{lem}

\begin{proof}
For \ref{rednew},
we see that $\xi>0$ since $p$ and $q$ are co-prime.
Also, notice that $q$ is odd and
\[
	\xi'=q-\kappa'\xi.
\]
If $\kappa'$ is even then $q^*=\xi'$ is odd.
If $\kappa'$ is odd then $\xi'$ and $\xi$ must have opposite parities so $q^*=\xi'-\xi$ is odd.
\bigskip{}

For \ref{rediso}, consider $\oGamma(p,q)$.
By Proposition \ref{segsizeprop}\ref{segbig}, we know that $\oGamma(p,q)$ has $2\xi$ $(\kappa+1)$-segments so
$R(\oGamma)(p,q)$ has $2\xi$ edges and $2\xi$ vertices.
Let $\{Q_0,\ldots,Q_{2\xi-1}\}$ be the vertex set of $R(\oGamma)(p,q)$,
and $\{P^*_0,\ldots,P^*_{2\xi-1}\}$ be the vertex set of $\oGamma(p^*,q^*)$.
Since $R(\oGamma)(p,q)$ and $\oGamma(p^*,q^*)$ are cycle graphs with the same number of vertices, there is a unique ungraded directed graph isomorphism between them by mapping $Q_i\mapsto P^*_i$.
Since $\gr(Q_0)$ and $\gr(P^*_0)$ are both 0 by definition, it only remains to show
\[
	\gr(Q_{i+1})-\gr(Q_i)=\gr(P^*_{i+1})-\gr(P^*_i)
\]
for each $i=0,\ldots,2\xi-1$.

For $i=0,\ldots,2\xi-1$, define
$$\varepsilon_i:=\gr(Q_{i+1})-\gr(Q_i)$$
and
$$\eta_i:=(-1)^{\lfloor \frac{i\xi'}{\xi}\rfloor}.$$
If $q^*=\xi'$, then
\[
    \gr(P^*_{i+1})-\gr(P^*_i)=\eta_i,
\]
and if $q^*=\xi'-\xi$, then
\[
    \gr(P^*_{i+1})-\gr(P^*_i)=(-1)^{\lfloor\frac{i(\xi'-\xi)}{\xi} \rfloor}=(-1)^i\eta_i.
\]

\begin{figure}[t]
\includegraphics{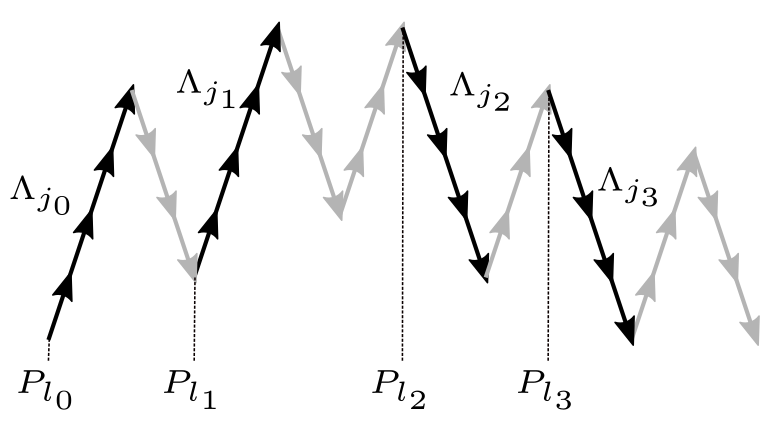}
\caption{The $(\kappa+1)$-segments of $\oGamma(17,5)$. The indices of the segments are $j_0=0$, $j_1=2$, $j_2=5$, and $j_3=7$. The indices of the vertices at the beginning of each $(\kappa+1)$-segment are $l_0=0$, $l_1=7$, $l_2=17$, and $l_3=24$.}
\label{kappap1segs}
\end{figure}

Let $j_0,\ldots,j_{2\xi-1}$ be the indices in ascending order of the $(\kappa+1)$-segments in the decomposition in (\ref{segmentdecomp}), and
let $l_i$ be the index of the vertex in $\oGamma(p,q)$ at the beginning of $\Lambda_{j_i}$; see Figure \ref{kappap1segs}.
By definition of $R(\oGamma)(p,q)$, $\varepsilon_i$ is positive precisely when $\Lambda_{j_i}$ is a positive segment.
Thus, $\varepsilon_{i+1}=\varepsilon_i$ when $\Lambda_{j_i}$ and $\Lambda_{j_{i+1}}$ are separated by an even number of $\kappa$-segments, and
$\varepsilon_{i+1}=-\varepsilon_i$ when $\Lambda_{j_i}$ and $\Lambda_{j_{i+1}}$ are separated by an odd number of $\kappa$-segments.
The desired result will follow from three claims.
\bigskip{}

\begin{claim}{ 1}
    Whenever $0\leq (i\xi'\modop\xi)<\xi-\xi'$,
    \[
        \eta_{i+1}=\eta_i,
    \]
    and whenever $(i\xi'\modop\xi)\geq\xi-\xi'$, 
    \[
        \eta_{i+1}=-\eta_i.
    \]
\end{claim}
\bigskip{}

When $0\leq (i\xi'\modop\xi)<\xi-\xi'$,
there are integers $s$ and $t$ with
\[
    i\xi'= s\xi+t \text{ and } 0\leq t<\xi-\xi'
\]
so
\[
    s\xi\leq (i+1)\xi'=s\xi+t+\xi'<(s+1)\xi .
\]
Thus,
\[
    \eta_{i+1}=(-1)^{s}=\eta_i.
\]

When $(i\xi'\modop\xi)\geq\xi-\xi'$,
there are integers $s$ and $t$ with
\[
    i\xi'= s\xi+t \text{ and } \xi-\xi'\leq t<\xi
\]
so
\[
    (s+1)\xi\leq (i+1)\xi'=s\xi+t+\xi'<(s+1)\xi+\xi'<(s+2)\xi .
\]
Thus,
\[
    \eta_{i+1}=(-1)^{s+1}=-\eta_i.
\]

\begin{claim}{ 2}
    The segments $\Lambda_{j_i}$ and $\Lambda_{j_{i+1}}$ are separated by a $\kappa$-block of length $\kappa'$ when
    \[
        \xi-\xi'\leq (l_iq\modop p)<\xi
    \]
    and a $\kappa$-block of length $\kappa'-1$ (possibly zero) when
    \[
        0\leq (l_iq\modop p)<\xi-\xi'.
    \]
\end{claim}

By Proposition \ref{blocksizeprop}\ref{blockstart}, every $\kappa$-block begins at a vertex $P_l$ where 
\[
    q-\xi\leq (lq\modop p)<q .
\]
The length of the block is $\kappa'$ when
\begin{equation}\label{bigblock}
    q-\xi'\leq (lq\modop p)<q,
\end{equation}
and the length is $\kappa'-1$ when
\begin{equation}\label{smallblock}
    q-\xi \leq (lq\modop p)<q-\xi'.
\end{equation}

The vertex at the end of the segment $\Lambda_{j_i}$ is the same as the vertex at the beginning the segment $\Lambda_{j_i+1}$ so
$\Lambda_{j_i+1}$ begins at the vertex with index $l':=l_i+\kappa+1$.
By Proposition \ref{segsizeprop}\ref{segstart},
\[
    0\leq l_iq\modop p+q-\xi<q<p
\]
so
\begin{align*}
    l'q\modop p=&(l_i+\kappa+1)q\modop p\\
    =&(l_iq\modop p+q-\xi)\modop p \\
    =&l_iq\modop p+q-\xi .
\end{align*}
By (\ref{bigblock}), $\Lambda_{j_i}$ and $\Lambda_{j_{i+1}}$ are separated by a $\kappa$-block of length $\kappa'$ when
\[
    q-\xi'\leq(l'q\modop p)<q
\]
so
\[
    \xi-\xi'\leq (l_iq\modop p)<\xi.
\]
By (\ref{smallblock}),
$\kappa$-block of length $\kappa'-1$ when
\[
    q-\xi\leq(l'q\modop p)<q-\xi'
\]
so
\[
    0\leq (l_iq\modop p)<\xi-\xi'.
\]

\begin{claim}{ 3}
    For each $i=0,\ldots,2\xi-1$
    \[
    l_iq\modop p=i\xi'\modop\xi.
    \]
\end{claim}

$P_{l_i}$ and $P_{l_{i+1}}$ are separated by a $(\kappa+1)$-segment and a $\kappa$-block.
Therefore, when the length of the $\kappa$-block is $\kappa'$,
\[
    l_{i+1}=l_i+(\kappa+1)+\kappa'\kappa
\]
so
\begin{align*}
    l_{i+1}q\modop p=&(l_iq+\kappa q+q+\kappa'\kappa q)\modop p\\
    =&(l_iq\modop p+\xi'-\xi)\modop p
\end{align*}
where last equality follows from (\ref{eucdivpq}) and (\ref{eucdivqr}).
By Claim 2,
\[
    0\leq l_iq\modop p+\xi'-\xi<\xi'<p .
\]
Therefore, 
\begin{equation}\label{clm3a}
    l_{i+1}q\modop p=l_iq\modop p+\xi'-\xi .
\end{equation}

When the length of the $\kappa$-block is $\kappa'-1$,
\[
    l_{i+1}=l_i+(\kappa + 1) + (\kappa'-1)\kappa=l_i+1+\kappa'\kappa
\]
so
\begin{align*}
    l_{i+1}q\modop p=&(l_iq+q+\kappa'\kappa q)\modop p\\
    =&(l_iq\modop p+\xi')\modop p .
\end{align*}
By Claim 2,
\[
    0<\xi'\leq l_iq\modop p+\xi'<\xi<p .
\]
Therefore, 
\begin{equation}\label{clm3b}
    l_{i+1}q\modop p=l_iq\modop p+\xi'.
\end{equation}

In either the case (\ref{clm3a}) or (\ref{clm3b}),
\[
    l_{i+1}q\modop p=(l_iq\modop p +\xi')\modop \xi
\]
so since $l_0=0$,
\[
    l_iq\modop p=i\xi'\modop\xi
\]
for each $i=0,\ldots, 2\xi-1$ by induction.
This completes the proof of the claim.
\bigskip{}

Suppose $\kappa'$ is even.
When $\Lambda_{i+1}$ and $\Lambda_i$ are separated by a $\kappa$-block of length $\kappa'-1$,
$\Lambda_{i+1}$ and $\Lambda_i$ have the same sign so
\[
    \varepsilon_{i+1}=\varepsilon_i.
\]
By the three claims,
\[
    0\leq (i\xi'\modop\xi)<\xi-\xi'
\]
so
\[
    \eta_{i+1}=\eta_i.
\]

When $\Lambda_{i+1}$ and $\Lambda_i$ are separated by a $\kappa$-block of length $\kappa'$,
$\Lambda_{i+1}$ and $\Lambda_i$ have opposite signs so
\[
    \varepsilon_{i+1}=-\varepsilon_i.
\]
By the three claims,
\[
    (i\xi'\modop\xi)\geq\xi-\xi'
\]
so
\[
    \eta_{i+1}=-\eta_i.
\]

Since $\varepsilon_0=\eta_0=1$, for every $i=0,\ldots,2\xi-1$,
\[
    \varepsilon_i=\eta_i
\]
so when $q^*=\xi'$,
\[
    \gr(P^*_{i+1})-\gr(P^*_i)=\eta_i=\varepsilon_i=\gr(Q_{i+1})-\gr(Q_i).
\]

Suppose $\kappa'$ is odd.
When $\Lambda_{i+1}$ and $\Lambda_i$ are separated by a $\kappa$-block of length $\kappa'$, then $\varepsilon_{i+1}=\varepsilon_i$.
When $\Lambda_{i+1}$ and $\Lambda_i$ are separated by a $\kappa$-block of length $\kappa'-1$,
then $\varepsilon_{i+1}=-\varepsilon_i$.

Thus, by the claims, $\varepsilon_{i+1}=\varepsilon_i$ when $\eta_{i+1}=-\eta_i$, and
$\varepsilon_{i+1}=-\varepsilon_i$ when $\eta_{i+1}=\eta_i$.
Again, $\varepsilon_0=\eta_0=1$.
Therefore, for every $i=0,\ldots,2\xi-1$, 
\[
    \varepsilon_i=(-1)^i\eta_i
\]
so when $q^*=\xi'-\xi$, then
\[
    \gr(P^*_{i+1})-\gr(P^*_i)=(-1)^i\eta_i=\varepsilon_i=\gr(Q_{i+1})-\gr(Q_i).
\]
\end{proof}

\begin{example}
    Consider the co-prime pair $(33,23)$.
    $R(\oGamma)(33,23)$ is isomorphic to $\Gamma(10,3)$ (see Figure \ref{redgamma3323}).
\end{example}

\subsection{Expanding Cycle Graphs}

We can also reverse the reduction process $R$.
Let $\Gamma$ be an incremental path with vertices $P_0,\ldots,P_n$ indexed such that $(P_i,P_{i+1})$ is an edge in $\Gamma$ for each $i=0,\ldots,n-1$.
Let $s$ and $b$ be positive integers,
and let $e=\pm1$.
Define $\tilde{E}(\Gamma,s,b,e)$ to be the incremental path graph constructed as follows:
\begin{enumerate}
	\item Create a $(s+1)$-segment, $\Lambda_i$, for each edge $(P_i,P_{i+1})$ in $\Gamma$.
	Choose $\Lambda_i$ to be positive or negative according to the sign of the edge $(P_i,P_{i+1})'$.
	\item Between each pair $\Lambda_i$ and $\Lambda_{i+1}$, for $i=0,\ldots, n-2$, add a $s$-block of length $b$ or $b-1$.
    The length of the $s$-block is odd if the edges $\Lambda_i$ and $\Lambda_{i+1}$ have the same sign, and
    the length is even if $\Lambda_i$ and $\Lambda_{i+1}$ have opposite signs.
    Also, the first $s$-segment in the block has sign opposite of the sign of $\Lambda_i$.
	\item Add another $s$-block to the beginning of $\Lambda_i$ of length $b$ or $b-1$ depending on the signs of $\Lambda_0$ and $e$ following the same convention as the previous step.
	Also, the first $s$-segment in the block has sign opposite of $e$.
	\item Finally, set the grading of the first vertex $Q_0$ as follows.
	\begin{equation}\label{expandstart}
	    \gr(Q_0)=\left\{
	    \begin{array}{ll}
	        \gr(P_0)+s & \text{when } e\text{ and }(P_0,P_1)\text{ are both positive}\\
            \gr(P_0)-s & \text{when } e\text{ and }(P_0,P_1)\text{ are both negative}\\
            \gr(P_0) & \text{when } e\text{ and }(P_0,P_1)\text{ have opposite sign}\\
	    \end{array}
	    \right.
	\end{equation}
\end{enumerate}
For an example, see Figure \ref{figexpand}.

\begin{figure}[t]
    \begin{subfigure}[t]{10cm}
        \centering
        \includegraphics{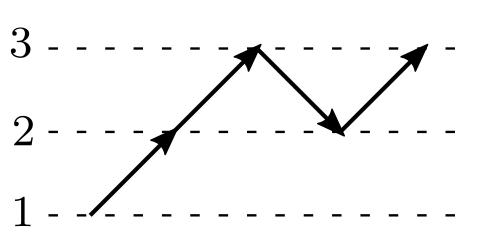}
        \caption{An incremental path $\Gamma$}
        \label{fe:gamma}
    \end{subfigure}
    \begin{subfigure}[t]{10cm}
        \centering
        \includegraphics{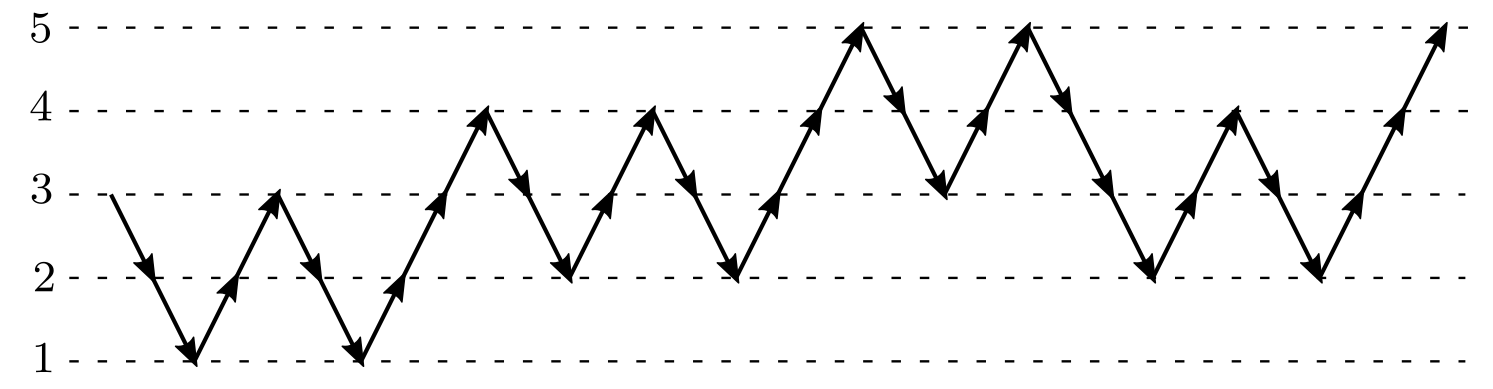}
        \caption{$\tilde{E}(\Gamma,2,3,+)$}
        \label{fe:egamma}
    \end{subfigure}
    \caption{}
    \label{figexpand}
\end{figure}

By construction, the following property holds

\begin{lem} \label{expandiso}
    Suppose $\Gamma$ and $\Gamma'$ are isomorphic incremental paths.
    For any positive integers $s$ and $b$ and any sign $e=\pm1$,
    \[
        \tilde{E}(\Gamma,s,b,e)\cong \tilde{E}(\Gamma',s,b,e).
    \]
\end{lem}

We begin by investigating the gradings of the vertices in $\tilde{E}(\Gamma,s,b,e)$.
Let $Q_0$ be the vertex at the beginning of $\tilde{E}(\Gamma,s,b,e)$.
For $i=1,\ldots,n$,
let $Q_{i}$ be the vertex at the end of $(s+1)$-segment $\Lambda_{i-1}$ as defined in the definition of $\tilde{E}$.

\begin{lem} \label{tegradings}
For each $i=1,\ldots,n$,
\begin{enumerate}[label=(\alph*)]
    \item if the sign of $\Lambda_i$ and $e$ are the same, then
    \[
        \gr(Q_i)-\gr(Q_0)=\gr(P_i)-\gr(P_0),
    \]
    \item if $\Lambda_{i-1}$ is positive and $e$ is negative, then
    \[
        \gr(Q_i)-\gr(Q_0)=\gr(P_i)-\gr(P_0)+s, and
    \]
    \item if $\Lambda_{i-1}$ is negative and $e$ is positive, then
    \[
        \gr(Q_i)-\gr(Q_0)=\gr(P_i)-\gr(P_0)-s.
    \]

\end{enumerate}
\end{lem}

\begin{proof}
Since the vertices $Q_0$ and $Q_i$ are separated some number of segments.
Let $D^+$ and $D^-$ be the number of positive or negative $(s+1)$-segments.
Likewise, let $d^+$ and $d^-$ be the number of positive or negative $s$-segments.
Note that $D^+$ and $D^-$ are also the number of positive and negative edges in $\Gamma$ so
\[
    D^+-D^-=\gr(P_i)-\gr(P_0).
\]

Suppose $\Lambda_{i-1}$ and $e$ are have the same sign, then the number of positive segments in $\tilde{E}(\Gamma,s,b,e)$ is equal to the number of negative segments so
\[
    D^++d^+=D^-+d^- .
\]
Thus,
\begin{align*}
    \gr(Q_i)-\gr(Q_0)=&D^+(s+1)-D^-(s+1)+d^+s-d^-s\\
    =&(D^++d^+)s-(D^-+d^-)s+D^+-D^-\\
    =&D^+-D^-\\
    =&\gr(P_i)-\gr(P_0).
\end{align*}

Suppose $\Lambda_{i-1}$ is positive and $e$ is negative, then the total number of positive segments in $\tilde{E}(\Gamma,s,b,e)$ is one more than the total number of negative segments so
\begin{align*}
    \gr(Q_i)-\gr(Q_0)=&D^+(s+1)-D^-(s+1)+d^+s-d^-s\\
    =&(D^++d^+)s-(D^-+d^-)s+D^+-D^-\\
    =&s+D^+-D^-\\
    =&\gr(P_i)-\gr(P_0)+s.
\end{align*}

Suppose $\Lambda_{i-1}$ is negative and $e$ is positive, then the total number of positive segments in $\tilde{E}(\Gamma,s,b,e)$ is one less than the total number of negative segments so
\begin{align*}
    \gr(Q_i)-\gr(Q_0)=&D^+(s+1)-D^-(s+1)+d^+s-d^-s\\
    =&(D^++d^+)s-(D^-+d^-)s+D^+-D^-\\
    =&-s+D^+-D^-\\
    =&\gr(P_i)-\gr(P_0)-s.
\end{align*}
\end{proof}

From this, we can show that concatenation behaves well under expansion.

\begin{lem} \label{expandconcat}
    Suppose $\Gamma$ and $\Gamma'$ are incremental paths where the last vertex in $\Gamma$ has the same grading as the first vertex in $\Gamma'$.
    Let $e'$ be the sign of the last edge in $\Gamma$.
    For any positive integers $s$ and $b$ and any sign $e=\pm1$,
    \[
        \tilde{E}(\Gamma*\Gamma',s,b,e)\cong\tilde{E}(\Gamma,s,b,e)*\tilde{E}(\Gamma',s,b,e') .
    \]
\end{lem}

\begin{proof}
The conclusion will be true by definition of the expansion procedure as long as $\tilde{E}(\Gamma,s,b,e)$ and $\tilde{E}(\Gamma',s,b,e')$ can be concatenated.
Thus, our goal is to show that the last vertex in $\tilde{E}(\Gamma,s,b,e)$ has the grading as the first vertex in $\tilde{E}(\Gamma',s,b,e')$.
This can be done by computing $\tilde{E}(\Gamma*\Gamma',s,b,e)$ for many cases depending on the signs of $e$, the last edge in $\Gamma$, and the first edge in $\Gamma'$.

For example, suppose $e$, the last edge in $\Gamma$, and the first edge in $\Gamma'$ are all positive.
Let $P_0$ and $P_n$ be the first and last vertices of $\Gamma$.
Let $P'_0$ be the first vertex in $\Gamma'$ so $\gr(P_n)=\gr(P'_0)$.
Let $Q_0$ and $Q_n$ be the first and last vertices of $\tilde{E}(\Gamma,s,b,e)$.
Finally, let $Q'_0$ be the first vertex in $\tilde{E}(\Gamma',s,b,e')$.

By (\ref{expandstart}),
\[
    \gr(Q'_0)=\gr(P'_0)+s=\gr(P_n)+s
\]
By Lemma \ref{tegradings},
\begin{align*}
    \gr(Q_n)=&\gr(P_n)-\gr(P_0)+\gr(Q_0)\\
    =&\gr(Q'_0)-s-\gr(P_0)+\gr(P_0)+s\\
    =&\gr(Q'_0).
\end{align*}
The proofs of all the other cases are similar.
\end{proof}

Let $\Gamma$ be a closable incremental path, and 
let $e$ be the sign of the last edge in $\Gamma$.
For any two positive integers $s$ and $b$, define
\[
    E(\Gamma,s,b):=\tilde{E}(\Gamma,s,b,e).
\]
When $\Gamma$ is closable, $E(\Gamma,s,b)$ is also closable.

Suppose $\Gamma'$ is a closable incremental path such that $\cl(\Gamma)\cong\cl(\Gamma')$.
By construction,
\begin{equation} \label{expandcloseiso}
    \cl(E(\Gamma,s,b))\cong\cl(E(\Gamma',s,b))
\end{equation}
for all positive integers $s$ and $b$.

For a incremental cycle $\oGamma$,
define
\[
    E(\oGamma,s,b):=\cl(E(\Gamma,s,b)).
\]
where $\Gamma$ is any incremental path such that $\cl(\Gamma)\cong\oGamma$.
By (\ref{expandcloseiso}), $E(\oGamma,s,b)$ is well-defined.

By construction reduction and expansion natural opposite operations.

\begin{prop} \label{reduceexpand}
    Suppose $(p,q)$ is a co-prime pair with $q>0$.
    Define $\kappa$ and $\kappa'$ as in (\ref{eucdivpq}) and (\ref{eucdivqr}).
    \[
        E(R(\oGamma)(p,q),\kappa,\kappa')\cong\oGamma(p,q)
    \]
\end{prop}

Given an arbitrary co-prime pair $(p^*,q^*)$ and integers $s$ and $b$,
$E(\oGamma(p^*,q^*),s,b)$ may not be $\oGamma(p,q)$ for any co-prime $(p,q)$ with $q$ odd.
Consider the pair $(5,3)$.
Suppose $E(\oGamma(5,3),2,3)\cong\oGamma(p,q)$ for some pair $(p,q)$.
Then, $q=3(5)+3=18$.

\section{Proof of Lemma \ref{nestedword}} \label{SecLemmaProof}

In this section, we reinterpret Lemma \ref{nestedword} as set of properties of the cycle graph $\oGamma(p,q)$.
These properties will hold for simple co-prime pairs $(p,q)$ with $q=1$ or $p\modop q=1$.
Then, it is shown these conditions hold for any co-prime pair of integers $p$ and $q$ with $p$ positive and $q$ odd by a strong induction argument using the relative isomorphism between $\oGamma(p,q)$ and $\oGamma(p,-q)$ and the reduction from $\oGamma(p,q)$ to $R(\oGamma)(p,q)$.

\subsection{Making Words From Graphs}

Given an incremental path $\Gamma$, a word $\rho(\Gamma)$ in $\mathcal{S}$ can be defined as follows.
Let $\{P_1,\ldots,P_n\}$ be the vertices of $\Gamma$ indexed so that the edge $(P_i,P_{i+1})$ is in $\Gamma$.
For $i=2,\ldots,n$, let $s_i=\gr(P_i)-\gr(P_{i-1})$ and
let $N_i=\gr(Q_i)+\theta(s_i)$
where $\theta(1)=1$ and $\theta(-1)=0$.
Define
\begin{equation}\label{tauprime}
\rho(\Gamma):=\left\{\begin{array}{ll}
S_{N_3}^{s_3}S_{N_5}^{s_5}\cdots S_{N_k}^{s_k}&\text{if }n>2\text{ and }\gr(P_1)\text{ is even}\\
S_{N_2}^{s_2}S_{N_4}^{s_4}\cdots S_{N_k}^{s_k}&\text{if }n>1\text{ and }\gr(P_1)\text{ is odd}\\
1&\text{otherwise}
\end{array}\right.
\end{equation}
where $k=n-1$ if $n\equiv\gr(P_1)$ modulo 2, and 
$k=n$ if $n\not\equiv\gr(P_1)$ modulo 2.
Given a two-bridge link $L(p/q)$, by Proposition \ref{subscriptsprop}, $\rho(\Gamma(p,q))$ is the word $R_0$.

\begin{lem}\label{taulem}
Given incremental paths $\Gamma$ and $\Gamma'$ such that the last vertex of $\Gamma$ has the same grading as the first vertex of $\Gamma'$,
\[
    \rho(\Gamma*\Gamma')=\rho(\Gamma)\rho(\Gamma').
\]
\end{lem}

\begin{proof}
Let $\{P_1,\ldots,P_n\}$ and $\{P_1',\ldots,P'_{n'}\}$ be the vertex sets for incremental paths $\Gamma$ and $\Gamma'$ respectively.
Also, define $N_2,\ldots,N_n$ and $s_2,\ldots,s_n$ for $\Gamma$ as in the definition of $\rho$.
Similarly, define $N_2',\ldots,N'_{n'}$ and $s_2',\ldots,s'_{n'}$ for $\Gamma'$.
Let $\Gamma''=\Gamma*\Gamma'$, which has length $n+n'-1$, and
define $N_2'',\ldots,N''_{n+n'-1}$ and $s_2'',\ldots,s''_{n+n'-1}$ for $\Gamma''$ as the analogous integers are defined for $\Gamma$ and $\Gamma'$.

This result is just a matter of computing $\rho(\Gamma*\Gamma')$ for each case of (\ref{tauprime}) for $\Gamma$ and $\Gamma'$.
For example, suppose $\gr(P_1)$ and $n$ are even, $n>2$, and $n'>1$.
Then, since $n$ is even,
\[
    \gr(P'_1)=\gr(P_n)\equiv(\gr(P_1)+n-1)\equiv\gr(P_1)+1\hspace{0.7cm}(\mathsf{mod}\;2)
\]
so since $\gr(P_1)$ is even, $\gr(P'_1)$ is odd.
Thus,
\[
    \rho(\Gamma)=S_{N_3}^{s_3}S_{N_5}^{s_5}\cdots S_{N_{n-1}}^{s_{n-1}}
\]
and
\[
    \rho(\Gamma')=S_{N_2}^{s_2}S_{N_4}^{s_4}\cdots S_{N_k}^{s_k}
\]
where $k=n'$ when $n'$ is even and $k=n'-1$ when $n'$ is odd.

For each $i=1,\ldots,n+n'-1$,
\[
    \gr(P''_i)=\left\{
    \begin{array}{ll}
    \gr(P_i)&\text{when }1\leq i\leq n\\
    \gr(P'_{i-n+1})&\text{when }n\leq i\leq n+n'-1
    \end{array}
    \right.
.
\]
Thus, when $2\leq i\leq n$, $s''_i=s_i$ and $N''_i=N_i$, and
when $n+1\leq i\leq n+n'-1$, $s''_i=s_{i-n+1}$ and $N''_i=N_{i-n+1}$.
Therefore,
\[
\rho(\Gamma*\Gamma')=S_{N_3}^{s_3}S_{N_5}^{s_5}\cdots S_{N_{n-1}}^{s_{n-1}}S_{N_2'}^{s_2'}S_{N_4'}^{s_4'}\cdots S_{N'_{k}}^{s'_{k}}=\rho(\Gamma)\rho(\Gamma')
\]

The proofs of all the other cases are similar.
\end{proof}

\begin{lem}\label{taucyc}
Given two closable incremental paths $\Gamma$ and $\Gamma'$ such that $\cl(\Gamma)$ is isomorphic to $\cl(\Gamma')$,
there is a subgraph $\Upsilon$ of $\Gamma$ such that
$$\rho(\Gamma')=\rho(\Upsilon)^{-1}\rho(\Gamma)\rho(\Upsilon).$$
\end{lem}

\begin{proof}
If $\cl(\Gamma)\cong\cl(\Gamma')$ then there are some graphs $\Upsilon$ and $\Omega$ such that $\Gamma=\Upsilon*\Omega$ and $\Gamma'=\Omega*\Upsilon$ (see Figure \ref{figtaucyc} for an example).
Therefore,
\[
\rho(\Gamma')=\rho(\Omega)\rho(\Upsilon)=\rho(\Upsilon)^{-1}\rho(\Upsilon)\rho(\Omega)\rho(\Upsilon)=\rho(\Upsilon)^{-1}\rho(\Gamma)\rho(\Upsilon)
\]
\end{proof}

\begin{figure}[t]
\includegraphics{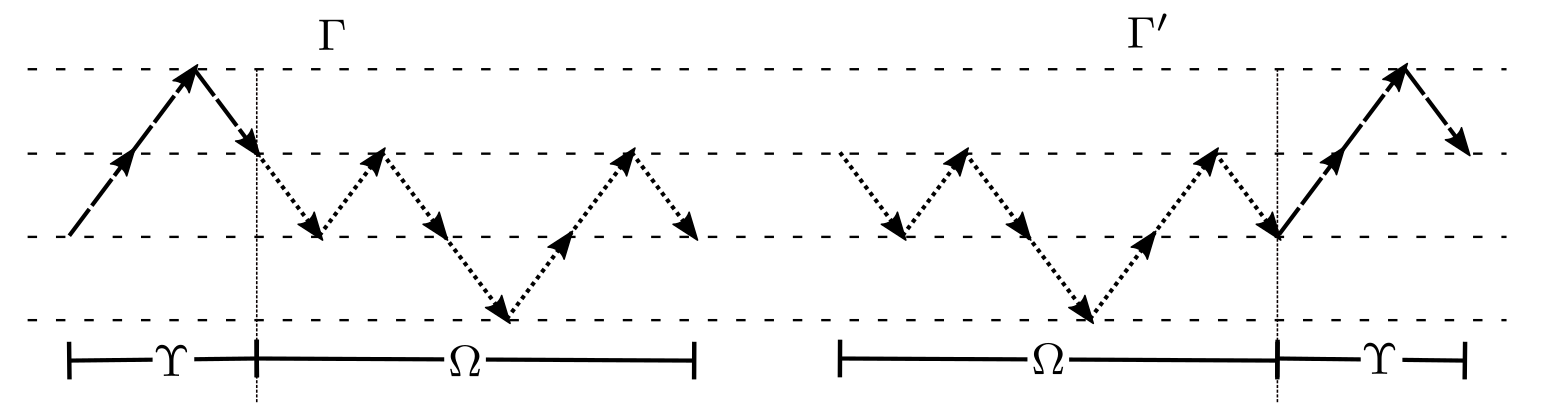}
\caption{Closable graphs $\Gamma$ and $\Gamma'$ with isomorphic closures with the subgraphs $\Upsilon$ (dashed) and $\Omega$ (dotted) shown.}
\label{figtaucyc}
\end{figure}

\subsection{Summits and Bottoms in Cycle Graphs}
Let $(p,q)$ be a co-prime pair, and define $M$ and $m$ for $L(p/q)$ as in section \ref{GroupPres}.
In Lamma \ref{nestedword}, we are primarily interest in the appearances of $S^{\pm}_M$ and $S^{\pm}_m$ in the word $R_0$.
When $M$ is odd, the $i$th $S$-generator of $R_0$ is $S^{\pm}_M$ precisely when $\sigma_{2i}=M+1$, and
when $M$ is even, the $i$th $S$-generator of $R_0$ is $S^{\pm}_M$ when $\sigma_{2i-1}=M+1$.
Thus, appearances $S^{\pm}_M$ in $R_0$ correspond to the indices when $\sigma_i$ is maximal.
Similarly, the $i$th $S$-generator of $R_0$ is $S^{\pm}_m$ precisely when $\sigma_{2i-1}=m$ when $m$ is odd or $\sigma_{2i}=m$ when $m$ is even.
Thus, appearances $S^{\pm}_m$ in $R_0$ correspond to the indices when $\sigma_i$ is minimal.

A vertex, $P$, in a graded graph $\Gamma$ is called a \emph{summit} if $\gr(P)\geq\gr(Q)$ for any vertex $Q$ in $\Gamma$.
Similarly, $P$ is called a \emph{bottom} if $\gr(P)\leq\gr(Q)$ for any vertex $Q$ in $\Gamma$.
For each co-prime pair $(p,q)$ the grading of a summit of $\Gamma(p,q)$ is always $M+1$
and the grading of a bottom of $\Gamma(p,q)$ is always $m$.
Furthermore, the appearances of $S_M$ in $R_0$ correspond precisely to the summits in $\Gamma(p,q)$,
and the appearances of $S_m$ correspond to bottoms.

\begin{figure}[b]
\includegraphics{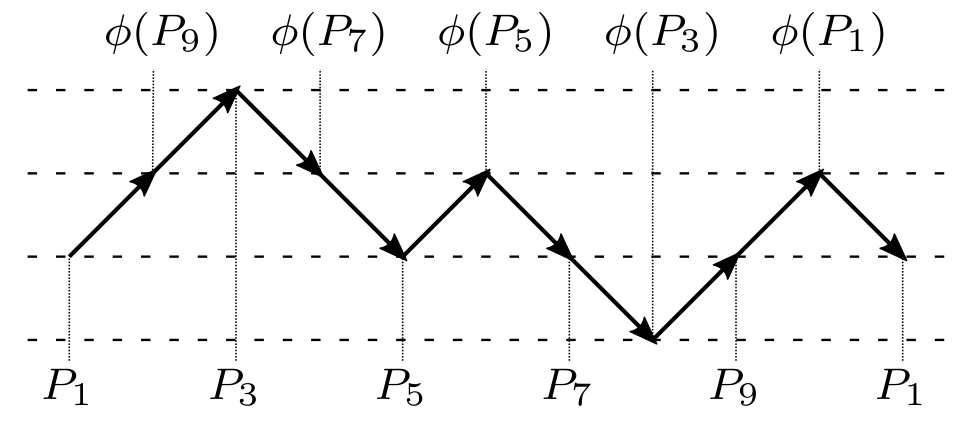}
\caption{A symmetric incremental cycle. The first and last vertices are identified. $\phi$ is the unique order reversing bijection defined by $\phi(P_1)=P_{10}$.}
\label{figsymmetric}
\end{figure}

\subsection{Symmetric Incremental Paths and Cycles}

It is useful to know when an incremental cycle is relatively isomorphic to itself after rotating $180^{\circ}$ and reversing its edges.
More precsiely, we call an incremental cycle $\Gamma$ \emph{symmetric} if there is a bijection $\phi:V(\Gamma)\to V(\Gamma)$ such that
\begin{enumerate}
	\item $(P,Q)$ is an edge of $\Gamma$ if and only if $(\phi(Q),\phi(P))$ is an edge of $\Gamma$ for any two vertices $P$ and $Q$ in $\Gamma$ and
	\item for some integer $k$, $\gr(P)+\gr(\phi(P))=k$ for every vertex $P$ in $\Gamma$.
\end{enumerate}
An incremental path $\Gamma$ is called \emph{symmetric} if $\cl(\Gamma)$ is symmetric (see Figure \ref{figsymmetric}).
The symmetry of incremental paths and cycles plays an important role in investigating properties \ref{nwsym} and \ref{nwsym2} of Lemma \ref{nestedword}.

\subsection{Reinterpretation of Lemma \ref{nestedword}}

Here we reinterpret Lemma \ref{nestedword} in terms of incremental paths and cycles.
Given a closable incremental path $\Gamma$ and positive integer $n$, define $\Gamma^n$ to be the concatenation of $n$ copies of $\Gamma$.
We call a co-prime pair $(p,q)$ an \emph{pre-RTFN pair} if there is a positive integer $N$, sequences of incremental paths
\[
    \Gamma_0,\ldots,\Gamma_N
\]
and
\[
    \Upsilon_0,\ldots,\Upsilon_N
\]
and a sequence of positive integers
\[
    n_0,\ldots,n_N
\]
such that the following conditions are satisfied:
\begin{figure}[t]
    \includegraphics{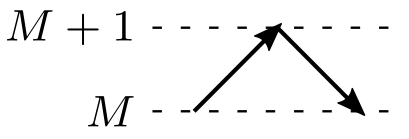}
    \caption{The graph $\Gammatop$}
    \label{gammatop}
\end{figure}
\begin{enumerate}[label=(R\arabic*)]
    \item $\Gamma_0=\Gamma(p,q)$, \label{ngstart}
    \item $\Gamma_N$ is isomorpic to the graph $\Gammatop$ defined in Figure \ref{gammatop}. \label{ngend}
    \item for each $i=1,\ldots,N$,
    \[
        \cl(\Gamma_{i-1})\cong\cl(\Gamma_i^{n_i}*\Upsilon_i) ,
    \] \label{ngrec}
    \item for each $i=1,\ldots,N$, no summits appear in $\Upsilon_i$, and \label{ngvsm}
    \item for each $i=0,\ldots,N$, $\Gamma_i$ is symmetric,
    and when $i\geq 1$,
    $\Gamma_i$ contains no bottoms. \label{ngsym}
\end{enumerate}
For an example, Figure \ref{prertfn} demonstrates that
$(33,23)$ is a pre-RTFN pair.

\begin{figure}[b]
    \begin{subfigure}[t]{10cm}
        \centering    
        \includegraphics{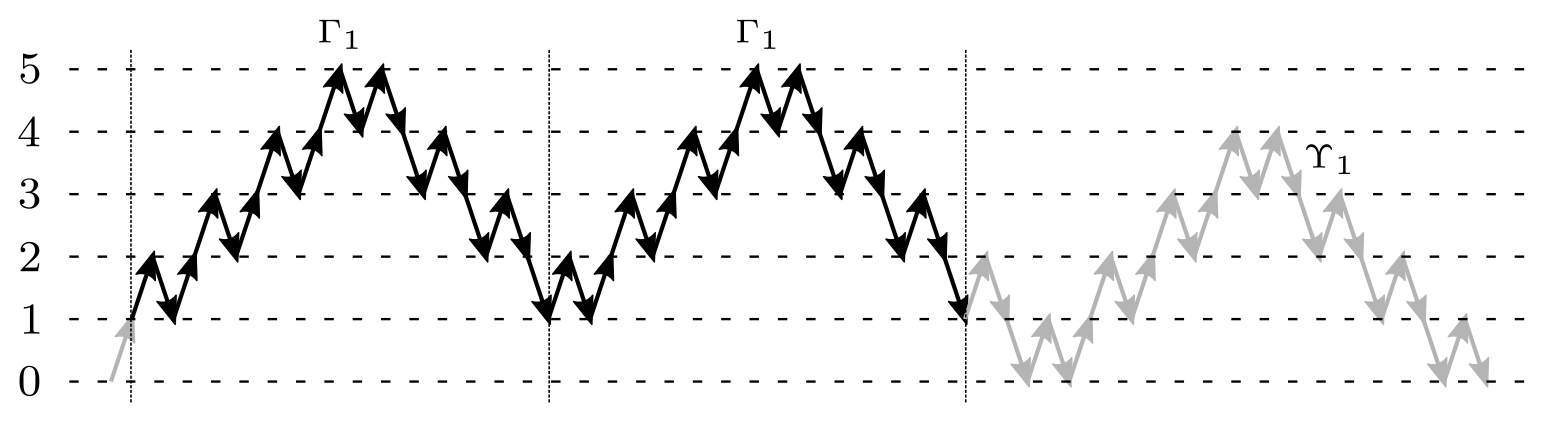}
        \caption{$\Gamma_0=\oGamma(33,23)$ with $\Upsilon_1$ in gray}
        \label{pr:a}
    \end{subfigure}
    \begin{subfigure}[t]{10cm}
        \centering    

        \includegraphics{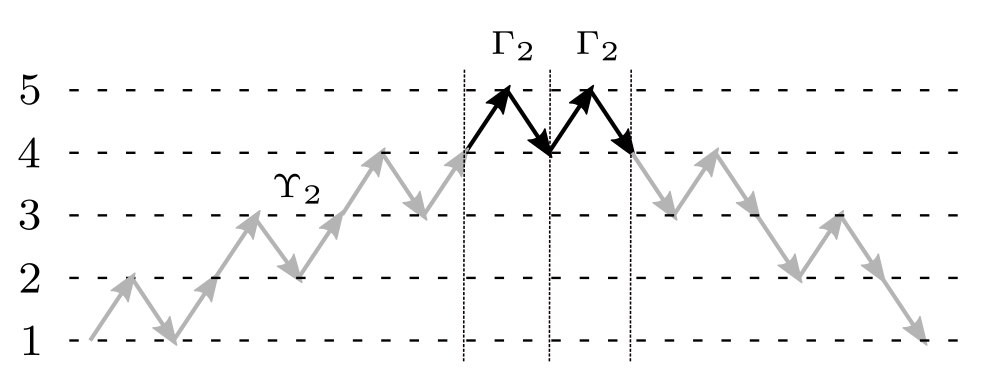}
        \caption{$\Gamma_1$ with $\Upsilon_2$ in gray}
        \label{pr:b}
    \end{subfigure}
    \caption{$(33,23)$ is a pre-RTFN pair.}
\label{prertfn}
\end{figure}

\begin{lem}\label{negsignscor}
    $(p,q)$ is a pre-RTFN pair if and only if $(p,-q)$ is a pre-RTFN pair.
\end{lem}

\begin{proof}
    This follows immediately from Proposition \ref{negrelsio}.
\end{proof}

\begin{lem}\label{nestedgraphs}
    Suppose $(p,q)$ is a co-prime pair.
    If $(p,q)$ is a pre-RTFN pair, then $L(p/q)$ satisfies Lemma \ref{nestedword}.
\end{lem}

\begin{proof}
Let $(p,q)$ be a pre-RNTF pair.
For each $i=0,\ldots, N$, define
\[
	\widehat{A}_i:=\rho(\Gamma_{i}),
\]
and when $i>0$, define
\[
	\widehat{V}_i:=\rho(\Upsilon_{N-i}) .
\]

\textit{Proof of \ref{nwstart} and \ref{nwend}.}
By \ref{ngstart} and \ref{ngend},
\[
    \widehat{A}_0=\rho(\Gamma_0)=\rho(\Gamma(p,q))=R_0,
\]
and
\[
    A_N=\rho(\Gamma_N)=S^{\pm1}_M .
\]

\textit{Proof of \ref{nwrec}.}
Suppose $i$ is an integer with $1\leq i\leq N$.
By \ref{ngrec},
\[
    \cl(\Gamma_{i-1})\cong\cl(\Gamma_i^{n_i}*\Upsilon_i)
\]
so by Lemma \ref{taucyc}, there exists a word $W$ such that
\[
    \rho(\Gamma_{i-1})=W^{-1}\rho(\Gamma_i^{n_i}*\Upsilon_i)W.
\]
Therefore,
\begin{align*}
    \widehat{A}_{i-1}=&\rho(\Gamma_{i-1})\\
    =&W^{-1}\rho(\Gamma_i^{n_i}*\Upsilon_i)W\\
    =&W^{-1}\widehat{A}_{i}^{n_i}\widehat{V}_{i}W .
\end{align*}

\begin{figure}[b]
\includegraphics{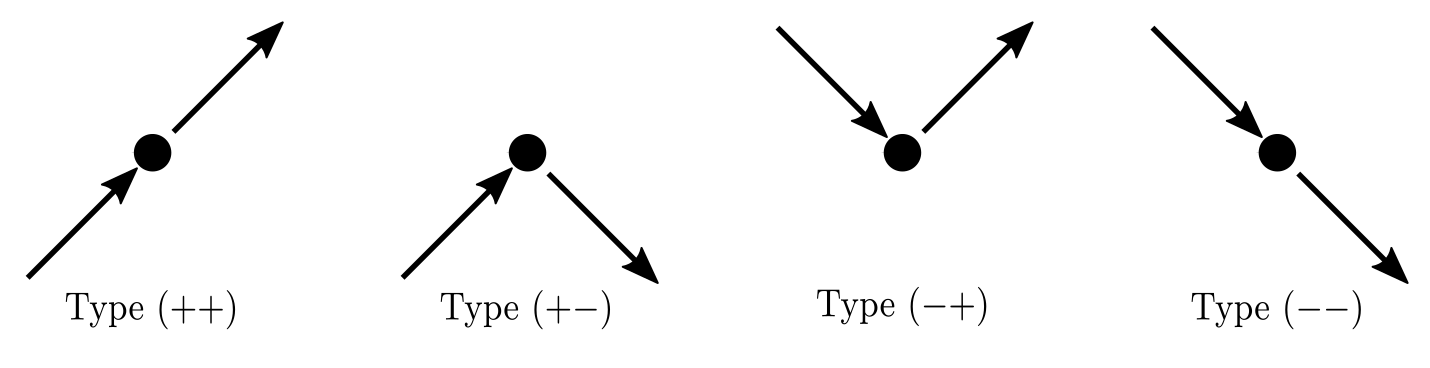}
\caption{The four vertex types}
\label{vertextypes}
\end{figure}

\textit{Proof of \ref{nwvsm}.}
For each $i=1,\ldots,N$,
since no summits appear in $\Upsilon_i$,
$S^{\pm1}_M$ cannot appear in $\widehat{V}_i$.

\textit{Proof of \ref{nwsym}.}
Suppose $i$ is an integer with $0\leq i\leq N$.
The maximum grading of a vertex in $\Gamma_i$ is $M+1$.
Let $l$ be the minimum grading of a vertex in $\Gamma_i$.
For some integer coefficients $b_l,b_{l+1}\ldots,b_M$,
\[
    [\rho(\Gamma_i)]=b_lS'_l+b_{l+1}S'_{l+1}+\cdots+b_MS'_M.
\]
Our goal is to show that for each $j=0,\ldots,M-l$, $|b_{l+j}|=|b_{M-j}|$.

The vertices of $\cl(\Gamma_i)$ can be classified into four types according to Figure \ref{vertextypes}.
Define $v_{(**)}(n)$ to be the number vertices in $\cl(\Gamma_i)$ of type $(**)$ with grading $n$.

Suppose $n=l,\ldots, M$.
When $n$ is even,
$S_n$ always has exponent $-1$ in $\rho(\Gamma_i)$,
and $S^{-1}_n$ appears precisely when there is negative edge followed a vertex in $\cl(\Gamma_i)$ with grading $n$
so
\begin{equation} \label{beven}
	|b_{n}|=v_{(--)}(n)+v_{(-+)}(n).
\end{equation}
Similarly, When $n$ is odd,
$S_n$ always has exponent $1$ in $\rho(\Gamma_i)$,
and $S_n$ appears precisely when there is a vertex in $\cl(\Gamma_i)$ with grading $n$ followed by a positive edge
so
\begin{equation} \label{bodd}
	|b_{n}|=v_{(++)}(n+1)+v_{(+-)}(n+1).
\end{equation}

\begin{figure}[t]
\includegraphics{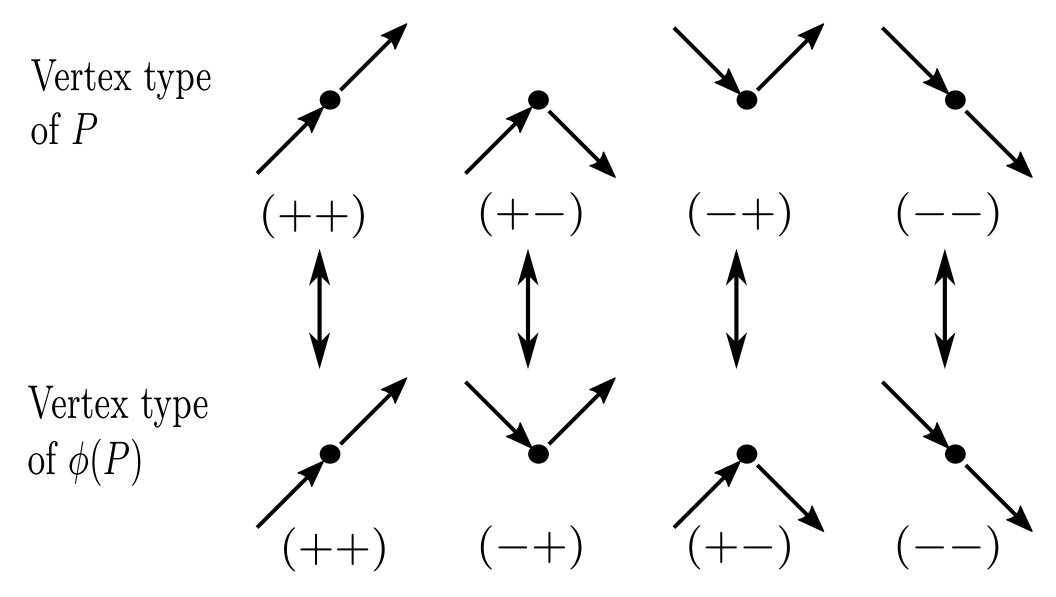}
\caption{The effect of $\phi$ on vertex type}
\label{symvertextypes}
\end{figure}

Since $\Gamma_i$ is symmetric by \ref{ngsym}, there is an order reversing bijection $\phi$ of the vertex set of $\cl(\Gamma_i)$ such that $\gr(P)+\gr(\phi(P))=l+M+1$ for each vertex $P$ in $\cl(\Gamma_i)$.
Furthermore, $P$ and $\phi(P)$ have types rotated $180^{\circ}$ with arrows reversed (see Figure \ref{symvertextypes}).
As a consequence,
\begin{equation} \label{numtypesphi}
    \begin{split}
        v_{(--)}(n)=v_{(--)}(l+M+1-n)\\
        v_{(-+)}(n)=v_{(+-)}(l+M+1-n)\\
	    v_{(++)}(n)=v_{(++)}(l+M+1-n)\\
	    v_{(+-)}(n)=v_{(-+)}(l+M+1-n)
    \end{split}
    .
\end{equation}
Each positive edge connects a vertex of type $(*+)$ to a vertex of type $(+*)$.
Likewise, each negative edge connects a vertex of type $(*-)$ to a vertex of type $(-*)$ (see Figure \ref{figadj}).
Thus,
\begin{equation} \label{adjacenttypes}
    \begin{split}
	v_{(++)}(n)+v_{(-+)}(n)=v_{(++)}(n+1)+v_{(+-)}(n+1)\\
	v_{(--)}(n)+v_{(+-)}(n)=v_{(--)}(n-1)+v_{(-+)}(n-1)
	\end{split}
	.
\end{equation}
Since $\Gamma_i$ is closable and the gradings of adjacent vertices differ by $\pm1$,
every time $\Gamma_i$ passes from below to above some grading level at a vertex, $\Gamma_i$ must pass from above to below the same grading level at some other vertex.
Thus, in each grading $n$,
\begin{equation} \label{ppequalsnn}
	v_{(++)}(n)=v_{(--)}(n).
\end{equation}

\begin{figure}[b]
\includegraphics{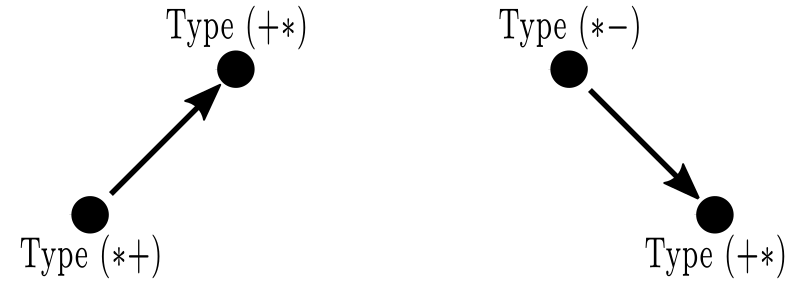}
\caption{}
\label{figadj}
\end{figure}

Now, we show that $|b_{l+j}|=|b_{M-j}|$.
Let $j$ be an integer such that $0\leq j \leq M-l$.
When $l+j$ and $M-j$ are both even, by (\ref{beven}), (\ref{numtypesphi}), and (\ref{adjacenttypes}),
\begin{align*}
	|b_{l+j}|=&v_{(--)}(l+j)+v_{(-+)}(l+j)\\
	=&v_{(--)}(M-j+1)+v_{(+-)}(M-j+1)\\
	=&v_{(--)}(M-j)+v_{(-+)}(M-j)\\
	=&|b_{M-j}|.
\end{align*}

When $l+j$ and $M-j$ are odd, by (\ref{bodd}), (\ref{numtypesphi}), and (\ref{adjacenttypes})
\begin{align*}
	|b_{l+j}|=&v_{(++)}(l+j+1)+v_{(+-)}(l+j+1)\\
	=&v_{(++)}(M-j)+v_{(-+)}(M-j)\\
	=&v_{(++)}(M-j+1)+v_{(+-)}(M-j+1)\\
	=&|b_{M-j}|.
\end{align*}

When $l+j$ is even and $M-j$ is odd, by (\ref{beven}), (\ref{numtypesphi}), (\ref{ppequalsnn}), and (\ref{bodd}),
\begin{align*}
	|b_{l+j}|=&v_{(--)}(l+j)+v_{(-+)}(l+j)\\
	=&v_{(--)}(M-j+1)+v_{(+-)}(M-j+1)\\
	=&v_{(++)}(M-j+1)+v_{(+-)}(M-j+1)\\
	=&|b_{M-j}|.
\end{align*}

When $l+j$ is odd and $M-j$ is even, by (\ref{bodd}), (\ref{numtypesphi}), (\ref{ppequalsnn}), and (\ref{beven}),
\begin{align*}
	|b_{l+j}|=&v_{(++)}(l+j+1)+v_{(+-)}(l+j+1)\\
	=&v_{(++)}(M-j)+v_{(-+)}(M-j)\\
	=&v_{(--)}(M-j)+v_{(-+)}(M-j)\\
	=&|b_{M-j}|.
\end{align*}

When $i\geq 1$, no bottoms appear in $\Gamma_i$ so $l>m$.

\textit{Proof of \ref{nwstart2}, \ref{nwend2}, \ref{nwrec2}, \ref{nwvsm2}, and \ref{nwsym2}.}
Since $\Gamma_0=\Gamma(p,q)$ is symmetric, there is an order reversing bijection $\overline{\phi}$ on the vertices of $\oGamma$ such that 
\[
	\gr(P)+\gr(\overline{\phi}(P))=m+M+1
\]
for each vertex $P$ in $\oGamma(p,q)$
Thus, $\overline{\phi}$ induces a map on the subgraphs of $\oGamma(p,q)$.

For each $i=0,\ldots, N$, define
\[
    \widecheck{A}_i:=\rho(\overline{\phi}(\Gamma_{N-i})),
\]
and when $i>0$, define
\[
	\widecheck{V}_i:=\rho(\overline{\phi}(\Upsilon_{N-i})).
\]

\ref{nwstart2}, \ref{nwend2}, \ref{nwrec2}, \ref{nwvsm2}, and \ref{nwsym2} follow from proofs similar to the those used for \ref{nwstart}, \ref{nwend}, \ref{nwrec}, \ref{nwvsm}, and \ref{nwsym}.
\end{proof}

\subsection{Using Reductions for Induction}

Suppose $(p,q)$ is a co-prime pair with $q>1$ and with $(p\modop q)\neq 1$.
By Lemma \ref{redpair}, $R(\oGamma)(p,q)$ is isomorphic to $\oGamma(p^*,q^*)$ for some co-prime pair $(p^*,q^*)$
so along with Lemma \ref{negsignscor}, $\oGamma(p,q)$ can be simplified through a sequence of reductions and relative isomorphisms to $\oGamma(p_0,q_0)$ such that $q_0=1$ or $(p\modop q)=1$.
\begin{example}
\[
\oGamma(119,43)\overset{R}{\rightarrow}\oGamma(33,-23)\overset{rel.}{\cong}\oGamma(33,23)\overset{R}{\rightarrow}\oGamma(10,3)
\]
\end{example}

The goal now is to show that when $(p^*,q^*)$ is a pre-RTFN pair, $(p,q)$ is also a pre-RTFN pair.

\subsection{Leading and Trailing Vertices}
Call a vertex in $\oGamma(p,q)$ at the end of a $(\kappa+1)$-segment a \emph{leading vertex}, and
any vertex at the beginning of a $(\kappa+1)$-segment a \emph{trailing vertex} (see Figure \ref{figleadtrail}).
Let $P$ be a leading vertex in $\oGamma(p,q)$,
and let $\Lambda_L$ be the $(\kappa+1)$-segment of $\oGamma(p,q)$ immediately preceding $P$.
Define $f_L(P)$ to be the vertex at the end of the edge in $R(\oGamma)(p,q)$ corresponding to $\Lambda_L$.
Let $P$ be a trailing vertex in $\oGamma(p,q)$,
and let $\Lambda_T$ be the $(\kappa+1)$-segment of $\oGamma(p,q)$ immediately following $P$.
Define $f_T(P)$ to be the vertex at the beginning of the edge in $R(\oGamma)(p,q)$ corresponding to $\Lambda_T$.

$f_L$ is a bijection from the leading vertices of $\Gamma(p,q)$ to the vertex set of $R(\oGamma)(p,q)$, and
$f_T$ is a bijection from the trailing vertices of $\Gamma(p,q)$ to the vertex set of $R(\oGamma)(p,q)$.
Let $P^*$ be a vertex in $R(\oGamma)(p,q)$.
Since $f_L^{-1}(P^*)$ and $f_T^{-1}(P^*)$ are separated by a $\kappa$-block of length $\kappa'$ or $\kappa'-1$, the gradings of $f_L^{-1}(P^*)$ and $f_T^{-1}(P^*)$ are either the same of differ by $\pm\kappa$.

\begin{figure}[t]
    \includegraphics{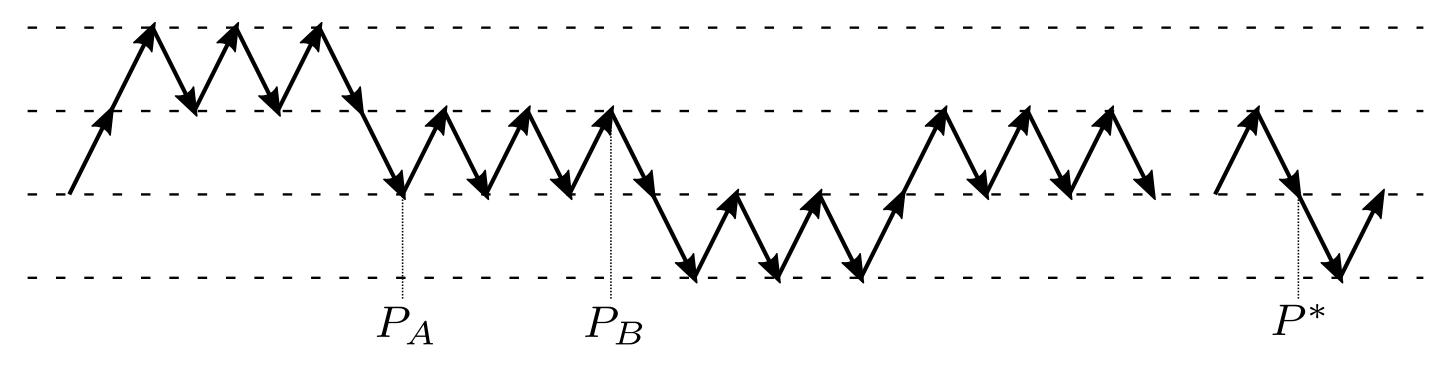}

    \caption{$P_A$ is a leading vertex of $\oGamma(13,11)$, and $P_B$ is a trailing vertex of $\oGamma(13,11)$ (left). $f_L(P_A)=f_T(P_B)=P^*$ in $R(\oGamma)(13,11)$ (right).}
    \label{figleadtrail}
\end{figure}
Any vertex in $\oGamma(p,q)$ at the end of a positive (or negative) segment is called a \emph{peak} (resp. \emph{valley}).
There is a relationship between the gradings of the vertices in $\oGamma(p,q)$ and $R(\oGamma)(p,q)$.

\begin{prop} \label{phigradings}
Let $P$ and $Q$ be leading vertices of $\oGamma(p,q)$.
\begin{enumerate}
    \item If $P$ and $Q$ are both peaks or both valleys, then
    \[
    \gr(f_L(P))-\gr(f_L(Q))=\gr(P)-\gr(Q).
    \]
    \item If $P$ is a valley and $Q$ is a peak, then
    \[
    \gr(f_L(P))-\gr(f_L(Q))=\gr(P)-\gr(Q)+\kappa.
    \]
    \item If $P$ is a peak and $Q$ is a valley, then
    \[
    \gr(f_L(P))-\gr(f_L(Q))=\gr(P)-\gr(Q)-\kappa.
    \]
\end{enumerate}
\end{prop}

\begin{proof}
    This follows immediately from Lemma \ref{tegradings}
    by consider the unique path subgraph of $R(\oGamma)(p,q)$ beginning with $f_L(P)$ and ending $f_L(Q)$.
\end{proof}

\begin{cor}
$P$ is a leading summit of $\oGamma(p,q)$ if and only if $f_L(P)$ is a summit of $R(\oGamma)(p,q)$.
\end{cor}

\subsection{Proof of Lemma \ref{nestedword}} \label{LemmaProof}

We now have everything we need to show that every co-prime pair $(p,q)$ with $p$ positive and $q$ odd is a pre-RTFN pair.
For each co-prime pair,
we need to find a positive integer $N$, subgraphs
\[
    \Gamma_0,\ldots,\Gamma_N
\]
and
\[
    \Upsilon_1,\ldots,\Upsilon_N
\]
and integers
\[
    n_1,\ldots,n_N
\]
satisfying \ref{ngstart},\ref{ngend},\ref{ngrec},\ref{ngvsm}, and \ref{ngsym}.
We prove this using a strong induction starting with the base cases below.

Given a subgraph $\Upsilon$ of a incremental cycle $\oGamma$,
define $\oGamma-\Upsilon$ to be the incremental path obtained by removing the edges and the interior vertices of $\Upsilon$ from $\oGamma$; see Figure \ref{subtraction}.

\begin{figure}[t]
\includegraphics{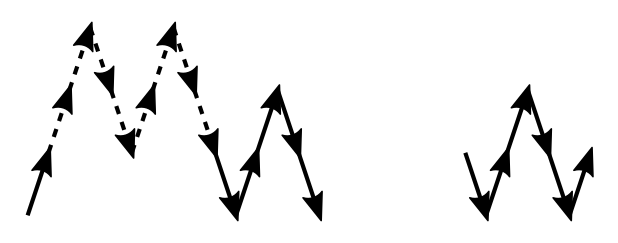}
\caption{A graph $\oGamma$ (left) with subgraph $\Upsilon$ (dashed) and $\oGamma-\Upsilon$ (right).}
\label{subtraction}
\end{figure}

\begin{figure}[b]
    \begin{subfigure}[t]{4cm}
        \centering
        \includegraphics{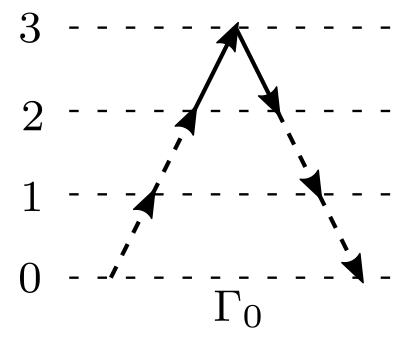}
        \caption{$\Gamma(3,1)$ (left) only has one summit.
        The solid arrows indicate $\Gamma_1$.}
        \label{bc:g31}
    \end{subfigure}
    \hspace{1cm}
    \begin{subfigure}[t]{7cm}
        \centering
        \includegraphics{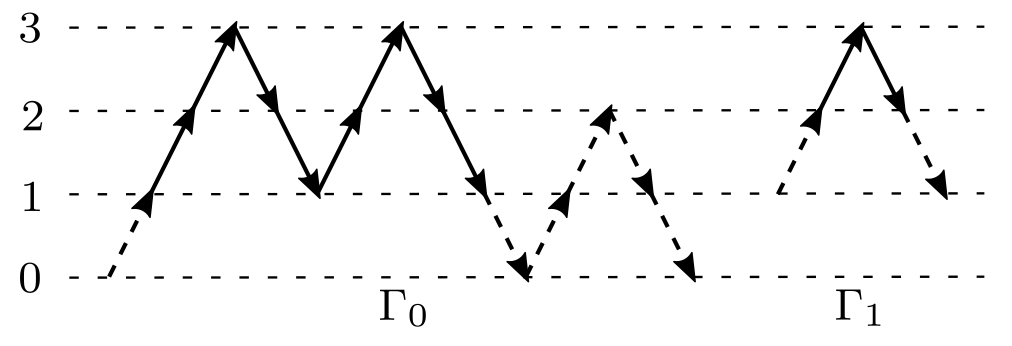}
        \caption{$\Gamma(7,3)$ (right) has two summits both in one 2-block of length 2.
        The solid arrows indicate $\Gamma_1$ and $\Gamma_2$ (in $\Gamma_1$).}
        \label{bc:g73}
    \end{subfigure}
    \caption{}
    \label{basecasefig}
\end{figure}

\begin{lem}\label{basecase}
Let $(p,q)$ be a co-prime pair with $p$ and $q$ positive and $q$ odd.
If $q=1$ or $(p\modop q)=1$ then $(p,q)$ is a pre-RTFN pair.
\end{lem}

\begin{proof}
Define $\kappa$ as in Proposition \ref{segsizeprop}.

When $q=1$, $\oGamma(p,q)$ is the closure of a positive $p$-segment followed by a negative $p$-segment
so $\oGamma(p,q)$ only has one summit;
see Figure \ref{bc:g31}.
It can be clearly seen that $(p,q)$ is a pre-RTFN pair by making the following choice.
\begin{itemize}
    \item Let $N=1$.
    \item Let $\Gamma_0=\Gamma(p,q)$.
    \item Let $\Gamma_1=\Gammatop$.
    \item Let $n_1=1$.
    \item Let $\Upsilon_1=\oGamma(p,q)-\Gammatop$.
\end{itemize}

When $p\modop q=1$, $\oGamma(p,q)$ is the closure of a positive $(\kappa+1)$-segment, a $\kappa$-block of length $q-1$, a negative $(\kappa+1)$-segment, followed by another $\kappa$-block of length $q-1$
so $\oGamma(p,q)$ has $(q+1)/2$ summits all contained in the same $\kappa$-block;
see Figure \ref{bc:g73}.

Again, it's not hard to see that $(p,q)$ is a pre-RTFN pair.

When $\kappa=1$, make the following choices.
\begin{itemize}
    \item Let $N=1$.
    \item Let $\Gamma_0=\Gamma(p,q)$.
    \item Let $\Gamma_1=\Gammatop$.
    \item Let $n_1=(q+1)/2$.
    \item Let $\Upsilon_1$ be the subgraph of $\oGamma(p,q)$ with the all summits and their incident edges removed.
\end{itemize}

When $\kappa>1$, make the following choices.
\begin{itemize}
    \item Let $N=2$.
    \item Let $\Gamma_0=\Gamma(p,q)$.
    \item Let $\Gamma_1$ be a positive $\kappa$-segment followed by a negative $\kappa$-segment with a summit between them.
    \item Let $\Gamma_2=\Gammatop$.
    \item Let $\Upsilon_1$ be the subgraph of $\oGamma(p,q)$ with the $\kappa$-block containing all the bottoms along with the edges immediately preceding and following the block.
    \item Let $\Upsilon_2$ be $\cl(\Gamma_1)-\Gammatop$.
    \item Let $n_1=(q+1)/2$.
    \item Let $n_2=1$.
\end{itemize}
\end{proof}

Let $(p,q)$ be a co-prime pair with $q>0$,
and $(p^*,q^*)$ be the co-prime pair defined by Lemma \ref{redpair}.
Suppose $(p^*,q^*)$ is a pre-RTFN pair
so there is a positive integer $N^*$ subgraphs
\[
    \Gamma^*_0,\ldots,\Gamma^*_N
\]
and
\[
    \Upsilon^*_1,\ldots,\Upsilon^*_N
\]
and integers
\[
    n^*_1,\ldots,n^*_N
\]
satisfying \ref{ngstart},\ref{ngend},\ref{ngrec},\ref{ngvsm}, and \ref{ngsym}.

Define $\kappa$ and $\kappa'$ as in (\ref{eucdivpq}) and (\ref{eucdivqr}) so
$\oGamma(p,q)\cong E(\oGamma(p^*,q^*),\kappa,\kappa')$ by Proposition \ref{reduceexpand}.
For simplicity of notation, define
\[
    E(\Gamma^*):=E(\Gamma^*,\kappa,\kappa')
\]
for any closable subgraph $\Gamma^*$ of $\oGamma(p^*,q^*)$.

To show that $(p,q)$ is a pre-RTFN pair, we need to define $N$, the subgraphs $\{\Gamma_i\}^N_0$ and $\{\Upsilon_i\}^N_1$, and the integers $\{n_i\}^N_1$ for $(p,q)$.
This choice depends on how expansion effects the nested repeating pattern of summits in $\oGamma(p^*,q^*)$.

In general, we want to define $\Gamma_i$ to be $E(\Gamma^*_i)$.
By \ref{ngrec}, $(\Gamma^*_i)^{n^*_i}$ is a subgraph of $\Gamma^*_{i-1}$ for all $i=1,\ldots,N^*$.
It follows that for all $i=1,\ldots,N^*$, $E((\Gamma^*_i)^{n^*_i})$ is a subgraph of $E(\Gamma^*_{i-1})$.
We want $\Gamma_{i}^{n_i}$ to be a subgraph of $\Gamma_{i-1}$ which is equal to $E(\Gamma^*_{i-1})$.
However, if $\Gamma_i$ is $E(\Gamma^*_i)$, then $\Gamma_i^{n_i}$ is $(E(\Gamma^*_i))^{n^*_i}$, and
$E((\Gamma^*_i)^{n^*_i})$ may not be equal to $(E(\Gamma^*_i))^{n^*_i}$.
Nevertheless, $(E(\Gamma^*_i))^{n^*_i}$ is a subgraph of $E(\Gamma^*_{i-0})$ by adding or removing $\kappa$ edges.

While the leading summits of $\oGamma(p,q)$ corresponds to the summits of $\oGamma(p^*,q^*)$, we must also consider the non-leading summits in $\oGamma(p,q)$.
Let $d$ be $\kappa'$ or $\kappa'-1$ whichever is even.
Let $\Gammatop^*$ be the subgraph of a summit in $\oGamma(p^*,q^*)$ with its two adjacent vertices.
$E(\Gammatop^*)$ is always the concatenation of a $\kappa$-block of even length, positive $(\kappa+1)$-segment, another $\kappa$-block of even length, and a negative $(\kappa+1)$-segment.
It follows that every summit in $\oGamma(p^*,q^*)$ corresponds to $d/2+1$ summits in $\Gamma(p,q)$.

We define $N$, $\{\Gamma_i\}^N_0$, and $\{n_i\}^N_1$ as follows.

Suppose $\kappa'=1$ or $\kappa=1$.
\begin{itemize}
    \item Let $N=N^*+1$.
    \item For each $i=0,\ldots,N^*$, let 
    $\Gamma_i=E(\Gamma^*_i)$.
    \item For each $i=1,\ldots,N^*$, let $n_i=n^*_i$.
    \item Let $\Gamma_{N}=\Gammatop$.
    \item Let $n_{N}=d/2+1$.
\end{itemize}

Suppose $\kappa'>1$ and $\kappa>1$.
\begin{itemize}
    \item Let $N=N^*+2$.
    \item For each $i=0,\ldots,N^*$, let 
    $\Gamma_i=E(\Gamma^*_i)$.
    \item For each $i=1,\ldots,N^*$, let $n_i=n^*_i$.
    \item Let $\Gamma_{N-1}$ be a positive $\kappa$-segment followed by a negative $\kappa$-segment.
    \item Let $n_{N-1}=d/2+1$.
    \item Let $\Gamma_{N}=\Gammatop$.
    \item Let $n_{N}=1$.
\end{itemize}

In either case, define $\Upsilon_i=\cl(\Gamma_{i-1})-(\Gamma_i^{n_i})$
for $i=1,\ldots,N$; see figures \ref{expandeffect}.

\begin{figure}[t]
	\begin{subfigure}[t]{10cm}
        \centering
        \includegraphics{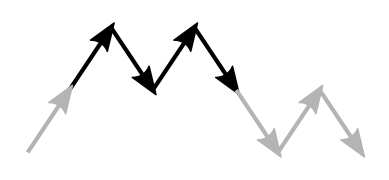}
        \caption{The graph $R(\oGamma)(26,11)=\oGamma(4,3)$ with $(\Gamma^*_1)^2$ in black and $\Upsilon^*_1$ in gray.}
        \label{ee:before}
	\end{subfigure}

	\begin{subfigure}[t]{10cm}
        \centering
        \includegraphics{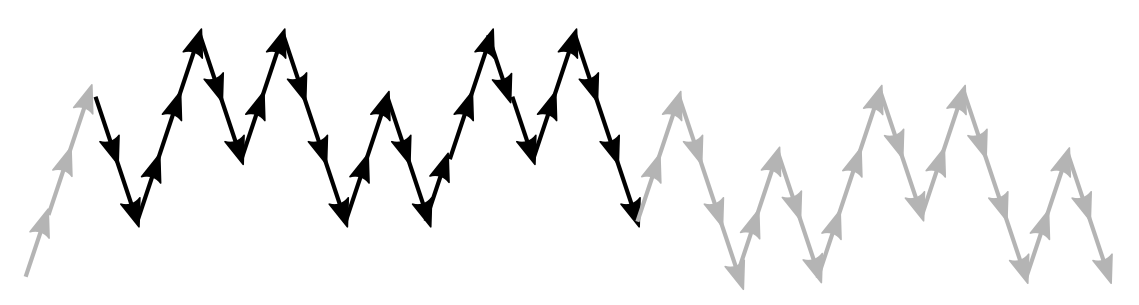}
        \caption{The graph $\oGamma(26,11)=E(\oGamma(4,3))$ with $E((\Gamma^*_1)^2)$ in black and $E(\Upsilon^*_1)$ in gray.}
        \label{ee:after1}
	\end{subfigure}

	\begin{subfigure}[t]{10cm}
        \centering
        \includegraphics{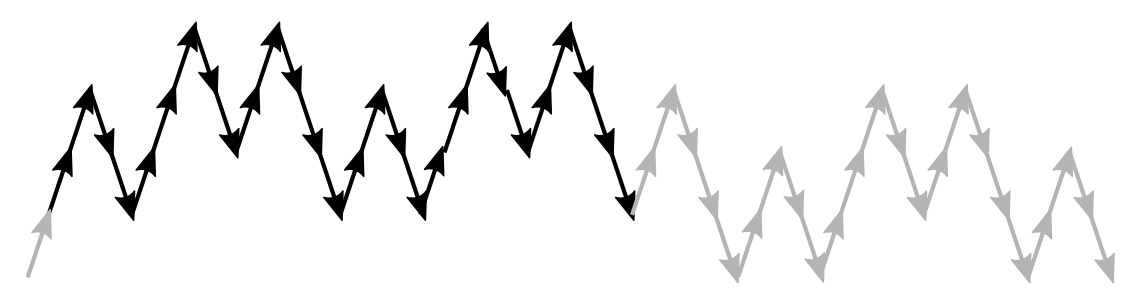}
        \caption{The graph $\Gamma_0=\oGamma(26,11)$ with $\Gamma_1^2=(E(\Gamma^*_1))^2$ in black and $\Upsilon_1$ in gray.}
        \label{ee:after2}
	\end{subfigure}
	\caption{Expanding $\oGamma(4,3)$ to $\oGamma(26,11)$}
    \label{expandeffect}
\end{figure}

\begin{lem} \label{inductionstep1}
    The integers $\{n_i\}^N_1$ and the subgraphs $\{\Gamma_i\}^N_0$ and $\{\Upsilon_i\}^N_1$ satisfy
    \ref{ngstart}, \ref{ngend}, \ref{ngrec} and \ref{ngvsm}.
\end{lem}

\begin{proof}
    Since $\Gamma^*_0\cong\Gamma(p^*,q^*)$,
    \[
        \Gamma_0\cong E(\Gamma(p^*,q^*))\cong\Gamma(p,q)
    \]
    so \ref{ngstart} is satisfied.

    By definition, $\Gamma_N\cong\Gammatop$
    so \ref{ngend} is satisfied.

    For each $i=1,\ldots,N$,
    $\Upsilon_{i}=\cl(\Gamma_{i-1})-(\Gamma_i^{n_i})$
    so
    \[
        \cl(\Gamma_{i-1})\cong \cl(\Gamma_i^{n_i}*\Upsilon_{i}) .
    \]
    Therefore, \ref{ngrec} is satisfied.

    When $i>N^*$, all of the summits in $\Gamma_{i-1}$ are contained in $\Gamma_i^{n_i}$ by construction
    so $\Gamma_{i}=\Gamma_{i-1}-\Gamma_i^{n_i}$ has no summits.
    
    For each $i=1,\ldots,N^*$,
    \[
        \Upsilon_{i}=\cl(\Gamma_{i-1})-(\Gamma_i^{n_i})
        =\cl(\Gamma_{i-1})-(E(\Gamma^*_i)^{n^*_i})
    \]
    and
    \[
        E(\Upsilon^*_{i})=\cl(E(\Gamma^*_{i-1}))-E((\Gamma^*_i)^{n^*_i})
        =\cl(\Gamma_{i-1})-E((\Gamma^*_i)^{n^*_i}).
    \]
    $E((\Gamma^*_i)^{n^*_i})$ is $(E(\Gamma^*_i))^{n^*_i}$ possibly with $\kappa$ edges added of removed.
    It follows that $\Upsilon_i$ is $E(\Upsilon^*_i)$ with possibly $\kappa$ edges added or removed; see Figure \ref{expandeffect}.
    Since no summits are in $\Upsilon^*_i$,
    there are no summits $E(\Upsilon^*_i)$.
    The edges added or removed from $E(\Upsilon^*_i)$ to get $\Upsilon_i$ are not summits.
    Thus, there are no summits in $\Upsilon_i$
    Therefore, \ref{ngvsm} is satisfied.
\end{proof}

\begin{lem} \label{inductionstep2}
    The subgraphs $\{\Gamma_i\}^N_0$ satisfy
    \ref{ngsym}.
\end{lem}

\begin{proof}
    First, we show what $\Gamma_i$ has no bottoms for each $i=1,\ldots, N$.
    Since $N^*\geq 1$, $\Gamma_1=E(\Gamma^*_1)$.
    Since $\Gamma^*_1$ has no bottoms, $\Gamma_1$ does not have bottoms.
    When $1\leq i\leq N$,
    \[
        \Gamma_{i-1}\cong\cl(\Gamma_i^{n_i}*\Upsilon_i)
    \]
    so $\Gamma_i$ is a subgraph of $\Gamma_1$
    Therefore, $\Gamma_i$ has no bottoms.
    
    Suppose $0 \leq i \leq N$.
    Here we show that $\Gamma_i$ is symmetric.
    When $i>N^*$, $\Gamma_i$ is either the concatenation of a positive $\kappa$-segments and a negative $\kappa$-segment or $\Gammatop$.
    In both case, $\Gamma_i$ is clearly symmetric.

    Suppose $0\leq i\leq N^*$.
    In this case, $\Gamma_i=E(\Gamma^*_i)$.
    Our goal is to show that since $\Gamma^*_i$ is symmetric, $\Gamma_i$ is also symmetric.

    Since $\Gamma^*_i$ is symmetric,
    there is an order reversing bijection $\phi^*$ on the set of vertices of $\cl(\Gamma^*_i)$ and an integer $k^*$
    such that for each $P^*$ in $\cl(\Gamma^*_i)$,
    \[
	    \gr(P^*)+\gr(\phi^*(P^*))=k^*.
    \]
    Let $V_L$ and $V_T$ be the sets of leading and trailing vertices of $\cl(\Gamma_i)$ respectively,
    and let $V^*$ be the vertex set of $\cl(\Gamma^*_i)$.
    Define $\phi$ to be the unique order reversing bijection on the vertices of $\cl(\Gamma_i)$ such that the following diagram commutes,
    \[
        \begin{tikzcd}
            V_L \arrow{r}{\phi|_{V_L}} \arrow[swap]{d}{f_L}
            & V_T \arrow{d}{f_T} \\
            V^* \arrow{r}{\phi^*} & V^*
        \end{tikzcd}
    \]

    In particular, $\phi$ maps leading vertices bijectively to trailing vertices (see Figure \ref{figphistar}).
    Let $P_S$ be a leading summit of $\Gamma_i$, and
    let $P^*_S=f_L(P_S)$ in $\Gamma^*_i$.

   \begin{figure}[t]
        \includegraphics{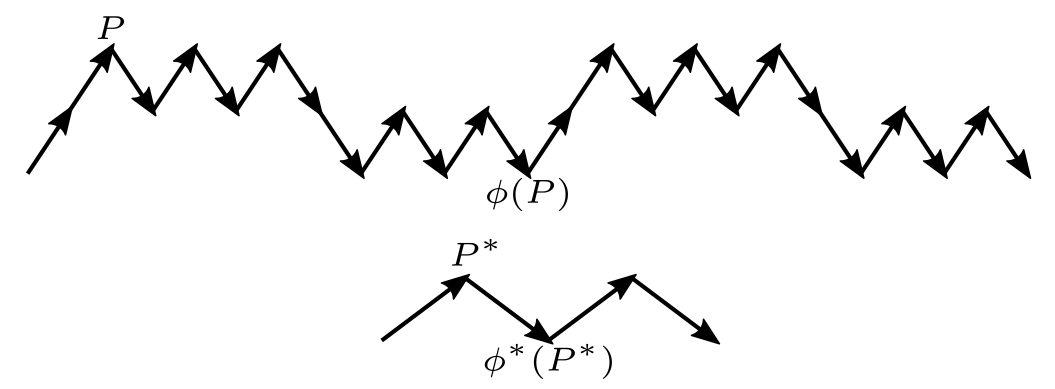}
        \caption{The incremental cycles $\cl(\Gamma_i)$ (top) and $\cl(\Gamma^*_i)$ (bottom) are shown. $P$ is a leading vertex, and $f_L(P)$ is denoted $P^*$. $\phi(P)$ is a trailing vertex, and $\phi^*(P^*)=f_T(\phi(P))$.}
        \label{figphistar}
    \end{figure}

Let $k=\gr(P_S)+\gr(\phi(P_S))$, and
let $P$ be an arbitrary vertex in $\Gamma_i$.
The goal is to show that $\gr(P)+\gr(\phi(P))=k$ which is done in four cases.

\textit{Case 1.} Suppose $P$ is a leading vertex and $P^*:=f_L(P)$ has the same type as $P$, either a peak (type $(-+)$) or valley (type $(+-)$).
If $P^*$ is of type $(-+)$, then $\phi^*(P^*)$ is of type $(+-)$, and
if $P^*$ is of type $(+-)$, then $\phi^*(P^*)$ is of type $(-+)$.
Therefore, either $f_L^{-1}(P^*)$ and $f_T^{-1}(P^*)$ are both peaks and $f_L^{-1}(\phi^*(P^*))$ and $f_T^{-1}(\phi^*(P^*))$ are both valleys or $f_L^{-1}(P^*)$ and $f_T^{-1}(P^*)$ are both valleys and $f_L^{-1}(\phi^*(P^*))$ and $f_T^{-1}(\phi^*(P^*))$ are both peaks.
In either case, 
\begin{equation} \label{equalgradings}
	\gr(f_L^{-1}(\phi^*(P^*)))=\gr(f_T^{-1}(\phi^*(P^*))).
\end{equation}
Thus,
\begin{align*}
    \gr(P)+\gr(\phi(P))-k=&\gr(P)-\gr(P_S)+\gr(\phi(P))-\gr(\phi(P_S))\\
    =&\gr(f_L^{-1}(P^*))-\gr(f_L^{-1}(P^*_S))\\
    &+\gr(\phi(f_L^{-1}(P^*)))-\gr(\phi(f_L^{-1}(P^*_S)))\\
    =&\gr(f_L^{-1}(P^*))-\gr(f_L^{-1}(P^*_S))\\
    &+\gr(f_T^{-1}(\phi^*(P^*)))-\gr(f_T^{-1}(\phi^*(P^*_S)))
\end{align*}
Summits are of type $(-+)$ so by (\ref{equalgradings}),
\[
    \gr(f_T^{-1}(\phi^*(P^*)))-\gr(f_T^{-1}(\phi^*(P^*_S)))=\gr(f_L^{-1}(\phi^*(P^*)))-\gr(f_L^{-1}(\phi^*(P^*_S)))
\]
By Proposition \ref{phigradings},
\begin{align*}
    \gr(P)+\gr(\phi(P))-k=&\gr(f_L^{-1}(P^*))-\gr(f_L^{-1}(P^*_S))\\
    &+\gr(f_L^{-1}(\phi^*(P^*)))-\gr(f_L^{-1}(\phi^*(P^*_S)))\\
    =&\gr(P^*)-\gr(P^*_S)+\gr(\phi^*(P^*))-\gr(\phi^*(P^*_S))\\
    =&\gr(P^*)+\gr(\phi^*(P^*))-(\gr(P^*_S)+\gr(\phi^*(P^*_S)))\\
    =&k^*-k^*=0.
\end{align*}
Therefore,
\[
    \gr(P)+\gr(\phi(P))=k.
\]

\textit{Case 2.} Suppose $P$ is a leading peak and $P^*:=f_L(P)$ has type $(++)$.
In this case, $f_L^{-1}(P^*)$ and $f_L^{-1}(\phi^*(P^*))$ are both peaks and $f_T^{-1}(P^*)$ and $f_T^{-1}(\phi^*(P^*))$ are both valleys.
Thus,
\[
\gr(f_L^{-1}(\phi^*(P^*)))=\gr(f_T^{-1}(\phi^*(P^*)))+\kappa,
\]
and
\begin{align*}
    \gr(P)+\gr(\phi(P))-k=&\gr(P)-\gr(P_S)+\gr(\phi(P))-\gr(\phi(P_S))\\
    =&\gr(f_L^{-1}(P^*))-\gr(f_L^{-1}(P^*_S))\\
    &+\gr(\phi(f_L^{-1}(P^*)))-\gr(\phi(f_L^{-1}(P^*_S)))\\
    =&\gr(f_L^{-1}(P^*))-\gr(f_L^{-1}(P^*_S))\\
    &+\gr(f_T^{-1}(\phi^*(P^*)))-\gr(f_T^{-1}(\phi^*(P^*_S)))\\
    =&\gr(f_L^{-1}(P^*))-\gr(f_L^{-1}(P^*_S))\\
    &+\gr(f_L^{-1}(\phi^*(P^*)))-\gr(f_L^{-1}(\phi^*(P^*_S)))-\kappa\\
    =&\gr(P^*)-\gr(P^*_S)+\gr(\phi^*(P^*))-\gr(\phi^*(P^*_S))-\kappa+\kappa\\
    =&0.
\end{align*}

\textit{Case 3.} Suppose $P$ is a leading valley and $P^*:=f_L(P)$ has type $(--)$.
In this case, $f_T^{-1}(P^*)$ and $f_T^{-1}(\phi^*(P^*))$ are both peaks and $f_L^{-1}(P^*)$ and $f_L^{-1}(\phi^*(P^*))$ are both valleys.
Thus,
\[
    \gr(f_L^{-1}(\phi^*(P^*)))=\gr(f_T^{-1}(\phi^*(P^*)))-\kappa,
\]
and
\begin{align*}
    \gr(P)+\gr(\phi(P))-k=&\gr(P)-\gr(P_S)+\gr(\phi(P))-\gr(\phi(P_S))\\
    =&\gr(f_L^{-1}(P^*))-\gr(f_L^{-1}(P^*_S))\\
    &+\gr(\phi(f_L^{-1}(P^*)))-\gr(\phi(f_L^{-1}(P^*_S)))\\
    =&\gr(f_L^{-1}(P^*))-\gr(f_L^{-1}(P^*_S))\\
    &+\gr(f_T^{-1}(\phi^*(P^*)))-\gr(f_T^{-1}(\phi^*(P^*_S)))\\
    =&\gr(f_L^{-1}(P^*))-\gr(f_L^{-1}(P^*_S))\\
    &+\gr(f_L^{-1}(\phi^*(P^*)))-\gr(f_L^{-1}(\phi^*(P^*_S)))+\kappa\\
    =&\gr(P^*)-\gr(P^*_S)+\gr(\phi^*(P^*))-\gr(\phi^*(P^*_S))+\kappa-\kappa\\
    =&0.
\end{align*}

\textit{Case 4.} Suppose $P$ is not a leading vertex.
Let $P'$ be the leading vertex in $\cl(\Gamma_i)$ such that the length of the path $\omega(P',P)$, the path in $\cl(\Gamma_i)$ from $P'$ to $P$, is minimal.
It follows that $\omega(P',P)$ is isomorphic to a subgraph of a $\kappa$-block as in Figure \ref{subblockfig}.
In particular, there are no leading vertices between $P'$ and $P$ in $\cl(\Gamma_i)$;
therefore, there are no trailing vertices between $\phi(P)$ and $\phi(P')$ in $\cl(\Gamma_i)$
so $\omega(\phi(P),\phi(P'))$, the path from $\phi(P)$ to $\phi(P')$ in $\cl(\Gamma_i)$, is also isomorphic to a subgraph of a $\kappa$-block.

\begin{figure}[t]
\includegraphics{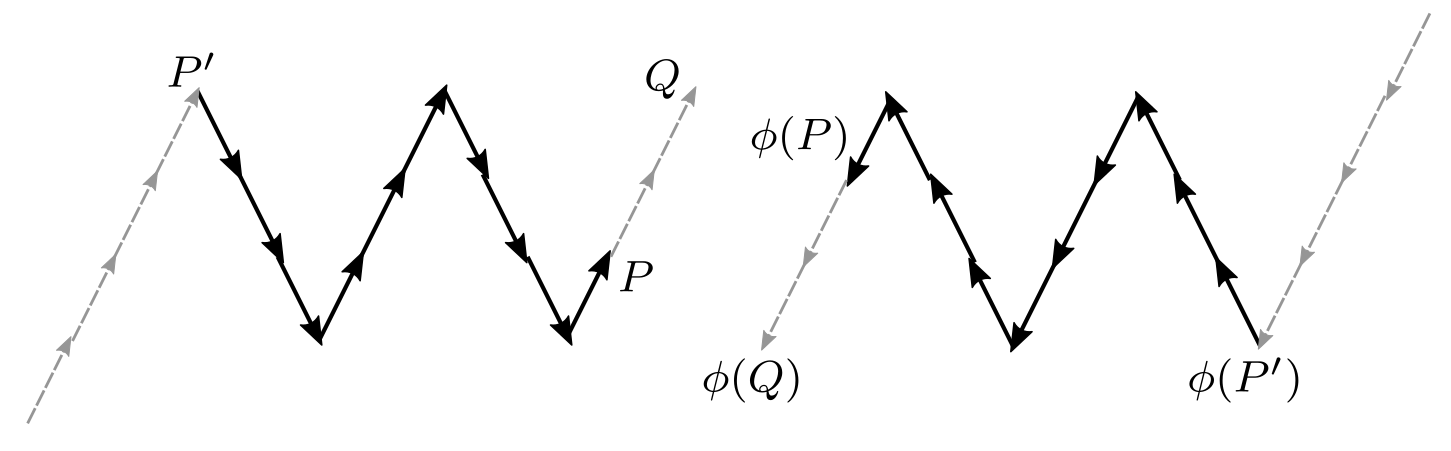}
\caption{$\omega(P',P)$ (left) and $\omega(\phi(P),\phi(P'))$ (right) are shown in solid black. The dashed gray arrows are other edges in $\cl(\Gamma_i$. The case shown is when $P'$ is a peak.}

\label{subblockfig}
\end{figure}

Let $Q$ be the closest vertex to $P$ with grading $\gr(Q)=\gr(P')$.
When $P'$ is a peak, $Q$ is a peak.
Likewise, when $P'$ is a valley, $Q$ is a valley.
Define $\delta$ be the distance from $P'$ to $Q$.
Since $P$ is in a $\kappa$-block which starts at $P`$, $Q$ and $P$ lie on the same segment so
\[
    \gr(Q)-\gr(P)=\left\{
    \begin{array}{ll}
    \delta&\text{when }Q\text{ is a peak}\\
    -\delta&\text{when }Q\text{ is a valley}
    \end{array}
    \right.
    .
\]
also, $\phi(Q)$ and $\phi(P)$ lie on the same segment so
\[
    \gr(\phi(Q))-\gr(\phi(P))=\left\{
    \begin{array}{ll}
    -\delta&\text{when }Q\text{ is a peak}\\
    \delta&\text{when }Q\text{ is a valley}
    \end{array}
    \right.
    .
\]

If $P'$ and $Q$ are peaks,
then
\[
	\gr(P)=\gr(Q)-\delta=\gr(P')-\delta
\]
and
\[
	\gr(\phi(P))=\gr(\phi(Q))+\delta=\gr(\phi(P'))+\delta.
\]
If $P'$ and $Q$ are valleys,
then
\[
	\gr(P)=\gr(Q)+\delta=\gr(P')+\delta
\]
and
\[
	\gr(\phi(P))=\gr(\phi(Q))-\delta=\gr(\phi(P'))-\delta.
\]
In both cases,
\[
	\gr(P)+\gr(\phi(P))=\gr(P')+\gr(\phi(P'))=k.
\]
Therefore, for every vertex $P$ in $\cl(\Gamma_i)$,
$\gr(P)+\gr(\phi(P))=k$
so $\Gamma_i$ is symmetric.
\end{proof}

\begin{proof}[Proof of Lemma \ref{nestedword}]
By Lemma \ref{nestedgraphs}, it is sufficient to show that every co-prime pair is a pre-RTFN pair.

Let $(p,q)$ be a co-prime pair with $p$ positive and $q$ odd.
If $q=1$ or $(p\modop q)=1$ with $q$ positive, then $(p,q)$ is a pre-RTFN pair by Lemma \ref{basecase}.
If $q=-1$ then $(p,q)$ is a pre-RTFN pair by Lemma \ref{negsignscor}.

Suppose $|q|\neq1$ and $(p\modop q)>1$, and
assume every co-prime pair $(p',q')$ with $|q'|<|q|$ is a pre-RTFN pair.
When $q$ is positive, define the co-prime pair $(p^*,q^*)$ as in Lemma \ref{redpair}.
Since $|q^*|<|q|$, $(p^*,q^*)$ is a pre-RTFN pair.
By Lemma \ref{inductionstep1} and Lemma \ref{inductionstep2}, $(p,q)$ is also pre-RTFN pair.
When $q$ is negative, the pair $(p,-q)$ is a pre-RTFN pair by the above argument.
Thus $(p,q)$ is a pre-RTFN pair by Lemma \ref{negsignscor}.

By strong induction, every co-prime pair $(p,q)$ with $p$ positive and $q$ odd is a pre-RTFN pair.
\end{proof}


\printbibliography

\end{document}